\documentclass[12pt]{amsart}

\allowdisplaybreaks
\usepackage{mathpazo}
\usepackage{amsmath,amsfonts,amssymb}
\usepackage{hyperref} 
\usepackage{amsrefs}
\usepackage{amsthm}
\usepackage{latexsym,amsmath,amssymb,amsfonts}
\usepackage{rotating}
\usepackage{mathrsfs}
\usepackage{xypic} \xyoption{all}
\usepackage{amscd}
\usepackage{euscript}
\usepackage{hhline}
\usepackage{graphicx,epstopdf}
\usepackage{epsfig}
\usepackage{xcolor}
\usepackage[all,color]{xy}

\newlength{\fighskip} \fighskip=2pt
\newlength{\figvskip} \figvskip=3pt

\usepackage{hyperref}
\newcommand*{\figbox}[2]{{
  \def\figscale{#1}
  \def\arraystretch{0.8}
  \arraycolsep=0pt
  \begin{array}{c}
    \vbox{\vskip\figscale\figvskip
      \hbox{\hskip\figscale\fighskip
        \includegraphics[scale=\figscale]{#2}}}
  \end{array}}}

\numberwithin{equation}{section}

\newcommand{\C}{\mathbb{C}}

\newcommand{\Z}{\mathbb{Z}}

\newcommand{\E}{{\mathcal E}}

\renewcommand{\H}{\mathbb{H}}

\newcommand{\g}{\mathfrak{g}}

\newcommand{\abracket}[1]{\left\langle#1\right\rangle}

\newcommand{\fbracket}[1]{\left\{#1\right\}}
\newcommand{\bracket}[1]{\left(#1\right)}

\newcommand{\pa}{\partial}

\newcommand{\OO}{{\mathcal O}}

\newcommand{\Ol}{\mathcal O_{loc}}

\DeclareMathOperator{\Sym}{Sym}

\DeclareMathOperator{\Tr}{Tr}

\newcommand{\A}{\mathcal A}

\theoremstyle{plain}
\newtheorem{thm}{Theorem}[section]
\newtheorem{thm-defn}{Theorem/Definition}[section]
\newtheorem{lem}[thm]{Lemma}
\newtheorem{lem-defn}[thm]{Lemma/Definition}
\newtheorem{prop}[thm]{Proposition}
\newtheorem{cor}[thm]{Corollary}

\newtheorem{example}[thm]{Example}
\newtheorem{defn}[thm]{Definition}

\theoremstyle{definition}
\newtheorem{notn}[thm]{Notation}
\newtheorem{eg}[thm]{Example}

\theoremstyle{remark}
\newtheorem{rmk}[thm]{Remark}

\allowdisplaybreaks[4]  

\begin{document}

\title[BV quantization of the RW model]{BV quantization of the Rozansky-Witten model}

\author[Chan]{Kwokwai Chan}
\address{Department of Mathematics\\ The Chinese University of Hong Kong\\ Shatin\\ Hong Kong}
\email{kwchan@math.cuhk.edu.hk}

\author[Leung]{Naichung Conan Leung}
\address{The Institute of Mathematical Sciences and Department of Mathematics\\ The Chinese University of Hong Kong\\ Shatin \\ Hong Kong}
\email{leung@math.cuhk.edu.hk}

\author[Li]{Qin Li}
\address{Department of Mathematics \\ Southern University of Science and Technology\\ Shenzhen\\China}
\email{liqin@sustc.edu.cn}

\begin{abstract}
We investigate the perturbative aspects of Rozansky-Witten's 3d $\sigma$-model \cite{RW} using Costello's approach to the Batalin-Vilkovisky (BV) formalism \cite{Kevin-book}. We show that the BV quantization (in Costello's sense) of the model, which produces a perturbative quantum field theory, can be obtained via the configuration space method of regularization due to Kontsevich \cite{Kontsevich} and Axelrod-Singer \cite{AS2}.
We also study the factorization algebra structure of quantum observables following Costello-Gwilliam \cite{Kevin-Owen}. In particular, we show that the cohomology of local quantum observables on a genus $g$ handle body is given by $H^*(X,(\wedge^*T_X)^{\otimes g})$ (where $X$ is the target manifold), and prove that the partition function reproduces the Rozansky-Witten invariants.
\end{abstract}

\maketitle


\section{Introduction}

In \cite{RW}, Rozansky and Witten discovered an interesting 3-dimensional $\sigma$-model whose target space is given by a hyperk\"ahler manifold $X$ (or more generally, a holomorphic symplectic manifold, as demonstrated shortly after by Kontsevich \cite{Kontsevich2} and Kapranov \cite{Kapranov}; see also the appendix of \cite{RW}) -- this is the famous {\em Rozansky-Witten model}. In particular, the perturbative expansion of the partition function, as a rigorously defined combinatorial sum over Feynman diagrams $\Gamma$:
\begin{equation}\label{eqn:RW_partition_function}
Z_X(M)=\sum_{\Gamma}b_\Gamma(X)I_{\Gamma}(M),
\end{equation}
gives rise to the {\em Rozansky-Witten weight system}, which provides a new construction of certain finite-type 3-manifold topological invariants (that turned out to be the Vassiliev invariants). Their work immediately opened up a new research direction and has generated a lot of work by both mathematicians and physicists \cite{Kontsevich2, Kapranov, Hitchin-Sawon, Sawon, Sawon2, Roberts-Sawon, Roberts-Willerton, Thompson, Habegger-Thompson, Kapustin-Saulina, Kapustin, Kapustin-Rozansky, Kapustin-Rozansky2, Qiu-Zabzine, Qiu-Zabzine2, K-Qiu-Zabzine, Teleman, Xu}.

Nevertheless, many questions concerning the Rozansky-Witten model remain unanswered. Among them is a conjecture by Rozansky and Witten themselves \cite{RW}*{Section 5} asserting that the Hilbert space $\mathcal{H}_g$ associated to a genus $g$ Riemann surface $\Sigma_g$ should be given by the cohomology group $H^*\left( X, \left(\wedge^*T_X\right)^{\otimes g} \right)$, where $X$ is the target compact hyperk\"ahler manifold and $T_X$ its holomorphic tangent bundle. Other important problems such as the construction of an (extended) topological quantum field theory (TQFT) from the Rozansky-Witten model \cite{Sawon2, Roberts-Sawon} and the closely related theory of boundary conditions \cite{Kapustin-Rozansky, Kapustin-Rozansky2, Teleman} are also currently under intensive investigation.

The main goal of the present paper is to apply the pioneering work of Costello \cite{Kevin-book} to construct a quantization of the Rozansky-Witten model in the {\em Batalin-Vilkovisky (BV) formalism} and investigate its observable theory making use of the recent foundational work of Costello-Gwilliam \cite{Kevin-Owen}. Besides providing yet another example where Costello's machinery can be applied, we expect that these new techniques and structures (in particular the structure of factorization algebra of observables) can shed new light on the study of the Rozansky-Witten model.

To apply Costello's (homological) method of renormalization, we shall focus only on the perturbative aspects of the Rozansky-Witten model. So instead of considering the full mapping space, we are going to describe the classical theory only in a formal neighborhood of the space of constant maps, and this will be done using the geometry of holomorphic Weyl bundles in Section \ref{sec:classical_theory}.

We then study the BV quantization of our model along the lines of \cite{Kevin-book}; such a quantization amounts to constructing an effective action (as a modification of the classical action functional) compatible with the {\em renormalization group (RG) flow} and satisfying the {\em quantum master equation (QME)} (the latter is a homological condition which essentially guarantees that the path integrals are well-defined).
The effective action will be defined as a sum over Feynman diagrams where the propagator is given by the Green kernel of $\frac{d^*}{\Delta}$ with respect to some a priori chosen Riemannian metric $g$ on the source 3-manifold $M$. In general, singularities may occur in the Feynman weights (owing to the infinite dimensional nature of the mapping space) and one needs to add counter terms to the action to remove the singularities.

For our model, however, we observe that the technique of configuration space developed by Kontsevich \cite{Kontsevich} and Axelrod-Singer \cite{AS2} in their study of perturbative Chern-Simons theory suffices.\footnote{This is certainly due to the similarity between the Chern-Simons theory and the Rozansky-Witten model, as pointed out by Rozansky and Witten in their original paper \cite{RW}.} As a result, we can construct a perturbative quantization of our model without dealing directly with the singularities and counter terms. In particular, we shall show that the so-called {\em naive quantization} (with certain local quantum corrections taken into account) already satisfies the QME, and hence the quantized theory we obtain is independent of the metric we chose. The details of the quantization process are contained in Section \ref{sec:quantization} (see in particular Theorems~\ref{theorem: QME} and \ref{theorem: QME2}.

Once our theory is quantized, we proceed to investigate the structure of {\em factorization algebra} of the classical and quantum observables, following the fundamental approach of Costello and Gwilliam developed in their recent two-volume book \cite{Kevin-Owen}. Both classical and quantum observables are defined as cochain complexes which are, roughly speaking, functionals on fields over open subsets in a 3-manifold $M$.

In this paper, we will focus on local quantum observables supported on open sets which are homeomorphic to (open) handle bodies $H_g\subset M$.
We show that the local quantum observables on $H_g$ do not receive any quantum corrections, and hence we have the following
\begin{thm}\label{thm:main1}
The cohomology of local quantum observables on a handle body $H_g$ is the same as that of local classical observables (except for a formal variable $\hbar$), i.e. we have an isomorphism of graded vector spaces:
$$
H^*\left(\text{Obs}^q(H_g),\hat{Q}\right)\cong H^*\left(\text{Obs}^{cl}(H_g),Q+\{I_{cl},-\}\right)[[\hbar]].
$$
\end{thm}
A direct computation of the cohomology of local classical observables via standard techniques in algebraic topology (see Section \ref{subsection:classical-observables}) then yields the following
\begin{cor}\label{cor:main}
The cohomology of local quantum observables on the handle body $H_g$ is given by
$$
H^*\left(\text{Obs}^q(H_g),\hat{Q}\right) \cong H^*(X,(\wedge^*T_X)^{\otimes g})[[\hbar]].
$$
\end{cor}
The details and precise statements can be found in Sections \ref{subsection:classical-observables}-\ref{subsection:local-quantum-observables} (see, in particular, Theorems \ref{thm:local-classical-observable} and \ref{thm:classical-to-quantum-observables}).

\begin{rmk}
As suggested in \cite{Kevin-Owen}, the cohomology of local quantum observables gives an equivalent definition of the physical Hilbert space $\mathcal{H}_g$ associated to a genus $g$ Riemann surface $\Sigma_g$, so in a sense the above corollary verifies the conjecture of Rozansky-Witten we alluded to above.

On the other hand, as pointed out by an anonymous referee, the work of Ayala-Francis \cite{Ayala-Francis15} suggests that for the 3-manifold $\Sigma_g \times \mathbb{R}$, the factorization algebra structure on the cohomology of quantum observables should produce an associative algebra. Corollary~\ref{cor:main} and the excision axiom would then provide an algebraic way to compute the quantum observables on a closed 3-manifold.
In view of this, our results might also be useful in constructing the (fully extended) TQFT underlying the Rozansky-Witten model. We plan to investigate this in a future work.
\end{rmk}

As explained in \cite{Kevin-CS}, a BV quantization produces a {\em projective volume form} on the space of global quantum observables, with which one can define correlation functions of quantum observables.
In particular, this defines the {\em partition function}, namely, the correlation function $\langle 1\rangle_M^X$ of the constant functional $1$, of our model. We prove that our partition function agrees with the original one \eqref{eqn:RW_partition_function} computed by Rozansky and Witten \cite{RW} (so that both give rise to the Rozansky-Witten invariants):
\begin{thm}\label{thm:main2}
The partition function of our model with domain $M$ and target $X$ coincides with the Rozansky-Witten partition function \eqref{eqn:RW_partition_function}:
$$\langle 1\rangle_M^X = Z_X(M).$$
\end{thm}
See Section \ref{subsection:correlation-function} and Theorem \ref{thm:partition-fcn-RW} for more details.

\section*{Acknowledgement}
We would like to thank Kevin Costello, Owen Gwilliam and Si Li for their interest in our work and also many helpful discussions. We also thank the anonymous referees for carefully reading an earlier version of our manuscript and giving a lot of useful comments and suggestions.

The work of K. Chan described in this paper was substantially supported by grants from the Research Grants Council of the Hong Kong Special Administrative Region, China (Project No. CUHK404412 $\&$ CUHK14300314). The work of N. C. Leung described in this paper was substantially supported by grants from the Research Grants Council of the Hong Kong Special Administrative Region, China (Project No. CUHK14302714 $\&$ CUHK14032215) and partially supported by a direct grant from CUHK. The work of Q. Li in this paper was supported by a grant from National Natural Science Foundation of China  for young scholars (Project No. 11501537).

\section{Classical theory}\label{sec:classical_theory}

Let $M$ be a closed 3-dimensional manifold, and let $X$ be a complex manifold equipped with a non-degenerate holomorphic $2$-form $\omega$. The original Rozansky-Witten model \cite{RW} is a supersymmetric $\sigma$-model with bosonic fields given by the space of smooth maps from $M$ to $X$ and fermionic fields given by the space of sections of certain bundles over $M$. (We will recall its field content and Lagrangian in Section \ref{subsection:comparison-physical-model} below.)

A mathematical framework of the perturbative theory of $\sigma$-models was proposed by Costello in \cite{Kevin-CS}, within which we give a definition of the classical theory of the Rozansky-Witten model, using the geometry of the holomorphic Weyl bundle on $X$. We remark that our formulation is based on the ingenious idea of Kapranov \cite{Kapranov} that a complex manifold can be encoded as an $L_\infty$-space via the Atiyah class of its holomorphic tangent bundle; see also \cite[Section B]{Qiu-Zabzine}.


\subsection{Holomorphic Weyl bundle}
In this subsection, we give a description of the geometry of the holomorphic Weyl bundle on $X$.

\begin{defn}
Let $X$ be a holomorphic symplectic manifold. The {\em holomorphic Weyl bundle} on $X$ is defined as:
\begin{equation*}
\mathcal{W} := \A_X^*\otimes_{\OO_X}\widehat{\Sym}\left(T_X^{\vee}\right)[[\hbar]],
\end{equation*}
where $T_X^\vee$ denotes the holomorphic cotangent bundle on $X$.
\end{defn}

More explicitly, a section of the holomorphic Weyl bundle $\mathcal{W}$ is locally of the following form:
$$
\sum_{k\geq 0}\alpha_{i_1\cdots i_k}\delta_{z}^{i_1}\cdots\delta_{z}^{i_k},
$$
where $\alpha_{i_1\cdots i_k}$'s are differential forms on $X$, and $\delta_z^{i}$'s are local basis of $T_X^\vee$ with respect to local holomorphic coordinates $z=(z^1,\cdots, z^n)$. We will call the sub-bundle $\A_X^{p,q}\otimes_{\mathcal{O}_X}\Sym^r\left(T_X^\vee\right) $ of $\mathcal{W}$ its $(p,q,r)$ component, and we will let $r$ be its {\em weight}. We will also assign $\hbar$ a weight of $2$.

Similar to the Weyl bundle of a real symplectic manifold (see e.g. \cite{Fed} for more details), there is a quantum Weyl product on $\mathcal{W}$ induced by the inverse to the holomorphic symplectic form. In local coordinates, this product can be written as:
$$
\alpha\circ \beta := \sum_{k\geq 1}\frac{1}{k!}\left(\frac{\hbar}{2}\right)^k\omega^{i_1j_1} \cdots \omega^{i_kj_k}
\frac{\partial^k\alpha}{\partial\delta_{z}^{i_1} \cdots \partial\delta_{z}^{i_k}}
\frac{\partial^k\beta}{\partial\delta_{z}^{j_1} \cdots \partial\delta_{z}^{j_k}}.
$$
In particular, there is an associated bracket on $\mathcal{W}$ which we denote by $[-,-]_{\mathcal{W}}$. Similar to the real Weyl bundle, we can define the following operators:
\begin{defn}
We define  the following operators on the Weyl bundle $\mathcal{W}$:
$$
\delta(a)=dz^i\wedge\frac{\partial a}{\partial \delta_z^i},\qquad \delta^*(a)=\delta_{z}^i\cdot \iota_{\partial_{z^i}}(a).
$$
\end{defn}
Moreover, we have the operator $\delta^{-1}:=\frac{1}{p+r}\delta^*$ on the $(p, q,r)$ component of the Weyl bundle. We will also let $\pi_0$ denote the projection from the whole Weyl bundle $\mathcal{W}$ onto its $(0,*,0)$ component. Similar to the Weyl bundle on a real symplectic manifold, there is the following lemma:

\begin{lem}
We have the following identities on the holomorphic Weyl bundle:
\begin{equation}\label{eqn: identity-delta-inverse}
\delta\circ \delta^{-1}+\delta^{-1}\circ\delta=id+\pi_0,
\end{equation}
\begin{equation}\label{eqn:delta-delta-inverse-square-0}
\delta^2=(\delta^{-1})^2=0.
\end{equation}
The operator $\delta$ can be expressed as the following bracket:
\begin{equation}\label{eqn:delta-bracket}
\delta=\frac{1}{\hbar}[\omega_{ij}dz^i\otimes\delta_z^j,-]_{\mathcal{W}}.
\end{equation}
\end{lem}

We would like to construct a flat connection on $\mathcal{W}$. We first pick a connection $\nabla$ on the holomorphic tangent bundle $T_X$ (and naturally induced on $T_X^\vee$) satisfying the following conditions:
\begin{enumerate}
 \item $\nabla$ is compatible with the complex structure,
 \item $\nabla$ is torsion free,
 \item $\nabla$ is compatible with the holomorphic symplectic form $\omega$.
\end{enumerate}
\begin{notn}
By abuse of notations, we will also use $\nabla$ to denote the associated exterior covariant derivative on the Weyl bundle.
\end{notn}

Clearly $\nabla$ is in general not flat on $\mathcal{W}$, and its curvature consists of both $(2,0)$- and $(1,1)$-parts, as follows:
$$
\nabla^2=R_{ijk}^ldz^i\wedge dz^j\otimes(dz^k\otimes \frac{\partial}{\partial z^l})+R_{\bar{i}jk}^ld\bar{z}^i\wedge dz^j\otimes(dz^k\otimes \frac{\partial}{\partial z^l}).
$$
Let $R_{ijkl}:=\omega_{ml}R_{ijk}^m$ and $R_{\bar{i}jkl}:=\omega_{ml}R_{\bar{i}jk}^m$, and we can define the following section of $\mathcal{W}$:
$$
R:=R_{ijkl}dz^i\wedge dz^j\otimes(\delta_z^k\delta_z^l)+R_{\bar{i}jkl}d\bar{z}^i\wedge dz^j\otimes(\delta_z^k\delta_z^l)
$$
The following lemma is clear:
\begin{lem}
The curvature of the connection $\nabla$ can be expressed as the following bracket:
$$
\nabla^2=\frac{1}{\hbar}[R,-]_{\mathcal{W}}.  
$$
Moreover, $R$ satisfies the property:
\begin{equation}\label{eqn:delta-inverse-annihilates-R}
\delta(R)=0,
\end{equation}
which follows from the Bianchi identity of the curvature tensor.
\end{lem}

\begin{rmk}
In the rest of this subsection, every identity will be at the classical level, i.e. modulo $\hbar$.
\end{rmk}

However, the connection $\nabla$ can be modified to a flat connection, using the bracket $[-,-]_{\mathcal{W}}$:
\begin{prop}
There is a connection on the holomorphic Weyl bundle of the following form:
$$
D=\nabla-\delta+\frac{1}{\hbar}[I,-]_{\mathcal{W}},
$$
which is flat modulo $\hbar$. Here $I$ is a $1$-form valued section of the Weyl bundle of weight $\geq 3$.
\end{prop}
\begin{proof}
First of all, we have the following straightforward calculation:
\begin{equation}\label{eqn:D-square}
\begin{aligned}
 D^2(\alpha)&=D(\nabla\alpha-\delta\alpha+\frac{1}{\hbar}[I,\alpha]_{\mathcal{W}}) \\
 &=\nabla\fbracket{\nabla\alpha-\delta\alpha+\frac{1}{\hbar}[I,\alpha]_{\mathcal{W}}}-\delta\fbracket{\nabla\alpha-\delta\alpha+\frac{1}{\hbar}[I,\alpha]_{\mathcal{W}}} \\
 &\qquad + \frac{1}{\hbar}\left[I, \nabla\alpha-\delta\alpha+\frac{1}{\hbar}[I,\alpha]_{\mathcal{W}}\right] \\
 &=\nabla^2\alpha+\frac{1}{\hbar}\left(\nabla[I,\alpha]_{\mathcal{W}}+[I,\nabla\alpha]_{\mathcal{W}}\right)-\delta\left(\frac{1}{\hbar}[I,\alpha]_{\mathcal{W}}\right)-\frac{1}{\hbar}[I,\delta\alpha]_{\mathcal{W}}
 + \frac{1}{\hbar^2}[I,[I,\alpha]]_{\mathcal{W}} \\
 &=\frac{1}{\hbar}\left[R-\delta I+\nabla I+\frac{1}{\hbar}I^2,\alpha\right]_{\mathcal{W}}.
\end{aligned}
 \end{equation}

It follows that a sufficient condition of the flatness of $D$ is the following equation:
\begin{equation}\label{eqn:equivalent-flatness-equation-D}
\delta I=R+\nabla I+\frac{1}{\hbar}I^2.
\end{equation}\label{eqn:D-flat}
Since the operator $\delta^{-1}$ increases the weight, we can easily find a solution of the above equation with leading term $\delta^{-1}(R)$ which is cubic, via an induction procedure on the weight. We will show that such a section $I$  satisfies equation \eqref{eqn:equivalent-flatness-equation-D}:
\begin{equation}\label{eqn:definition-I}
I=\delta^{-1}(R+\nabla I)+\frac{1}{\hbar}\delta^{-1}I^2.
\end{equation}

It is obvious from the construction that $I$ satisfies
\begin{equation}\label{eqn:delta-inverse-annihilates-I}
 \delta^{-1}I=0,
\end{equation}
and that $\pi_0(I)=0$. Let $A:=(\delta I-R-\nabla I-\frac{1}{\hbar}I^2)_{\hbar=0}$. A simple observation about $A$ is that $\pi_0(A)=0$.   And there is
\begin{equation}\label{eqn:delta-inverse-annihilates-A}
 \begin{aligned}
 \delta^{-1}A&=\delta^{-1}(\delta I-R-\nabla I-\frac{1}{\hbar}I^2)\\
 &=I-\delta^{-1}\left(R+\nabla I+\frac{1}{\hbar}I^2\right) = 0.
 \end{aligned}
\end{equation}
Next we will show that $A$ satisfies the following equation:
\begin{equation}\label{eqn:A}
\delta A=\nabla A+\frac{1}{\hbar}[I,A]_{\mathcal{W}}.
\end{equation}
Applying $\delta^{-1}$ to equation (\ref{eqn:A}), and using equation (\ref{eqn:delta-inverse-annihilates-A}), we get:
$$
A=\delta^{-1}\delta (A)=\delta^{-1}(\nabla A+\frac{1}{\hbar}[I,A]_{\mathcal{W}}).
$$
It then follows that $A=0$. Let us start the proof of  equation \eqref{eqn:A}. There is first the
$$
\delta A=\delta^2 I-\delta R-\delta\nabla I-\frac{1}{\hbar}\delta I^2=\nabla\delta I-\frac{1}{\hbar}\delta I^2.
$$
On the other hand,
\begin{equation}
 \begin{aligned}
  &\nabla A+\frac{1}{\hbar}[I,A]_{\mathcal{W}}\\
  =&\nabla(\delta I-R-\nabla I-\frac{1}{\hbar}I^2)+\frac{1}{\hbar}\left[I,\delta I-R-\nabla I-\frac{1}{\hbar}I^2\right]_{\mathcal{W}}\\
  =&\nabla\delta I-\nabla R-\nabla^2 I-\frac{1}{\hbar}\nabla I^2+\frac{1}{\hbar}\left[I,\delta I-R-\nabla I-\frac{1}{\hbar}I^2\right]_{\mathcal{W}}\\
  =&\nabla\delta I-\nabla^2 I-\frac{1}{\hbar}(\nabla I\circ I-I\circ\nabla I)+\frac{1}{\hbar}[I,\delta I]-\frac{1}{\hbar}[I,R]-\frac{1}{\hbar}[I,\nabla I]-\frac{1}{\hbar^2}[I, I^2]\\
  =&\nabla\delta I-\frac{1}{\hbar}[R, I]+\frac{1}{\hbar}[I,\delta I]-\frac{1}{\hbar}[I,R]\\
  =&\nabla\delta I+\frac{1}{\hbar}[I,\delta I].
 \end{aligned}
\end{equation}
Here we have used the Bianchi identity in the third equality. Now we need the following lemma which follows from the compatibility between $\nabla$ and the holomorphic symplectic structure $\omega$:
\begin{lem}\label{lem:delta-nabla-anicommutes}
The operators $\delta$ and $\nabla$ anti-commute with each other:
\begin{equation}\label{eqn:delta-nable-anticommutes}
\delta\circ \nabla+\nabla\circ\delta=0.
\end{equation}

\end{lem}
\end{proof}


\begin{rmk}
The above differential is only at the classical level, and in general this flat connection at the classical level can not be enhanced to a flat connection at the quantum level. The reason  that the above argument does not work is that the identity $\pi_0(\delta I-R-\nabla I-\frac{1}{\hbar}I^2)=0$ is only valid modulo $\hbar$.
\end{rmk}
There is an observation about $I$ which will be useful later: since $I$ is a $1$-form valued section of $\mathcal{W}$, we can decompose $I$ into its $(1,0)$ and $(0,1)$ components respectively. In particular, from equation \eqref{eqn:definition-I}, it is not difficult to see that the $(0,1)$ component is given by consecutively applying $\delta^{-1}\circ \nabla$ to the term $R_{\bar{i}(jkl)}d\bar{z}^i\otimes(\delta_z^j\delta_z^k\delta_z^l)$, or in other words, the  Taylor expansion of the Atiyah class.

We have the following:
\begin{prop}\label{prop:weyl-bundle-quasi-isomorphic-Dolbeault}
The cochain complex of sheaf $(\mathcal{W}, D)$ is quasi-isomorphic to the Dolbeault complex of $\mathcal{O}_X$.
\end{prop}
\begin{proof}
Let us consider the projection onto the $(0,*,0)$ component $\pi_0:\mathcal{W}\rightarrow \A_X^{0,*}(\mathcal{O}_X)$. It is not difficult to see that this is a cochain map. To see that this is a quasi-isomorphism, we only need to show that local flat sections of $\mathcal{W}$ can be identified with $\bar{\partial}$-closed differential forms. This follows from an iteration procedure. A simple observation is that suppose that a flat section $\alpha$ under the differential $D$ satisfies $\pi_0(\alpha)=0$, then $\alpha$ has to vanish since $\delta$ applies to the leading term of $\alpha$ nontrivially. On the other hand, starting from a $\bar{\partial}$-closed form $\alpha_0\in\A_X^{0,*}(\OO_X)$, we will construct a flat section $\alpha$ with constant term $\alpha_0:=\pi_0(\alpha)$:
$$
\alpha:=\alpha_0+\delta^{-1}(\nabla\alpha+\frac{1}{\hbar}[I,\alpha]_{\mathcal{W}}).
$$
To show that $D\alpha=0$ is similar to \eqref{eqn:D-flat}, and we give the details here. Let $A:=D\alpha=\nabla\alpha-\delta\alpha+\frac{1}{\hbar}[I,\alpha]_{\mathcal{W}}$. First we have
\begin{align*}
\delta^{-1}A & = \delta^{-1}(\nabla\alpha+\frac{1}{\hbar}[I,\alpha]_{\mathcal{W}})-\delta^{-1}\delta(\alpha)\\
             & = \delta^{-1}(\nabla\alpha+\frac{1}{\hbar}[I,\alpha]_{\mathcal{W}})-(\alpha-\alpha_0)=0.
\end{align*}
And there is
$$
A=\pi_0(A)+\delta^{-1}(\nabla A+\frac{1}{\hbar}[I,A]),
$$
which, together with the fact that $\pi_0(A)=\bar{\partial}(\alpha_0)=0$ implies that $A=0$.
\end{proof}
\begin{rmk}
This proposition is the complex analogue of the fact that flat sections of the real Weyl bundle under Fedosov's abelian connection has a one-to-one correspondence with smooth functions.
\end{rmk}

\subsection{Classical action functional}

Let $I$ be a solution of equation \eqref{eqn:equivalent-flatness-equation-D}. Then $$D:=\nabla-\delta+\frac{1}{\hbar}[I,-]_{\mathcal{W}}$$ defines a square zero operator and thus a (curved) $L_\infty$ structure on $\g_X$ (where $\g_X[1]:=\A_X^*\otimes_{\mathcal{O}_X} T_X$), which is equivalent to the one defined in \cite{Kapranov}. As we have seen in the construction of $I$, the leading cubic term is induced from the curvature tensor $R$ of the connection $\nabla$.

Let $M$ be a closed $3$-dimensional manifold, and let $\A(M)$ denote the space of differential forms on $M$. The space of fields of our Rozansky-Witten model is given by
$$
\E:=\A(M)\otimes_{\mathbb{C}} \g_X.
$$
Together with the Poincar\'{e} pairing on $\A(M)$, we obtain the following pairing on the $\E$:
\begin{align*}
\langle-,-\rangle:\E\otimes_{\A_X}\E&\rightarrow\A_X,\\
                  \langle\alpha\otimes g_1,\beta\otimes g_2\rangle&:=\omega(g_1,g_2)\cdot\int_M\alpha\wedge\beta.
\end{align*}
\begin{defn}
We define the space of functionals on $\E$ which are valued in  $\A_X[[\hbar]]$ by
\[
\OO(\E):=\widehat{\Sym}(\E^\vee)[[\hbar]]:=\prod_{k\geq 0} \OO^{(k)}(\E):=\prod_{k\geq 0} \Sym^k_{\A_X}(\E^\vee)[[\hbar]],
\]
where $\Sym^k_{\A_X}(\E^\vee)$ is the (graded)-symmetric $\A_X$-linear completed tensor product. Further,  we will denote by $\Ol(\E)\subset \OO(\E)$ the subspace of {\em local functionals}, i.e. those of the form given by the integration of a Lagrangian density on $M$
$$
     \int_M \mathcal L(\mu), \quad \mu \in \E.
$$
\end{defn}

There exists a natural map
\[
   \rho: \mathcal{W} \to \Ol(\E)
\]
defined as follows. Given  a section $I$ of
$
\Sym^k(T^\vee X),
$
we can associate an ($\A_X$-valued) functional on $\g_X[1]$ using the natural pairing between $T_X$ and $T^\vee_X$. This functional can be  further extended to $\E$ via integration of differential forms over $M$. Explicitly,
\begin{align}\label{rho-map}
\rho(I):\mathcal{E}\rightarrow \A_X,\quad
\alpha \mapsto\frac{1}{k!}\int_{M}I(\alpha,\cdots,\alpha).
\end{align}
The $\A_X$-linear extension of such assignment defines $\rho$. Since $\rho(I)$ requires that the total degree of differential forms on $M$ in the input to be $3$ (for the integration on $M$) , our convention is that
\[
  \rho(I)=0 \quad \text{if}\ I\in \A_X[[\hbar]] .
\]
\begin{defn}\label{defn:classical-action-functional}
Let $I\in \mathcal{W}$ be a solution of equation \eqref{eqn:equivalent-flatness-equation-D}, the classical action functional will be of the form
\begin{equation}\label{eqn:action-functional}
S(\alpha):=\int_{M}\omega(d_M\alpha+\nabla\alpha,\alpha)  +  \rho(I-\delta^{-1}(\omega))(\alpha), \quad \alpha\in \E.
\end{equation}
Here $d_M$ denotes the de Rham differential on $M$ and we have used the holomorphic symplectic form $\omega$ to pair factors in $T_X$. The first term constitutes the free part of the theory, which defines a derivative
\[
   Q= d_{M}+\nabla.
\]
\end{defn}

In terms of the $L_\infty$ structure on $\g_X$, the classical action functional can be written in the following explicit form:  Let $\alpha\in\E=\A(M)\otimes\g_X[1]$, then
\begin{equation}\label{eqn:classical-action-functional}
S(\alpha):=\langle d_M\alpha,\alpha\rangle+\sum_{k\geq 0}\frac{1}{(k+1)!}\langle l_k(\alpha^{\otimes k}),\alpha\rangle.
\end{equation}
We will split the action functional $S$ as the sum of its free and interaction parts:
\begin{align*}
I_{cl}(\alpha) & := \sum_{k\not=1}\frac{1}{(k+1)!}\langle l_k(\alpha^{\otimes k}),\alpha\rangle, \\
S_{free}(\alpha) & := \langle (d_M+\frac{l_1}{2})\alpha,\alpha\rangle=\langle Q\alpha,\alpha\rangle
\end{align*}
In particular, we will let $\tilde{l}_k$ denote the following component of the classical action:
\begin{equation}\label{eqn:l-k-tilde}
\tilde{l}_k(\alpha):=\frac{1}{(k+1)!}\langle l_k(\alpha^{\otimes k}),\alpha\rangle.
\end{equation}

\begin{rmk}
The space of fields $\E$ is $\Z_2$-graded, and the grading is induced by the standard $\mathbb{Z}$-grading on $\A(M)$ and $\A_X$.
\end{rmk}

\begin{lem-defn}\label{Poisson-bracket}
The symplectic pairing $\abracket{-,-}$ induces an odd Poisson bracket of degree $1$ on the space of local functionals, denoted by
$$
  \fbracket{-,-}: \Ol(\E)\otimes_{\A_{X}^{\sharp}} \Ol(\E) \to \Ol(\E),
$$
which is bilinear in $\A_X^{\sharp}$.
\end{lem-defn}

The flatness of the operator $D$ (or equivalently, equation \eqref{eqn:equivalent-flatness-equation-D}) implies the classical master equation of the action functional:
\begin{prop}
The classical action functional $S$ of the Rozansky-Witten model satisfies the following {\em classical master equation (CME)}:
\begin{equation}\label{eqn:classical-master-equation}
\nabla I_{cl}+\frac{1}{2}\fbracket{I_{cl},I_{cl}}+\rho(R)=0.
\end{equation}
\end{prop}
\begin{rmk}
Equation \eqref{eqn:classical-master-equation} is a little different from the usual classical master equation, due to the term $\rho(R)$. This term also reflects the curving
in the underlying $L_\infty$ structure.
\end{rmk}

\subsection{Comparison with the original RW model}\label{subsection:comparison-physical-model}

In this subsection, we compare the Rozansky-Witten model we defined with the original Rozansky-Witten theory defined in \cite{RW}.

We first recall the original definition of Rozansky-Witten model in \cite{RW}. Given a holomorphic symplectic manifold $(X,\omega)$ as the target space, the field theory has bosonic fields the space of all smooth maps $\phi: M\rightarrow X$, and fermionic fields $\eta\in\Gamma(M,\phi^*\bar{T}_X)$ and $\chi\in\Gamma(M,\A_M^1\otimes\phi^*T_X)$. In terms of local coordinates on the target and domain, these fields can be described as $\phi^I(x^\mu),\bar{\phi}^{\bar{I}}(x^\mu),\chi^I(x^\mu)$ and $\eta^{\bar{I}}(x^\mu)$. There exists an odd vector field $\bar{Q}$ on the whole space of fields, satisfying the condition
$$
\fbracket{\bar{Q},\bar{Q}}=0.
$$
This vector field is known as the BRST operator in the physics literature.

The Lagrangian of the Rozansky-Witten model is of the form $S=\bar{Q}(V)+S_0$, where $S_0$ is a $\bar{Q}$-closed functional explicitly given by:
\begin{equation}\label{eqn:original-RW-model}
S_0(\chi, \eta)=\int_M\frac{d^3x}{2}\epsilon^{\mu\nu\rho}\left(\omega_{IJ}\chi_\mu^I\nabla_\nu\chi_\rho^J-\frac{1}{3}\omega_{IJ}R^J_{KL\bar{M}}\chi_\mu^I\chi_\nu^K\chi_\rho^L\eta^{\bar{M}}+\frac{1}{3}(\nabla_L\omega_{IK})(\partial_\mu\phi^I)\chi_\nu^K\chi_\rho^L\right).
\end{equation}
Here  $\nabla$ denotes a symmetric connection on $T_X$ (and also its pull back on $\phi^*(T_X)$). Let us make the comparison of the Lagrangian in equation \eqref{eqn:original-RW-model} with the one we defined in terms of Weyl bundle more explicitly. It is easily seen that the first term in equation \eqref{eqn:original-RW-model} given by the pullback  connection on $\phi^*(T_X)$ exactly corresponds to the free term in our Lagrangian. For the second term given by the curvature tensor, we can see that the interaction terms in our Lagrangian is nothing but the Taylor expansion of the curvature tensor. In particular, if we choose a connection $\nabla$ compatible with the symplectic form $\omega$, then the last term in the above expression vanishes. In this case, we list the terms of Lagrangian density in these two models in the following table:

\begin{center}
  \begin{tabular}{ | l | p{4.5cm} | p{6cm} | }
    \hline
     & free term & interaction term \\ \hline
    Our RW model & $\omega(d_M\alpha+\nabla\alpha,\alpha)$ & $\rho(I-\delta^{-1}(\omega))(\alpha)$ \\ \hline
    Original RW model & $\frac{1}{2\sqrt{h}}\epsilon^{\mu\nu\rho}\omega_{IJ}\chi_\mu^I\nabla_\nu\chi_\rho^J$ & $-\frac{1}{2\sqrt{h}}\epsilon^{\mu\nu\rho}\frac{1}{3}\omega_{IJ}R^J_{KL\bar{M}}\chi_\mu^I\chi_\nu^K\chi_\rho^L\eta^{\bar{M}}$ \\
    \hline
  \end{tabular}
\end{center}

In the case where the curvature does not have a $(2,0)$-part, it is not difficult to observe that the action functional of our RW model is the formal version, i.e. the Taylor expansion of the functional $S_0$ around the space of constant maps from $M$ to $X$: locally around a point $x\in X$, the exponential map associated to the connection $\nabla$ trivializes the bundle $T_X$ with fiber $T_{X,x}$. The bosonic fields consisting of all smooth maps can be replaced by a function on $M$ valued in a tangent space of $X$. The fields $\chi$ (or $\rho$) can be locally described as $1$-forms (or functions) on $M$ valued in the vector space $T_{X,x}$ (or $\bar{T}_{X,x}$). Thus the field content in the formal version consists of both $0$ and $1$-forms valued in $T_X$.

Similar to the Chern-Simons theory, after the BV extension, our model has field content $\A_M\otimes T_X$ which contains differential forms on $M$ of all degrees.  The leading interaction term in our classical action functional is exactly given by the curvature of $\nabla$, and the higher order terms are nothing but the Taylor coefficients of the curvature around the fixed point $x\in X$.


\subsection{Dimension reduction and B-model}

Let $M$ be a $3$-dimensional manifold  of the form
$$
M=\Sigma_g\times S^1,
$$
where $\Sigma_g$ is a genus $g$ Riemann surface. In physics, the reduction of RW model on a circle $S^1$ is obtained by letting the size of the circle go to $0$, and only the zero modes in the corresponding Fourier modes are left. Mathematically speaking, this is nothing but looking at those fields which are harmonic when restricted to $S^1$. More explicitly, the space of fields of the reduced theory is
\begin{align*}
\mathcal{E}_{red}&=(\mathcal{A}^*(\Sigma_g)\otimes \mathcal{H}^0(S^1))\otimes\g_X[1]+(\mathcal{A}^*(\Sigma_g)\otimes \mathcal{H}^1(S^1))\otimes\g_X[1]\\
&\cong \mathcal{A}^*(\Sigma_g)\otimes\g_X[1]+\mathcal{A}^*(\Sigma_g)\otimes\g_X^\vee,
\end{align*}
where in the second line we have applied the isomorphism
$$
\g_X\cong \g_X^\vee
$$
using the holomorphic symplectic structure. It is then obvious that $\E_{red}$ is exactly the space of fields of the B-model with source $\Sigma_g$ and target $X$ defined in \cite{Li-Li}. Moreover, it is straightforward to check that the classical action functional of the RW model reduces to that of the B-model there. Thus, we have shown the following proposition:
\begin{prop}
On a $3$-dimensional manifold of the form $\Sigma_g\times S^1$, the reduction of the classical Rozansky-Witten model on the circle $S^1$ is the topological B-model on $\Sigma_g$ with the same target $X$ as defined in \cite{Li-Li}.
\end{prop}

\section{Quantization}\label{sec:quantization}
In this section, we establish the quantization of the Rozansky-Witten model by Costello's perturbative renormalization method \cite{Kevin-book}. Let us recall that a quantization of a classical field theory with interaction functional $I_{cl}$ is given by a family of functionals $\{I[L]\}$ parametrized by the scales $L\in\mathbb{R}_{>0}$, which is compatible with $I_{cl}$ in the following sense
\begin{equation}\label{eqn:compatibility-classical-quantum}
I_{cl}=\lim_{L\rightarrow 0}I[L]\ (\text{mod}\ \hbar),
\end{equation}
and which satisfies the renormalization group equation (RGE) that describes the compatibility between different scales and the quantum master equation (QME) which describes the quantum gauge symmetry. These functionals $\{I[L]\}$ are called effective interactions.

Costello's strategy for constructing a quantization of a classical theory is to first define interactions $\{I_{naive}[L]\}$ satisfying the RGE; this is called the {\em naive quantization}. In general, the interactions $\left\{ I_{naive}[L] \right\}$ {\em do not} satisfy the QME, and one needs to add quantum correction terms, which are power series in $\hbar$, in order for the QME to be satisfied. Moreover, this latter step cannot always be achieved due to certain {\em cohomological obstructions} (or so-called {\em anomaly} in the physics literature). We show in this section that there is no cohomological obstruction to the quantization of our RW model. More precisely, the naive quantization constructed using configuration spaces following Kontsevich \cite{Kontsevich} and Axelrod-Singer \cite{AS2}, satisfies the QME with any fixed (not necessarily flat) metric $g$ on $M$.

The organization of this section is as follows: in Section \ref{subsection: gauge-fixing}, we introduce the gauge fixing operator and propagators and then briefly recall the notion of Feynman weights and renormalization group flow (RG flow) operator. Next, we recall the definition of configuration spaces and their compactifications. In Section \ref{subsection: naive quantization}, we will analyze the singularities of the propagators $P_0^L$ and show that they can be lifted to the compactification of configuration spaces. This immediately leads to our construction of the naive quantization of the Rozansky-Witten model. In Section \ref{subsection:QME}, we prove that the naive quantization satisfies the quantum master equation for any fixed metric on $M$. Finally, in Section \ref{section: extended-QME}, we show that in the case when there is a family of Riemannian metrics on $M$, after adding certain Chern-Simons type local quantum corrections, the QME is also satisfied. This in particular implies that the corresponding quantization is independent of the metric.

\subsection{Gauge fixing}\label{subsection: gauge-fixing}
A perturbative quantization of a classical field theory is to model the infinite dimensional path integral on a subspace $L$ associated to a gauge fixing:
$$
\int_{L\subset\E}e^{S/\hbar}
$$
using the Feynman graph expansion. This is an analogue of Feynman's theorem on the asymptotic expansion of the Gaussian type integral in the finite dimension case. Roughly speaking, for a graph $\gamma$, we label each vertex in $\gamma$ of valency $k$ by the term $\tilde{l}_{k-1}$ in the classical interaction functional $I_{cl}$ and label each edge by the propagator (to be explained later). An essential difficulty is that the propagator has certain singularities originating from the fact that $\E$ (and $L$) are infinite dimensional. We are going to apply the technique of compactification of configuration spaces to regularize the Feynman integrals in our Rozansky-Witten model.

We first need to choose a gauge fixing for regularization. By choosing a Riemannian metric $g$ on the $3$-dimensional manifold $M$, we define the gauge fixing operator to be $Q^{GF}:=d_M^*$ associated to the metric $g$.
\begin{defn}
Let $H:=[Q,Q^{GF}]$ be the Laplacian operator. The {\em heat kernel} $\mathbb{K}_t$ for $t>0$ is the element in $\Sym^2\bracket{\E}$ defined by the equation
$$
\langle \mathbb{K}_t(x,y),\phi(y)\rangle=e^{-tH}(\phi)(x), \quad \forall \phi\in \E, x\in M.
$$
Contraction with $\mathbb{K}_t$ defines an operator $\Delta_t$ on $\mathcal{O}(\E)$. The failure of $\Delta_t$ being a derivation defines the scale $t$ BV-bracket $\fbracket{-,-}_t$:
$$
\fbracket{\phi_1,\phi_2}_t:=\Delta_t(\phi_1\phi_2)-\Delta_t(\phi_1)\cdot\phi_2-(-1)^{|\phi_1|}\phi_1\cdot\Delta_t(\phi_2).
$$
\end{defn}
\begin{rmk}
Recall that the pairing $\langle-,-\rangle$ on $\E$ is the tensor product of the Poincar\'{e} pairing on $\A(M)$ and the symplectic pairing $\omega$ on $\g_X$. This implies that the heat kernel $\mathbb{K}_t(x,y)$ splits into its analytic and combinatorial parts, which are given by the heat operator $K_t$ on $\A(M)$ and the Poisson kernel of $\omega$ respectively.
\end{rmk}
For any $0<\epsilon<L<\infty$, the effective propagator is defined as
$$
\mathbb{P}_\epsilon^L:=\int_{\epsilon}^L(d^*_M\otimes 1)\mathbb{K}_tdt,
$$
which also splits into its analytic and combinatorial parts.

\begin{defn}
A {\em graph} $\gamma$ consists of the following data:
\begin{enumerate}
 \item A finite set of vertices $V(\gamma)$;
 \item A finite set of half-edges $H(\gamma)$;
 \item An involution $\sigma: H(\gamma)\rightarrow H(\gamma)$. The set of fixed points of this map is denoted by $T(\gamma)$ and is called the set of tails of $\gamma$. The set of two-element orbits is denoted by $E(\gamma)$ and is called the set of internal edges of $\gamma$;
 \item A map $\pi:H(\gamma)\rightarrow V(\gamma)$ sending a half-edge to the vertex to which it is attached;
 \item A map $g:V(\gamma)\rightarrow \mathbb{Z}_{\geqslant 0}$ assigning a genus to each vertex.
\end{enumerate}
It is clear how to construct a topological space $|\gamma|$ from the above abstract data. A graph $\gamma$ is called {\em connected} if $|\gamma|$ is connected. A graph is called {\em stable} if every vertex of genus $0$ is at least trivalent, and every genus $1$ vertex is at
least univalent. The {\em genus} of the graph $\gamma$ is defined to be
$$g(\gamma):=b_1(|\gamma|)+\sum_{v\in V(\gamma)}g(v),$$
where $b_1(|\gamma|)$ denotes the first Betti number of $|\gamma|$.
\end{defn}

Let us now briefly recall the Feynman weights associated to a stable graph $\gamma$. Let $\OO^+(\E)\subset \OO(\E)$ denote the subspace of functionals on $\E$ which are at least cubic modulo $\hbar$, and let $I\in\OO^+(\E)$ be a functional with the following decomposition:
$$
I=\sum_{i\geq 0}\sum_{k\geq 0}\hbar^kI^{(k)}_i,
$$
where $I^{(k)}_i\in\Sym^i(\E^\vee)$. Let $\gamma$ be a stable graph. For every vertex $v$ of valency $i$ and genus $k$, we  label it by the term $I^{(k)}_i$. For every edge $e$, we label it by the propagator $\mathbb{P}_\epsilon^L\in\Sym^2(\E)$, and for every tail we label it by an input $\phi_i\in\E$. The tensor product of the propagators and the inputs $\phi_i$'s can be contracted with that of the functionals $I_i^{(k)}$'s labeling the vertices, whose output is defined to be the Feynman weight:
$$
W_\gamma(\mathbb{P}_\epsilon^L, I)(\phi_1,\cdots,\phi_n).
$$
\begin{defn}\label{defn:RG-flow-operator}
The {\em renormalization group flow (RG flow) operator} from scale $\epsilon$ to scale $L$ is the map
$$
W(\mathbb{P}_\epsilon^L,-):\mathcal{O}^+(\mathcal{E})\rightarrow\mathcal{O}^+(\mathcal{E})
$$
defined by taking the sum of Feynman weights over all connected stable graphs:
$$
I\mapsto\sum_{\gamma}\dfrac{\hbar^{g(\gamma)}}{|\text{Aut}(\gamma)|}W_\gamma(\mathbb{P}_\epsilon^L,I).
$$
A collection of functionals
$$
\{I[L]\in\mathcal{O}^+(\mathcal{E})|L\in\mathbb{R}_+\}
$$
is said to satisfy the {\em renormalization group equation (RGE)} if for any $0<\epsilon<L<\infty$, we have
$$
I[L]=W(\mathbb{P}_\epsilon^L,I[\epsilon]).
$$
\end{defn}
\begin{rmk}
The RGE  can be formally expressed as
$$
e^{I[L]/\hbar}=e^{\hbar\partial_{\mathbb{P}_\epsilon^L}}e^{I[\epsilon]/\hbar}.
$$
\end{rmk}
For more details on stable graphs and Feynman weights, we refer the reader to \cite{Kevin-book}. We now give the precise definition of a quantization:
\begin{defn}
Let $I\in\mathcal{O}_{loc}(\mathcal{E})$ be the interaction functional of a classical field theory satisfying the classical master equation.  A {\em quantization} of this classical theory consists of a collection $\{I[L]\in\mathcal{O}^+(\mathcal{E})|L\in\mathbb{R}_+\}$ of effective functionals such that
\begin{enumerate}
 \item The renormalization group equation (RGE) is satisfied;
 \item  The functionals $\{I[L]\}$ satisfy a locality axiom, saying that as $L\rightarrow 0$ the functional $I[L]$ becomes more and more local;
 \item For every $L>0$, the functional $I[L]$ satisfies the scale $L$ {\em quantum master equation (QME)}:
\begin{equation}\label{eqn:QME}
\left(Q_L+\hbar\Delta_L+\frac{\rho(R)}{\hbar}\right)e^{I[L]/\hbar}=0,
\end{equation}
where the operator $Q_L:\E\rightarrow\E$ is defined by
$$
Q_L:=Q-\nabla^2\int_0^LQ^{GF}e^{-t[Q,Q^{GF}]}dt.
$$
 \item Modulo $\hbar$, the $L\rightarrow 0$ limit of $I[L]$ agrees with the classical interaction functional $I_{cl}$ (equation \eqref{eqn:compatibility-classical-quantum}).
\end{enumerate}
\end{defn}

\begin{rmk}
Usually, the quantum master equation reads $(Q+\hbar\Delta_L)e^{I[L]/\hbar}=0$; the extra term $\frac{\rho(R)}{\hbar}$ in the above definition is due to the non-vanishing of the  square of the operator $Q$ in the free part of the classical action: $ Q^2 = \nabla^2=\{\rho(R),-\}$.
\end{rmk}

\begin{rmk}
As explained in \cite{Li-Li}, the quantum master equation \eqref{eqn:QME} implies that the operator
$$
Q_L+\fbracket{I[L],-}_L+\hbar\Delta_L
$$
squares $0$, which defines a quantization of the classical differential $Q+\fbracket{I_{cl},-}$. 
\end{rmk}

\subsection{Regularization}\label{subsection: naive quantization}

In this subsection, we construct a naive quantization of the classical theory of Rozansky-Witten model using the technique of configuration spaces.

\subsubsection{Asymptotic expansion of the heat kernel}
In the previous subsection, the fact that effective propagators $P_\epsilon^L$ are smooth $2$-forms on $M\times M$ is essential for the definition of the RG flow operator. However, in order to define a quantization of the classical Rozansky-Witten model, we need to take certain limits involving $P_\epsilon^L$ as $\epsilon\rightarrow 0$. But $P_0^L$ is only smooth on the complement of the diagonal $M\times M\setminus\Delta$, and this is the reason why we need to compactify the configuration spaces for regularization.

By blowing up the diagonal $\Delta$ in $M\times M$, we obtain a compactification of $M\times M\setminus\Delta$. We will show that the propagator $P_0^L$ can be smoothly extended to the compactification by analyzing the singularity of $P_\epsilon^L$ near the diagonal $\Delta$.

First of all, the heat kernel $K_t(x,y)$ is a smooth $3$-form on $M\times M$ for $t>0$. Since we will mainly focus on the behavior of $K_t$ around the diagonal $\Delta\subset M\times M$, we will use the Riemann normal coordinates: let $(x^1,x^2,x^3)$ be local coordinates on a small open set $U\subset M$, we can always represent a point $(x,y)\in U\times U$ uniquely as (since $U\times U$ is close enough to the diagonal $\Delta$)
\begin{equation}\label{eqn:change-of-coordinates}
(x,y)=(\exp_z(u),\exp_z(-u)),
\end{equation}
where $\exp_z(u)$ denotes the exponential map with initial point $z$ and initial velocity $u$.
Let $g_{ij}(z)dz^idz^j$ denote the Riemannian metric at the point $z\in M$. Then the geodesic distance between $x$ and $y$ can be expressed as $\rho(x,y)=2\left(g_{ij}(z)u^iu^j\right)^{1/2}$, which we denote by $||u||_z$.

It is well-known that $K_t(x,y)$ has a small $t$ asymptotic expansion of the form:
$$
K_t(x,y)\sim (4\pi t)^{-\frac{3}{2}}e^{-\frac{||u||_z^2}{4t}}\left(\phi_0(x,y)+t^{\frac{1}{2}}\phi_1(x,y)+\cdots\right),
$$
where $\phi_i$'s are  smooth $3$-forms on $M\times M$. By the change of coordinates in equation \eqref{eqn:change-of-coordinates} (on a neighborhood of the diagonal $\Delta$), and taking the Taylor series expansions of $\phi_i$'s along $\Delta\subset M\times M$, we obtain the following asymptotic expansion of $K_t$ as $t+||u||_z\rightarrow 0$:
\begin{equation}\label{eqn:asymptotic-expansion-no-forms}
K_t(x,y)\sim  (4\pi t)^{-\frac{3}{2}}e^{-\frac{||u||_z^2}{4t}}\left(\psi_0(t^{\frac{1}{2}},u,z)+\psi_1(t^{\frac{1}{2}},u,z)+\cdots\right),
\end{equation}
where $\psi_i$ is a polynomial in $t^{\frac{1}{2}}$ and $u$ of total degree $i$ with functions of $z$ as coefficients.

It is natural to assign a degree to each term in the asymptotic expansion \eqref{eqn:asymptotic-expansion-no-forms}:
$$
\deg((t^{\frac{1}{2}})^k):=k,\qquad \deg(u^i):=1.
$$
Then the leading singularity of $K_t$ with respect to this grading is given by
\begin{equation}\label{eqn:asym-expansion-leading-term}
(K_t)_{(-3)}= (4\pi t)^{-\frac{3}{2}}e^{-\frac{||u||_z^2}{4t}}\sqrt{\det(g_{ij}(z))}du^1du^2du^3.
\end{equation}

Following the notations in \cite{AS2}, we let $\OO_x$ denote the heat operator
$$\frac{\partial}{\partial t}+H_x,$$
where $H_x$ is the Laplacian acting on the first copy in $M\times M$. We can decompose $\OO_x$ with respect to the above grading, whose leading term is of degree $-2$ and given explicitly by
$$
(\OO_x)_{(-2)}=\frac{\partial}{\partial t}-\frac{1}{4}g^{ij}(z)\frac{\partial}{\partial u^i}\frac{\partial}{\partial u^j}.
$$
Here $\left(g^{ij}(z)\right)$ is the inverse to the Riemannian metric $(g_{ij}(z))$. Similarly, by letting $(K_t)_{(p,r)}$ denote the piece of $(K_t)_{(p)}$ which are $r$-forms in $du^i$'s, we can decompose the heat operator $\OO_x$ according to the grading.

Then an equivalent way of describing the terms $\psi_i$'s in equation \eqref{eqn:asymptotic-expansion-no-forms} is that they are uniquely determined by the following equations:
\begin{equation}\label{eqn:eqn-determine-asymp}
\begin{aligned}
(\OO_x)_{-2}\left((K_t)_{(-3)}\right)&=0,\\
(\OO_x)_{-2}\left((K_t)_{(p)}\right)&=\sum_{-3\leq l\leq p-1}(\OO_x)_{-2+p-l} (K_t)_{(l)}, \ \text{for\ } p\geq -2.
\end{aligned}
\end{equation}
This expression immediately leads to  the following key lemma:

\begin{lem}\label{lem:vanishing-asymp-heat}
The term $(K_t)_{(p,r)}$ in the asymptotic expansion of $K_t$ vanishes if $p+r<0$.
\end{lem}
\begin{proof}
We prove the lemma by induction on $p$. For $p=-3$, the statement follows from  the explicit expression in equation \eqref{eqn:asym-expansion-leading-term}. From equation \eqref{eqn:eqn-determine-asymp}, we have
$$
(\OO_x)_{(-2,0)}(K_t)_{(p,r)}=\sum_{-3\leq l\leq p-1}\sum_{q}(\OO_x)_{(-2+p-l,q)} (K_t)_{(l,r-q)}. 
$$
The fact that $p+r<0$ implies that either $-2+p-l+q<-2$ which implies that $(\OO_x)_{(-2+p-l,q)}=0$ by the following Lemma \ref{lem:leading-term-heat-operator-total-degree}, or $l+r-q<0$ which implies that $ (K_t)_{(l,r-q)}=0$ by induction hypothesis.
\end{proof}
\begin{rmk}
 This is the heat kernel analogue of the discussion after equation (4.55) in \cite{AS2}.
\end{rmk}


For our later discussion on the lifting of propagators on compactified configuration spaces, we will also need to consider the degree of the differential form $du^i$'s. Let
$$
(K_t)_{[k]}:=\sum_{p+q=k}(K_t)_{(p,q)}
$$
be the sum of terms  of ``total degree'' $k$ in the asymptotic expansion \eqref{eqn:asymptotic-expansion-no-forms}. The following lemma provides the leading term of the heat operator with respect to the total degree, which follows easily from equation (4.50.1) in \cite{AS2}:
\begin {lem}\label{lem:leading-term-heat-operator-total-degree}
 The leading term of the heat operator  is given by
\begin{align*}
(\OO_x)_{[-2]}=&\frac{\partial}{\partial t}-\frac{1}{4}g^{ij}(z)\left(\frac{\partial}{\partial
u^i}-\Gamma_{ik}^l(z)dz^ki\left(\frac{\partial}{\partial u^l}\right)\right)\left(\frac{\partial}{\partial
u^j}-\Gamma_{jm}^n(z)dz^mi\left(\frac{\partial}{\partial u^n}\right)\right)\\
&+\frac{1}{4}g^{jk}(z)\Omega^i_ki\left(\frac{\partial}{\partial u^i}\right)i\left(\frac{\partial}{\partial u^j}\right).
\end{align*}
Here $\Gamma_{ik}^l$'s denote the Christoffel symbols of the Levi-Civita connection on $M$, and $i\left(\dfrac{\partial}{\partial
u^n}\right)$ denotes the contraction of forms with the vector field $\dfrac{\partial}{\partial u^n}$.
\end {lem}

Lemma \ref{lem:vanishing-asymp-heat} implies that $(K_t)_{[k]}=0$ if $k<0$, and the leading term of its asymptotic expansion with respect to the total degree is uniquely determined by the following property:
\begin{enumerate}
 \item The component of degree $3$ in $du^i$'s is given by equation \eqref{eqn:asym-expansion-leading-term},
 \item $(\OO_x)_{[-2]}(K_t)_{[0]}=0$.
\end{enumerate}
A straightforward calculation gives the leading terms of $K_t$ and $(d^*\otimes 1)K_t$ as described in the following lemma, the proof of which involves long and tedious calculation and will be given in Appendix \ref{appendix:proof-regularization}:
\begin{lem}\label{lemma:leading-term-K-t}
The leading term of $K_t$ is given by
 $$(K_t)_{[0]}=\frac{(4\pi)^{-\frac{3}{2}}}{6}e^{-\frac{||u||_z^2}{4t}}\det(g_{mn}(z))^{1/2}\epsilon_{ijk}\left(t^{-\frac{3}{2}}d_{vert}u^id_{vert}u^jd_{vert}u^k+6t^{
-\frac{1}{2}} \Omega_l^ig^{lj}(z)d_{vert}u^k\right),$$
and the leading term of $(d^*\otimes 1)K_t$ is
\begin{align*}
\left((d^*\otimes 1)K_t\right)_{[-2]}=\ &(4\pi)^{-\frac{3}{2}}\frac{||u||_z^3}{4t^{5/2}}e^{-\frac{||u||_z^2}{4t}}\det(g(z))^{1/2}\epsilon_{ijk}\hat{u}^id_{vert}\hat{u}^jd_{vert}\hat{u}^k\\
&+(4\pi)^{-\frac{3}{2}}\frac{u^k}{2t^{3/2}}e^{-\frac{||u||_z^2}{4t}}\det(g(z))^{1/2}\epsilon_{ijk}\Omega_l^ig^{lj}(z).
\end{align*}
Here $d_{vert}u^i=du^i+\Gamma_{jk}^idz^ju^k$, and $\Omega_k^i$ denotes the curvature matrix of $2$-forms of the Levi-Civita
connection, and $\hat{u}=u/||u||_z$ is a unit tangent vector with respect to the Riemannian metric.
\end{lem}


This lemma will be used in subsequent sections to define the naive quantization of our model.

\subsubsection{Configuration spaces and naive quantization}
The idea of using configuration spaces to quantize a classical field theory was first introduced by Kontsevich in the case of Chern-Simons theory \cite{Kontsevich, Kontsevich-notes}. Not long after, Axelrod and Singer \cite{AS2} gave a systematic construction of compactifications of configuration spaces for general Riemannian manifolds, and showed that the ``infinite scale'' propagator $P_0^\infty$ or the Green function can be lifted smoothly to these configuration spaces. Here, we develop a heat kernel version of the Axelrod-Singer construction, and show that the propagator $P_0^L$ for all $L>0$ admits such a lifting. This enables us to define the RG flow operators from scale $0$ to $L$.

We briefly recall the construction of the compactified configuration space $M[V]$ as a smooth manifold with corners, and refer the readers to \cite{AS2} for details.
Let $V:=\{1,\cdots,n\}$ for some integer $n\geq 2$.
Let $S\subset V$ be any subset with $|S|\geq 2$. We denote by $M^S$ the set of all maps from $S$ to $M$ and by $\Delta_S \subset M^S$ the small diagonal. The real blow up of $M^S$ along $\Delta_S$, denoted by $\text{Bl}(M^S,\Delta_S)$, is a manifold with boundary whose interior is diffeomorphic to $M^S\setminus\Delta_S$ and whose boundary is diffeomorphic to the unit sphere bundle associated to the normal bundle of $\Delta_S$ inside $M^S$.

\begin{defn}
Let $V$ be as above and let $M_0^V$ be the configuration space of $n=|V|$ pairwise different points in $M$,
$$
M_0^V:=\{(x_1,\cdots, x_n)\in M^{V}: x_i\not=x_j\ \text{for}\ i\not=j\}.
$$
There is the following embedding:
$$
M_0^V\hookrightarrow M^V\times\prod_{|S|\geq 2}\text{Bl}(M^S,\Delta_S).
$$
The space $M[V]$ is defined as the closure of the above embedding.
\end{defn}

\begin{rmk}
We will denote $M[V]$ by $M[n]$ for $V=\{1,\cdots, n\}$.
\end{rmk}

\begin{eg}
The compactification $M[2]$ is a smooth manifold whose boundary is diffeomorphic to the unit sphere bundle of $TM$:
$$
\partial M[2]\cong S(TM).
$$
After fixing  local coordinates, every point in a neighborhood of the boundary $\partial M[2]$ can be expressed as
$$
((z,\hat{u}),r).
$$
Here $\hat{u}=u^i\frac{\partial}{\partial z^i}\in T_z(M)$ denotes a unit tangent vector at $z\in M$, and $r\geq 0$. Moreover, we have the following explicit map:
\begin{equation}\label{eqn:M[2]-embedding}
 \begin{aligned}
 S(TM)\times (0,\epsilon)&\rightarrow M\times M\setminus\Delta,\\
 ((z,\hat{u}),r)&\mapsto \left(\exp_z(r\hat{u}),\exp_z(-r\hat{u})\right).
 \end{aligned}
\end{equation}
\end{eg}

It follows from the definition of $M[V]$ that for every subset $S\subset V$, there is a natural projection
$$
\pi_S: M[V]\rightarrow M[S].
$$
We will need the following lemma to exclude graphs with self-loops (edges that is connected to the same vertex) in the RG flow of Rozansky-Witten model:
\begin{lem}\label{lem:self-loop-vanishing}
Let $\gamma$ be a graph that contains an internal edge which connects to the same vertex. Then the corresponding Feynman weight vanishes.
\end{lem}
\begin{proof}
It is enough to look at the following trivalent graph with a single vertex, the argument for other such graphs is similar.
$$
\figbox{0.23}{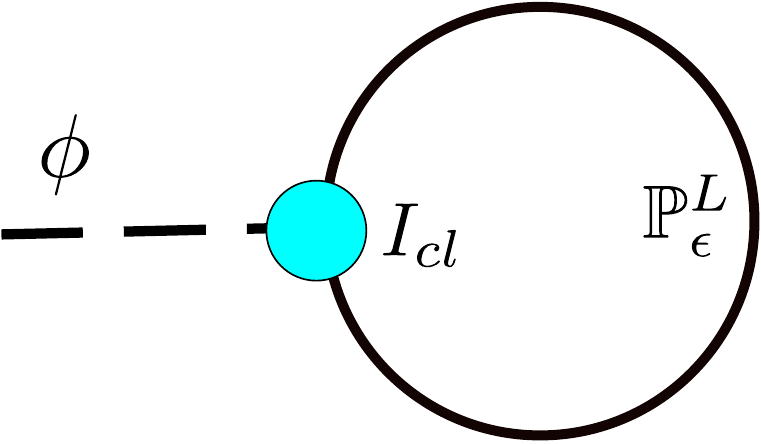}
$$
We consider the ``combinatorial'' part of the functional. Let $\phi$ be the input on the tail. Then the symmetric property of the functional implies the vanishing. For instance, there is the following
\begin{align*}
 \omega^{ij}\langle l_2(e_i\otimes e_j),\phi\rangle = \omega^{ij}\langle l_2(e_j\otimes e_i),\phi\rangle = -\omega^{ji}\langle l_2(e_j\otimes e_i),\phi\rangle.
\end{align*}
This proves the lemma.
\end{proof}
Let $\gamma$ be  a graph without self-loops. For every edge $e\in E(\gamma)$, let $v(e)$ denote the 2-point set of the vertices incident to $e$. There are then natural maps
\begin{align*}
&\pi_e: M[V(\gamma)]\rightarrow M[v(e)], \\
&\pi_v: M[V(\gamma)]\rightarrow M
\end{align*}
for all internal edges $e$ and vertices $v$ in $\gamma$.

\begin{prop}\label{prop:lift-propagator-boundary}
The propagator $P_0^L$ which is only smooth on $M\times M\setminus\Delta$ can be lifted to a smooth $2$-form on $M[2]\cong S(TM)$, which we denote by $\tilde{P}_0^L$. In particular, the restriction of $\tilde{P}_0^L$ to the boundary $\partial(M[2])$ is given by
\begin{equation}\label{eqn:extended-propagator-boundary-one}
(4\pi)^{-\frac{3}{2}}\det(g(z))^{1/2}\epsilon_{ijk}\left(\sqrt{\pi}\hat{u}^id_{vert}\hat{u}^jd_{vert}\hat{u}^k+\sqrt{\pi}\hat{u}^k\Omega_l^ig^{lj}(z)\right),
\end{equation}
where $\pi:S(TM)\rightarrow M=\Delta$ is the projection.
\end{prop}
\begin{proof}
By Lemma \ref{lemma:leading-term-K-t}, the propagator $P_0^L$ has an asymptotic expansion as $r\rightarrow 0$, and the leading term at a point $((z,\hat{u}),r)\in S(TM)\times (0,\epsilon)\subset M\times M\setminus\Delta$ (via the embedding \eqref{eqn:M[2]-embedding}) is given by
\begin{equation}\label{eqn:propagator-boundary-leading-term}
(4\pi)^{-\frac{3}{2}}\int_0^L\left(\frac{r^3}{4t^{5/2}}e^{-\frac{r^2}{4t}}\det(g(z))^{1/2}\epsilon_{ijk}\hat{u}^id_{vert}\hat{u}^jd_{vert}\hat{u}^k
+\frac{u^k}{2t^{3/2}}e^{-\frac{r^2}{4t}}\det(g(z))^{1/2}\epsilon_{ijk}\Omega_l^ig^{lj}(z)\right)dt.
\end{equation}
We compute the limit of the integral \eqref{eqn:propagator-boundary-leading-term} as $r\rightarrow 0$:
\begin{align*}
  & \lim_{r\rightarrow 0} (4\pi)^{-\frac{3}{2}}\int_0^L\left(\frac{r^3}{4t^{5/2}}e^{-\frac{r^2}{4t}}\det(g(z))^{1/2}\epsilon_{ijk}\hat{u}^id_{vert}\hat{u}^jd_{vert}\hat{u}^k + \frac{u^k}{2t^{3/2}}e^{-\frac{r^2}{4t}}\det(g(z))^{1/2}\epsilon_{ijk}\Omega_l^ig^{lj}(z)\right)dt\\
= & (4\pi)^{-\frac{3}{2}}\det(g(z))^{1/2}\epsilon_{ijk}\Bigg(\left(\lim_{r\rightarrow 0}\int_0^L\frac{r^3}{4t^{5/2}}e^{-\frac{r^2}{4t}}dt\right)\hat{u}^id_{vert}\hat{u}^jd_{vert}\hat{u}^k\\
&\qquad\qquad\qquad\qquad\qquad + \left(\lim_{r\rightarrow 0}\int_0^L\frac{r}{2t^{3/2}}e^{-\frac{r^2}{4t}}dt\right)\hat{u}^k\Omega_l^ig^{lj}(z)\Bigg)\\
= & (4\pi)^{-\frac{3}{2}}\det(g(z))^{1/2}\epsilon_{ijk}\left(\sqrt{\pi}\hat{u}^id_{vert}\hat{u}^jd_{vert}\hat{u}^k + \sqrt{\pi}\hat{u}^k\Omega_l^ig^{lj}(z)\right).
\end{align*}
It is clear that the higher order terms vanish as $r\rightarrow 0$. Thus, there exists a smooth $2$-form $\bar{\rho}$ on a neighborhood $U$ of $\Delta\subset M\times M$, such that on $U\setminus \Delta\cong S(TM)\times (0,\epsilon)$, the propagator $P_0^L|_{U\setminus\Delta}$ is the sum of $\bar{\rho}$ and the first terms of the asymptotic expansion of $P_0^L$. It then follows easily that $P_0^L$ can be extended to $S(TM)\times [0,\epsilon)$, which we call $\tilde{P}_0^L$. In particular, the restriction of $\tilde{P}_0^L$ on  $S(TM)\times \{0\}$ is the sum of equation \eqref{eqn:extended-propagator-boundary-one} and the smooth $2$-form $\pi^*(\bar{\rho}|_\Delta)$.

We claim that $\pi^*(\bar{\rho}|_\Delta)=0$.
To see this, observe that the propagator $P_0^L$ is antisymmetric with respect to the swap of the two components in $M\times M$. In particular, on $U\setminus \Delta\cong S(TM)\times (0,\epsilon)$, this swap is expressed as
\begin{equation}\label{eqn:boundary-antipodal}
((z,\hat{u}),r)\mapsto ((z,-\hat{u}),r).
\end{equation}
This antisymmetry extends to the blow up $M[2]$. On the other hand, note that the equation (\ref{eqn:extended-propagator-boundary-one}) is antisymmetric under the map \eqref{eqn:boundary-antipodal}. It follows that $\pi^*(\bar{\rho}|_\Delta)=0$.
\end{proof}

We can now use the projections $\pi_e$ to pull back $\tilde{P}_0^L$ and get a smooth $2$-form on the compactified configuration space $M[V]$. Also we can use $\pi_v$ to pull back the inputs of the Feynman graphs on the tails. Altogether, we can use the compactified configuration spaces to define the Feynman weights without worrying about the singularities of the propagators:
\begin{defn}
Let $\gamma$ be a connected stable graph without self-loops whose number of tails is $|T(\gamma)|=k$. The {\em Feynman weights} associated to $\gamma$ is defined by
$$
W_\gamma(\tilde{\mathbb{P}}_0^L, I_{cl})(\phi_1,\cdots,\phi_{k}):=\int_{M[V(\gamma)]}\prod_{e\in E(\gamma)}\pi_e^*(\tilde{\mathbb{P}}_0^L)\prod_{i=1}^k\phi_i.
$$
\end{defn}
Lemma \ref{lem:self-loop-vanishing} implies that we can define the naive quantization of our Rozansky-Witten model using the above Feynman weights:
\begin{defn}
The {\em naive quantization} of the classical interaction of our Rozansky-Witten model is defined by
\begin{equation}\label{eqn:naive-quantization}
I_{naive}[L]:=\sum_{\gamma}\frac{\hbar^{g(\gamma)}}{|\text{Aut}(\gamma)|}W_\gamma(\tilde{\mathbb{P}}_0^L,I_{cl}),
\end{equation}
where the sum is over stable connected graphs.
\end{defn}
It is clear from the construction that the effective interactions $\{I_{naive}[L]\}$ satisfy the renormalization group equation (RGE).

\subsubsection{Stratification of $M[V]$}
There is a stratification of $M[V]$, which was explained in detail in \cite{AS2}, that will be of use later. For readers' convenience, we will keep the same notations as in \cite{AS2} and only explain the part relevant for us.

The compactified configuration space $M[V]$ can be written as the disjoint union of open strata:
$$
M[V]=\bigcup_{\mathcal{S}}M(\mathcal{S})^0.
$$
Here $\mathcal{S}$ denotes a collection of subsets  of the index set $V$ such that
\begin{enumerate}
\item $\mathcal{S}$ is nested: if  $S_1,S_2$ belong to $\mathcal{S}$, then either they are disjoint or one contains the other,
\item Every subset $S \in \mathcal{S}$ is of cardinality $\geq 2$.
\end{enumerate}

A useful fact is that the open stratum $M(\mathcal{S})^0$ is of codimension $|\mathcal{S}|$. For later consideration of the quantum master equation, we will only need those codimension $1$ strata, which correspond to collections $\mathcal{S}=\{S\}$ consisting of a {\em single} subset $S\subset V$ with cardinality $\geq 2$. Without loss of generality, we can assume that $V=\{1,\cdots, n\}$ and $S=\{1,\cdots, k\}$ with $k\leq n$. We will denote such a codimension $1$ open stratum simply by $M(S)^0$, and it can be described explicitly as follows: Let $E:=S(N(\Delta_S\subset M^S))$ denote the sphere bundle of the normal bundle of the small diagonal $\Delta_S\subset M^S$. Then there is a homeomorphism
\begin{equation}\label{eqn:codim-1-strata}
M(S)^0\cong \{(e,x_1,\cdots,x_{|V(\gamma)|-|S|})\in E\times M^{|V(\gamma)|-|S|}:\pi(e),x_1,\cdots, x_{|V(\gamma)|-|S|}\text{\ not\ all\ equal}\}.
\end{equation}
Here $\pi$ denotes the projection $E\rightarrow \Delta_S$. 

\subsection{Quantization}\label{subsection:QME}
We are now ready to prove that the naive quantization of our Rozansky-Witten model \eqref{eqn:naive-quantization} satisfies the quantum master equation \eqref{eqn:QME}. First of all, a straightforward calculation shows that the quantum master equation \eqref{eqn:QME} is equivalent to the following:
\begin{equation}\label{eqn:QME-equivalent}
\begin{aligned}
&\frac{1}{\hbar}\left((d_M^\vee+\nabla)I_{naive}[L]+\fbracket{I_{naive}[L],I_{naive}[L]}_L+\hbar\Delta_LI_{naive}[L]\right)e^{I_{naive}[L]/\hbar}\\
&\qquad\qquad\qquad\qquad = -\left(\nabla^2\int_0^LQ^{GF}e^{-t[Q,Q^{GF}]}dt+\frac{\rho(R)}{\hbar}\right)e^{I_{naive}[L]/\hbar}.
\end{aligned}
\end{equation}
Recall that the operator $Q$ is given by $d_M^\vee+\nabla$, where $d_M^\vee$ is induced by the de Rham differential on  $M$ and acts on the Feynman weights $W_\gamma(\tilde{\mathbb{P}}_0^L, I_{cl})$ by
$$
d^\vee_M\left(W_\gamma(\tilde{\mathbb{P}}_0^L, I_{cl})\right)(\phi_1,\cdots, \phi_k)=\sum_{i=1}^kW_\gamma(\tilde{\mathbb{P}}_0^L, I_{cl})(\phi_1,\cdots, d_M\phi_i,\cdots, \phi_k).
$$
On the other hand, we have
\begin{lem}\label{lem:differential-of-propagator}
$d_M(\tilde{\mathbb{P}}_0^L)=-\mathbb{K}_L$.
\end{lem}
\begin{proof}
Since the propagator $\tilde{\mathbb{P}}_0^L$ is defined by lifting $\mathbb{P}_0^L$ smoothly from $M\times M\setminus\Delta$ to $M[2]$, we only need to show the equality on $M\times M\setminus\Delta$. But then this is clear since $d_M(\tilde{\mathbb{P}}_0^L)=\mathbb{K}_0-\mathbb{K}_L$, and $\mathbb{K}_0$ has support only on $\Delta$.
\end{proof}
\begin{thm}\label{theorem: QME}
For any fixed Riemannian metric $g$ on $M$, the naive quantization $\{I_{naive}[L]\}_{L>0}$ satisfies the quantum master equation \eqref{eqn:QME}.
\end{thm}
\begin{proof}
By the Leibniz rule, we have
\begin{equation}\label{eqn:QME-stokes}
\begin{aligned}
&d^\vee_M\left(W_\gamma(\tilde{\mathbb{P}}_0^L, I_{cl})\right)(\phi_1,\cdots,\phi_k)\\
=\ &\sum_{i=1}^k W_\gamma(\tilde{\mathbb{P}}_0^L, I_{cl})(\phi_1,\cdots,d_M(\phi_i),\cdots,\phi_k)  \\
=\ &\int_{M[V(\gamma)]}d\left(\prod_{e\in E(\gamma)}\pi_e^*(\tilde{\mathbb{P}}_0^L)\prod_{i=1}^k\phi_i\right)-\int_{M[V(\gamma)]}\sum_{e_0\in E(\gamma)}d(\pi_{e_0}^*(\tilde{\mathbb{P}}_0^L))\prod_{e\in E(\gamma)\setminus e_0}\pi_e^*(\tilde{\mathbb{P}}_0^L)\prod_{i=1}^k\phi_i   \\
\overset{(1)}{=}\ &\int_{\partial M[V(\gamma)]}\prod_{e\in E(\gamma)}\pi_e^*(\tilde{\mathbb{P}}_0^L)\prod_{i=1}^k\phi_i-\int_{M[V(\gamma)]}\sum_{e_0\in E(\gamma)}\pi_{e_0}^*(\mathbb{K}_L)\prod_{e\in E(\gamma)\setminus e_0}\pi_e^*(\tilde{\mathbb{P}}_0^L)\prod_{i=1}^k\phi_i,
\end{aligned}
\end{equation}
where we have used Stokes' theorem and Lemma \ref{lem:differential-of-propagator} in the last equality (1).

The edge $e_0$ in the RHS of the above equation  could be separating or not, namely, deleting the edge $e_0$ would result in two connected graphs or one as shown in the following picture:
\begin{align*}
&\text{separating:}\hspace{27mm} \figbox{0.2}{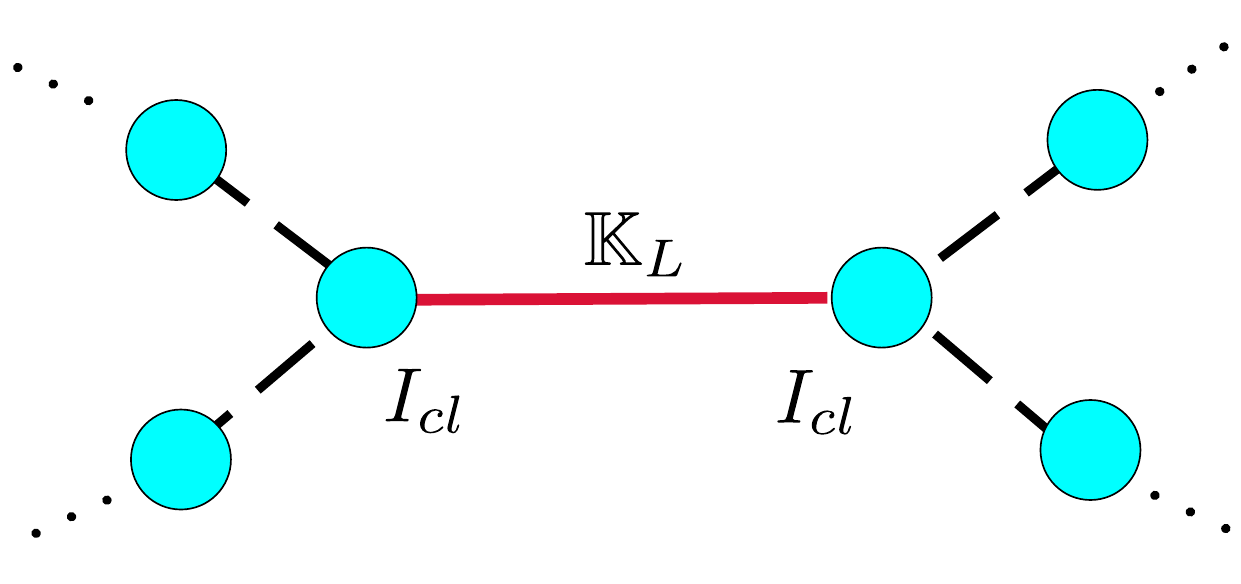}\\
&\text{non-separating:}\hspace{20mm} \figbox{0.2}{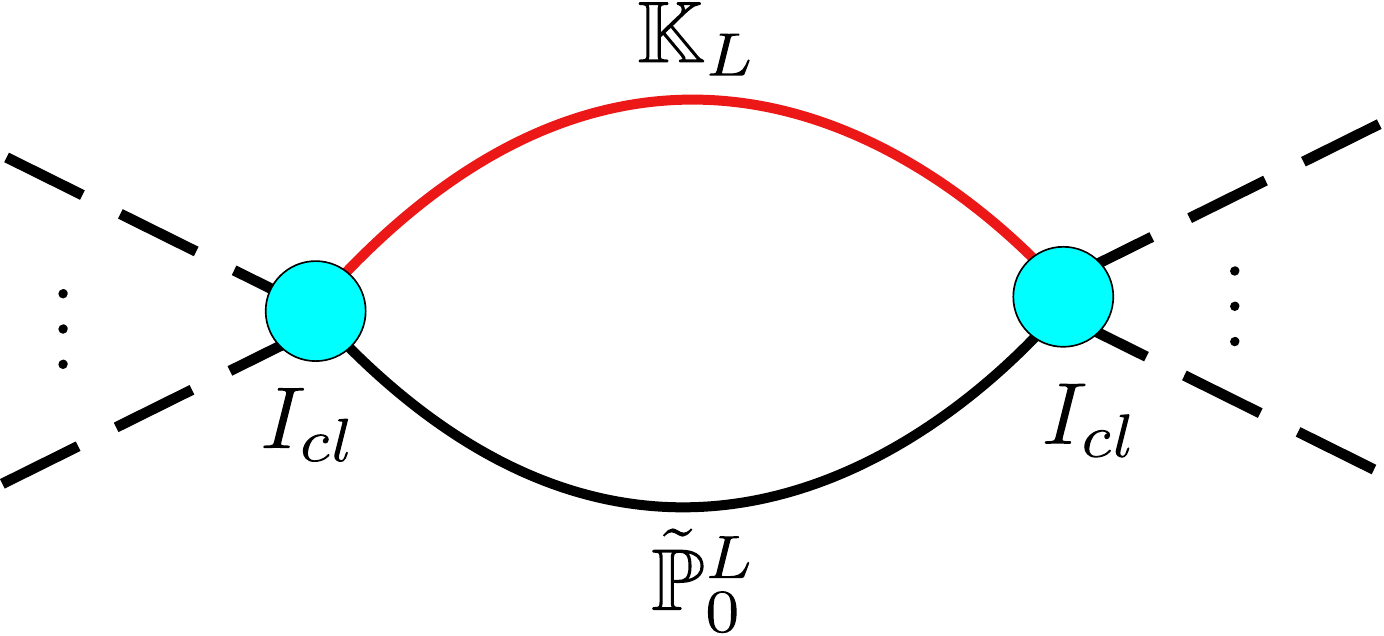}
\end{align*}
It is not difficult to see that the sum of all the terms corresponding to non-separating edges over all stable graphs $\gamma$ will cancel with the term $\hbar\Delta_LI_{naive}[L]$ in the quantum master equation, and the sum of all the other terms corresponding to separating edges will cancel with those terms in $\{I_{naive}[L],I_{naive}[L]\}_L$, except the following ones, containing a vertex labeled by $\tilde{l}_0$ with $k>1$ (since the terms $\tilde{l}_0$ is not involved in the RG flow by type reason):
\begin{equation}\label{eqn:separating-edges-not-canceled}
 \figbox{0.23}{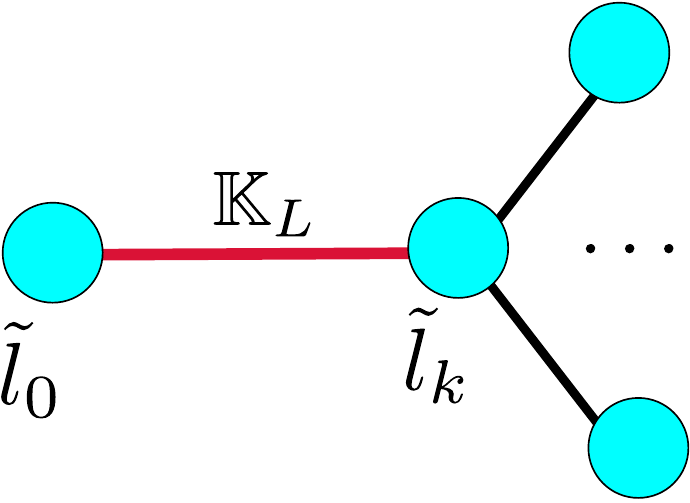}
\end{equation}
Let us now calculate the first term of the RHS of equation \eqref{eqn:QME-stokes}, i.e. an integral over the boundary components of the compactified configuration spaces.

Since the integrands are smooth forms, we can replace $\partial M[V(\gamma)]$ by the union of open codimension 1 strata  of $M[V]$. As explained in the previous subsection, such strata are in one-to-one correspondence with subsets $S\subset V$ with $|S|\geq 2$. We may assume that the vertices in $S$ are connected in the graph $\gamma$, since otherwise the integral vanishes by type reasons.

We first consider the cases where $|S|\geq 3$. There are then two graphs $\gamma',\gamma''$ associated to $\gamma$ and $S$: let $\gamma'$ denote a graph consisting of the following data:
\begin{enumerate}
 \item The vertices of $\gamma'$ are given by $V(\gamma')=S$;
 \item The internal edges consist of all edges of $\gamma$ that are only incident to vertices in $S$;
 \item The tails of $\gamma'$ are those half edges of $\gamma$ incident to $S$ but not part of the internal edges in $\gamma'$.
\end{enumerate}
And we define
$$
I_{\gamma'}:=W_{\gamma'}(\tilde{\mathbb{P}}_0^L,I_{cl}).
$$
We then define the other graph $\gamma''$ as the ``remaining part'' of $\gamma'$ in $\gamma$, as shown in the following picture: in the left picture, the yellow circles denote those vertices in $S$, and the subgraph in the dashed circle is $\gamma'$. In the right picture, we replace the subgraph $\gamma'$ by a single yellow vertex labeled by $I_{\gamma'}$, and the two Feynman weights are identified:
$$
\figbox{0.22}{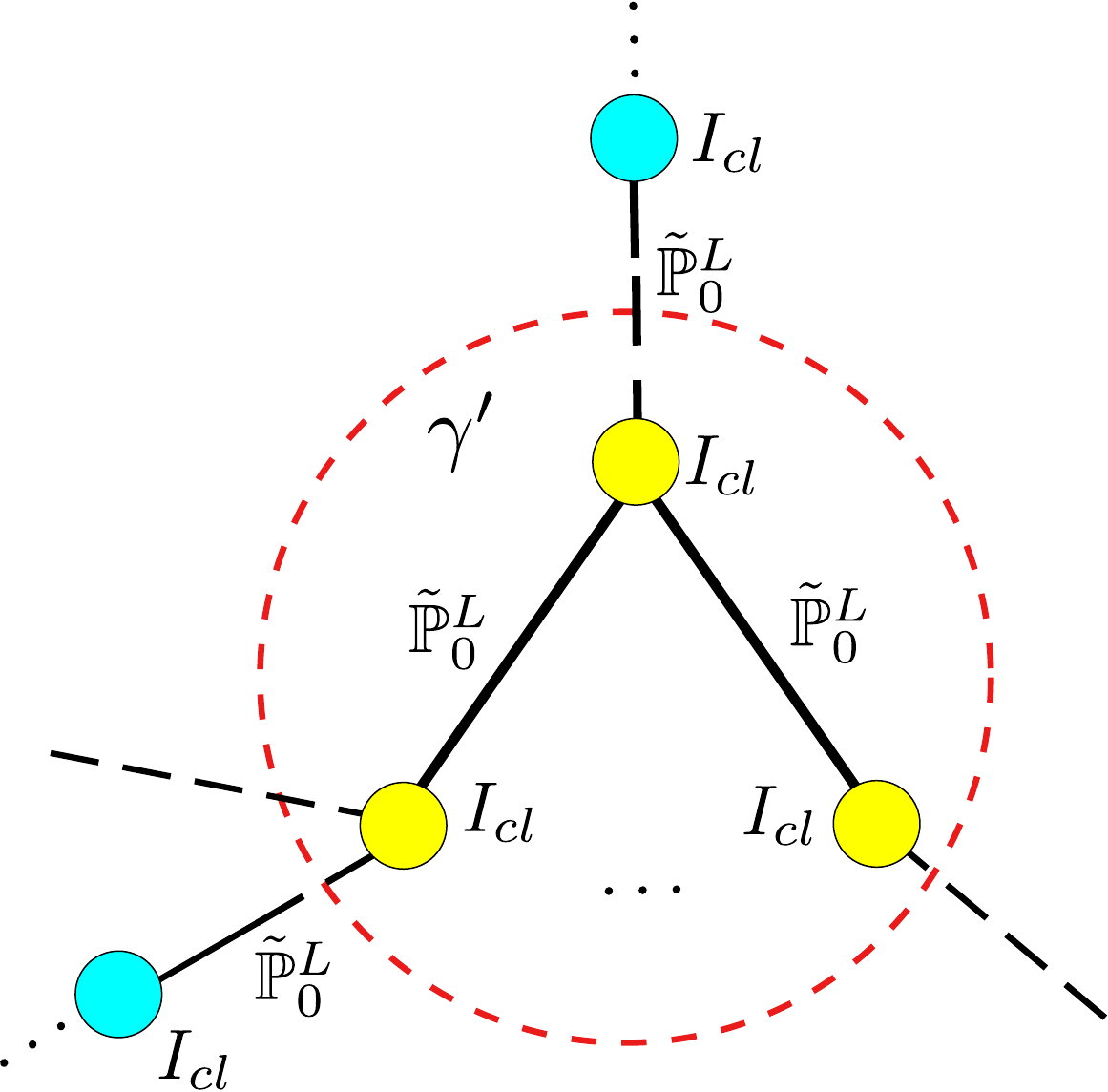}\hspace{8mm}=\hspace{5mm}\figbox{0.22}{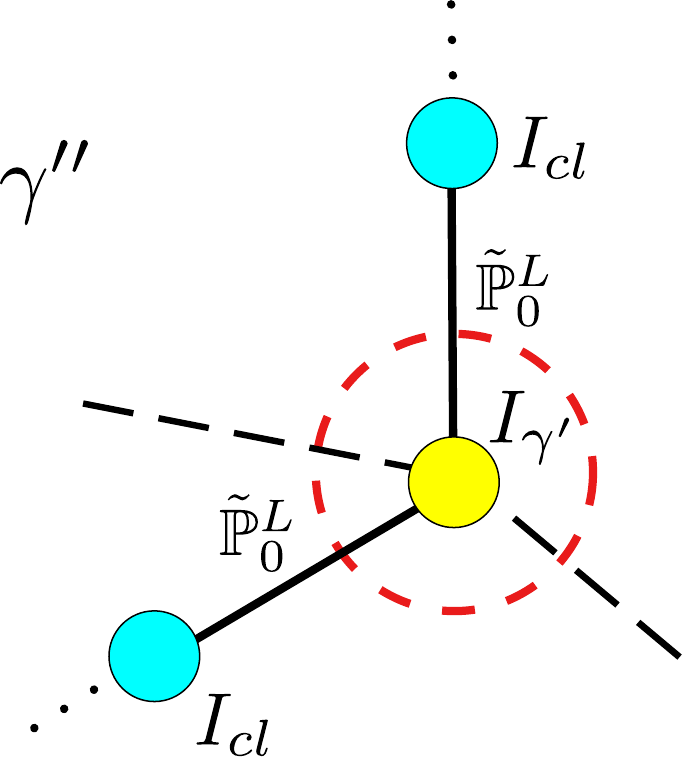}
$$

Let $M_\gamma(S)^0$ denote the codimension-$1$ stratum in the configuration space associated to the graph $\gamma$, and let $M_{\gamma'}(S)^0$ denote that associated to the graph $\gamma'$. Let $W_{\gamma,S}(\tilde{\mathbb{P}}_0^L, I_{cl})$ be the functional obtained by the Feynman integral over the boundary stratum corresponding to the subset $S$ in the graph $\gamma$. Then the following lemma follows easily from the explicit description of codimension $1$ open strata in equation \eqref{eqn:codim-1-strata}:
\begin{lem}\label{lem:boundary-strata-integral-subgraph}
We have
$$
W_{\gamma,S}(\tilde{\mathbb{P}}_0^L, I_{cl})=W_{\gamma''}\left(\tilde{\mathbb{P}}_0^L, I_{cl}, W_{\gamma',S}(\tilde{\mathbb{P}}_0^L,I_{cl})\right).
$$
\end{lem}
\begin{rmk}
This is equivalent to the computation in \cite{AS2} where a graph integral is divided into the ``regular'' and ``singular'' parts.
\end{rmk}
\begin{notn}
Let $\gamma$ be as before. We consider the boundary component of $M[V(\gamma)]$ corresponding to $S=V(\gamma)$, i.e. when all the points collide. We call the integral on this boundary component $primitive$.
\end{notn}
There are the following observations about the $primitive$ boundary integrals:
\begin{enumerate}
 \item The $primitive$ boundary integrals are independent of $L$, and give rise to local functionals on $\E$;
 \item Lemma \ref{lem:boundary-strata-integral-subgraph} is saying that the sum of all boundary integrals are exactly the RG flow of the primitive boundary integrals.
\end{enumerate}
\begin{rmk}
These primitive boundary integrals are, roughly speaking, corresponding to the ``scale-0 obstructions''  to quantizations, which are local functionals on $\E$.
\end{rmk}

Then we have the following
\begin{lem}\label{lem:Kontsevich-vanishing}
Let $S$ be a subset of $V(\gamma)$ with $|S|\geq 3$, then for any inputs $\phi_1,\cdots,\phi_k$, we have the vanishing of the following boundary integral:
\begin{equation}\label{eqn:QME-integral-boundary}
\int_{M_\gamma(S)^0}\prod_{e\in E(\gamma)}\pi_e^*(\tilde{\mathbb{P}}_0^L)\prod_{i=1}^k\phi_i=0.
\end{equation}
\end{lem}
\begin{proof}
By Lemma \ref{lem:boundary-strata-integral-subgraph}, we can assume that $S=V(\gamma)$.  Let $T_{z}M$ denote the tangent space at $z\in M$, and let $N=|S|$. The codimension $1$ stratum corresponding to $S=V(\gamma)$ is a fiber bundle over the small diagonal in $M^{V(\gamma)}$, and let us denote by $\pi^{-1}(z)$ the fiber over a point $z\in M$. Although the inputs $\phi_i$'s are put on separate copies of $M$ corresponding to the vertices they are labeling, when restricted to $M_\gamma(S)^0$, their product becomes a differential form on a single copy of $M$, i.e. the small diagonal.

Thus, to show that the integral \eqref{eqn:QME-integral-boundary} vanishes, it suffices to show that the push-forward of the forms $\prod_{e\in E(\gamma)}\pi_e^*(\tilde{\mathbb{P}}_0^L)$ vanish. We fix a local trivialization around $z\in M$ and look at the following integral along the fiber over $z$:
\begin{equation}\label{eqn:integral-boundary-fiber}
\int_{\pi^{-1}(z)}\prod_{e\in E(\gamma)}\pi_e^*(\tilde{\mathbb{P}}_0^L).
\end{equation}
Recall that the extended propagator $\tilde{\mathbb{P}}_0^L$ on the boundary of $M[2]$ is given in equation \eqref{eqn:extended-propagator-boundary-one}, which contains the Christoffel symbols of the Levi-Civita connection on $M$. By the standard trick of Riemannian geometry,  we can choose a local coordinate system such that the Christoffel symbols vanish at $z$. Then the proof of the claim resembles Kontsevich's original argument \cite{Kontsevich}, whose details we reproduce here for the readers' convenience.

Since we are only considering the integral over the fiber of $z\in M$ and we have chosen the local coordinates so that the Christoffel symbols vanish at $z$, the restriction of the extended propagator $\tilde{P}_0^L$ on the fiber $\pi^{-1}(z)$ is given by
\begin{equation}\label{eqn:propagator-boundary}
(4\pi)^{-\frac{3}{2}}\det(g(z))^{1/2}\epsilon_{ijk}\left(\sqrt{\pi}\hat{u}^id\hat{u}^jd\hat{u}^k+\sqrt{\pi}\hat{u}^k\Omega_l^ig^{lj}(z)\right).
\end{equation}
The fiber $\pi^{-1}(z)$ is of dimension $3N-4$ and \eqref{eqn:propagator-boundary} is a $2$-form. It is clear that equation \eqref{eqn:propagator-boundary} can either contribute the curvature forms on the base $M$, or a $2$-form on the fiber $\pi^{-1}(z)$. Hence the integral can be nontrivial only when $N=2k$ is even. There can be at most one edge $e$ such that  $\pi_e^*(\tilde{\mathbb{P}}_0^L)$ contributes a curvature form on $M$ since $M$ is $3$-dimensional, and we will call $e$ the distinguished edge.

We first show the vanishing of the integral \eqref{eqn:integral-boundary-fiber} for $k\geq 2$. If there is no distinguished edge, then the integral \eqref{eqn:integral-boundary-fiber} vanishes by \cite{Kontsevich}*{Lemma 2.1}. Otherwise, we need a slight modification of Kontsevich's argument. We note that the number of internal edges must be $3k-3$, one of which is the distinguished edge and others contribute the following terms:
$$
\frac{1}{4\pi}\det(g(z))^{1/2}\epsilon_{ijk}\hat{u}^id\hat{u}^jd\hat{u}^k
$$
in the integral \eqref{eqn:integral-boundary-fiber}. Let us first consider the constraint from the non-distinguished edges. By the type reason, the integrand vanishes if at least one vertex in the graph is less than bivalent to such edges. Thus in the worst cases, there exist $4$ vertices which are bivalent to these edges. Even if we add in the distinguished edge, there still remain two such vertices in the worst situation. (The following graph shows the case when $k=2$: the green edge denote the distinguished one and the yellow vertices are bivalent.)
$$
\figbox{0.2}{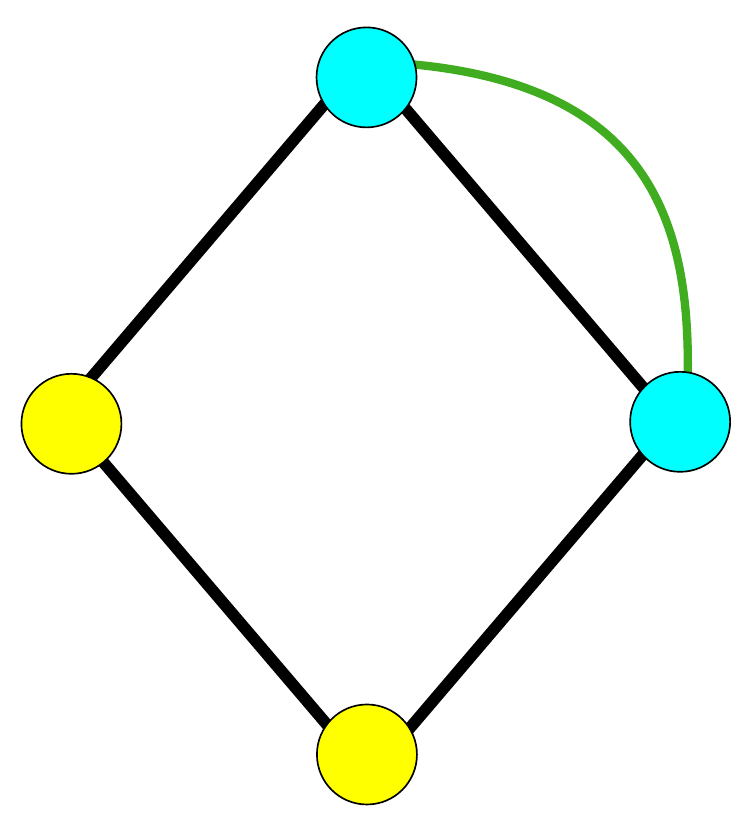}
$$
Then the argument in \cite[Lemma 2.1]{Kontsevich} still works. This completes the proof of the lemma.
\end{proof}

Hence it remains to deal with the cases when $|S|=2$. There are two possibilities:
\begin{equation}\label{eqn:boundary-strata-two-vertex}
\figbox{0.22}{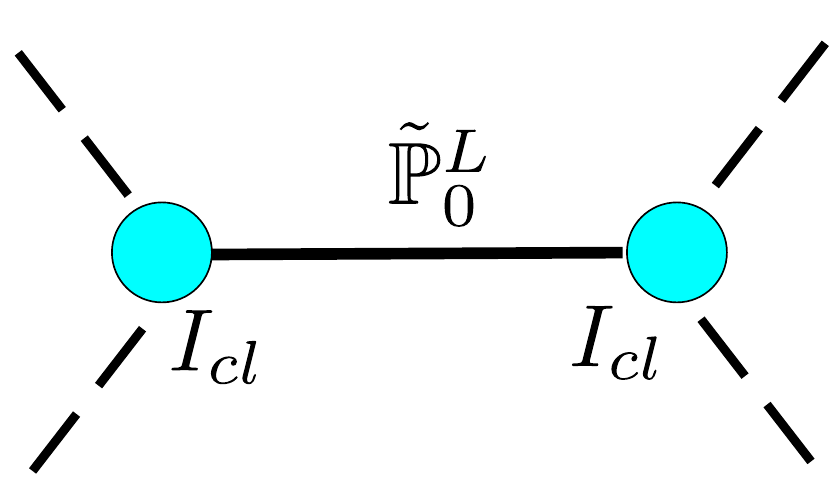}\hspace{10mm}\figbox{0.22}{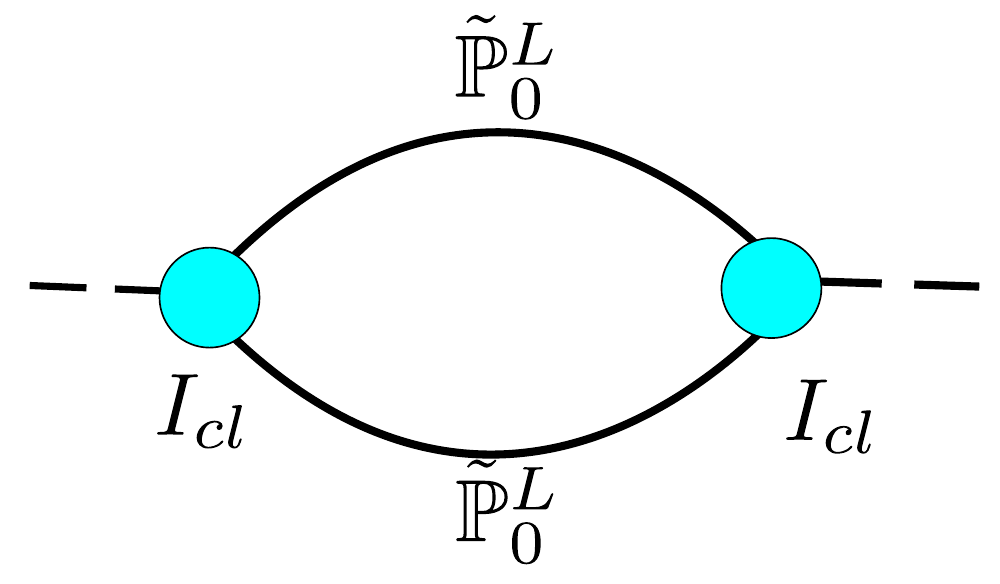}
\end{equation}
For the right picture in equation (\ref{eqn:boundary-strata-two-vertex}), the vanishing of the corresponding functionals can be verified by direct computation. For more details, we refer to \cite{Bott-Cattaneo}.

For the left picture in equation (\ref{eqn:boundary-strata-two-vertex}), we decompose it as the sum of the following functionals on $\E$, which we denote by $I_{m,n}$:
$$
I_{m,n}(\alpha):=\langle l_m(\alpha,\cdots,l_n(\alpha,\cdots,\alpha),\cdots,\alpha), \alpha\rangle, \qquad \alpha\in\E, m,n\geq 1.
$$
Since the QME is independent of the scale, we will now finish the computation for $L=\infty$.
\begin{lem}
Let $\gamma_{m,n}$ be the following graph with two vertices, such that $m,n>1$:
\begin{equation}\label{eqn:boundary-strata-I-mn}
\figbox{0.22}{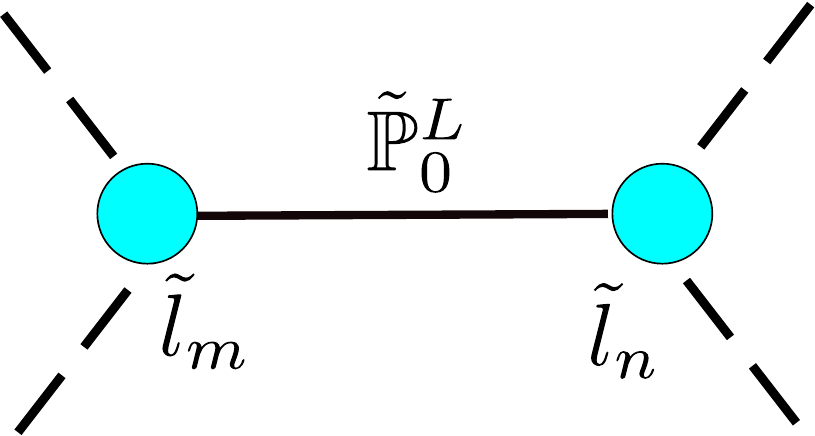}
\end{equation}
The functional on $\E$ given by the integral on its (unique) codimension-$1$ stratum is the same as the local functional $I_{m,n}$.
\end{lem}
\begin{proof}
We compute the analytic and combinatorial parts of the functional in the lemma respectively. For the analytic part, let $\phi_1,\cdots,\phi_{m+n}$ denote the inputs on the tails. The integral on the boundary stratum of the graph $\gamma_{m,n}$ is
$$
\int_{\partial M[2]}\pi^*(\phi_1\cdots\phi_{m+n})\tilde{P}_0^\infty,
$$
where $\pi: \partial M[2]\rightarrow M$ is the natural projection. For the integration of $\tilde{P}_0^\infty$ over the fibers , the only term that will contribute is $\frac{1}{4\pi}\det(g(z))^{1/2}\epsilon_{ijk}\hat{u}^id\hat{u}^jd\hat{u}^k$, and the integral on the fiber over $z$ is:
\begin{align*}
\int_{\pi^{-1}(z)}\frac{1}{4\pi}\det(g(z))^{1/2}\epsilon_{ijk}\hat{u}^id\hat{u}^jd\hat{u}^k=1.
\end{align*}
Thus we have
$$
\int_{\partial M[2]}\pi^*(\phi_1\cdots\phi_{m+n})\tilde{P}_0^\infty=\int_M \phi_1\cdots\phi_{m+n}.
$$
For the combinatorial part, it is not difficult to see that it is the same as that of $I_{m,n}$.

\end{proof}
Thus we have shown that
\begin{equation}\label{eqn:d_M-QME}
\begin{aligned}
d_M^\vee(e^{I_{naive}[\infty]/\hbar})&=\frac{1}{\hbar}\left(d_M^\vee(I_{naive}[\infty])\right)e^{I_{naive}[\infty]/\hbar}\\
&=\sum_{m,n>1}e^{\hbar\partial_{\tilde{\mathbb{P}}_0^\infty}}(\frac{I_{m,n}}{\hbar}\cdot e^{I_{cl}/\hbar}).
\end{aligned}
\end{equation}

On the other hand, since $\nabla$ is compatible with the symplectic form $\omega$. We obtain
\begin{equation}\label{eqn:l_1_QME}
\begin{aligned}
\nabla(e^{I_{naive}[\infty]/\hbar})&=\nabla\left(e^{\hbar\partial_{\tilde{\mathbb{P}}_0^\infty}}e^{I_{cl}/\hbar}\right)\\\
&=e^{\hbar\partial_{\tilde{\mathbb{P}}_0^\infty}}\left(\nabla e^{I_{cl}/\hbar}\right)\\
&=e^{\hbar\partial_{\tilde{\mathbb{P}}_0^\infty}}\left((\nabla \frac{I_{cl}}{\hbar}) e^{I_{cl}/\hbar}\right)\\
&=e^{\hbar\partial_{\tilde{\mathbb{P}}_0^\infty}}\left(\sum_{n\not=1}\frac{I_{1,n}}{\hbar} e^{I_{cl}/\hbar}\right).
\end{aligned}
\end{equation}
Thus we have shown that
$$
\frac{1}{\hbar}\left((d_M^\vee+\nabla)I_{naive}[\infty]+\fbracket{I_{naive}[\infty],I_{naive}[\infty]}_\infty+\hbar\Delta_\infty I_{naive}[\infty]\right)e^{I_{naive}[\infty]/\hbar}
$$
is given by the sum of the following terms:
\begin{enumerate}
 \item The last line of equation \eqref{eqn:d_M-QME}
 \item The last line of equation \eqref{eqn:l_1_QME}
 \item $\sum_{m>1}e^{\hbar\partial_{\tilde{\mathbb{P}}_0^\infty}}(\frac{I_{0,m}}{\hbar}\cdot e^{I_{cl}/\hbar})$. These are the terms in equation \eqref{eqn:separating-edges-not-canceled}.
\end{enumerate}
The classical master equation \eqref{eqn:classical-master-equation} says that the sum of
the above terms is $-e^{\hbar\partial_{\tilde{\mathbb{P}}_0^\infty}}\left(\frac{1}{\hbar}\rho(R) e^{I_{cl}/\hbar}\right)$. To conclude, we have
$$
\frac{1}{\hbar}\left((d_M^\vee+\nabla)I_{naive}[\infty]+\fbracket{I_{naive}[\infty],I_{naive}[\infty]}_\infty+\hbar\Delta_\infty I_{naive}[\infty]\right)e^{I_{naive}[\infty]/\hbar}=-e^{\hbar\partial_{\tilde{\mathbb{P}}_0^\infty}}\left(\frac{1}{\hbar}\rho(R) e^{I_{cl}/\hbar}\right).
$$
The following lemma then finishes the proof of the theorem.
\end{proof}
\begin{lem}
$ \left(\nabla^2\int_0^LQ^{GF}e^{-t[Q,Q^{GF}]}dt+\frac{\rho(R)}{\hbar}\right)e^{I_{naive}[L]/\hbar}=e^{\hbar\partial_{\tilde{\mathbb{P}}_0^L}}\left(\frac{1}{\hbar}\rho(R) e^{I_{cl}/\hbar}\right)$.
\end{lem}
\begin{proof}
There is the following compatibility between the operator $Q_L+\hbar\Delta_L+\frac{\rho(R)}{\hbar}$ and the RG flow: (see Lemma 3.13 in \cite{Li-Li}):
\begin{equation}\label{eqn:compatibility-RG-flow-quantum-differential}
\left(Q_L+\hbar\Delta_L+\frac{\rho(R)}{\hbar}\right)e^{\hbar\partial_{\tilde{\mathbb{P}}_\epsilon^L}}=e^{\hbar\partial_{\tilde{\mathbb{P}}_\epsilon^L}}\left(Q_\epsilon+\hbar\Delta_\epsilon+\frac{\rho(R)}{\hbar}\right),
\end{equation}
which is clearly equivalent to
$$
\left(\nabla+\nabla^2\int_0^L Q^{GF}e^{-t[Q,Q^{GF}]}dt+\frac{\rho(R)}{\hbar}\right)e^{\hbar\partial_{\tilde{\mathbb{P}}_\epsilon^L}}=e^{\hbar\partial_{\tilde{\mathbb{P}}_\epsilon^L}}\left(\nabla+\nabla^2\int_0^\epsilon Q^{GF}e^{-t[Q,Q^{GF}]}dt+\frac{\rho(R)}{\hbar}\right).
$$
Furthermore, the compatibility between $\nabla$ and the RG flow implies the following:
$$
\left(\nabla^2\int_0^L Q^{GF}e^{-t[Q,Q^{GF}]}dt+\frac{\rho(R)}{\hbar}\right)e^{\hbar\partial_{\tilde{\mathbb{P}}_\epsilon^L}}=e^{\hbar\partial_{\tilde{\mathbb{P}}_\epsilon^L}}\left(\nabla^2\int_0^\epsilon Q^{GF}e^{-t[Q,Q^{GF}]}dt+\frac{\rho(R)}{\hbar}\right).
$$
Thus we have
\begin{align*}
 &\left(\nabla^2\int_0^LQ^{GF}e^{-t[Q,Q^{GF}]}dt+\frac{\rho(R)}{\hbar}\right)e^{I_{naive}[L]/\hbar}\\
=&\lim_{\epsilon\rightarrow 0}\left(\nabla^2\int_0^LQ^{GF}e^{-t[Q,Q^{GF}]}dt+\frac{\rho(R)}{\hbar}\right)e^{\hbar\partial_{\tilde{\mathbb{P}}_\epsilon^L}}e^{I_{cl}/\hbar}\\
=&\lim_{\epsilon\rightarrow 0}e^{\hbar\partial_{\tilde{\mathbb{P}}_\epsilon^L}}\left(\nabla^2\int_0^\epsilon Q^{GF}e^{-t[Q,Q^{GF}]}dt+\frac{\rho(R)}{\hbar}\right)e^{I_{cl}/\hbar}\\
=&\lim_{\epsilon\rightarrow 0}e^{\hbar\partial_{\tilde{\mathbb{P}}_\epsilon^L}}\left(\figbox{0.26}{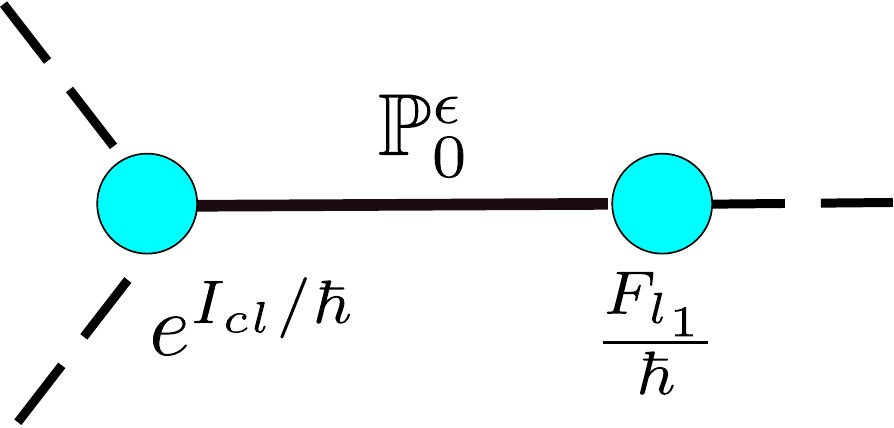}\hspace{2mm}+\hspace{2mm}\figbox{0.26}{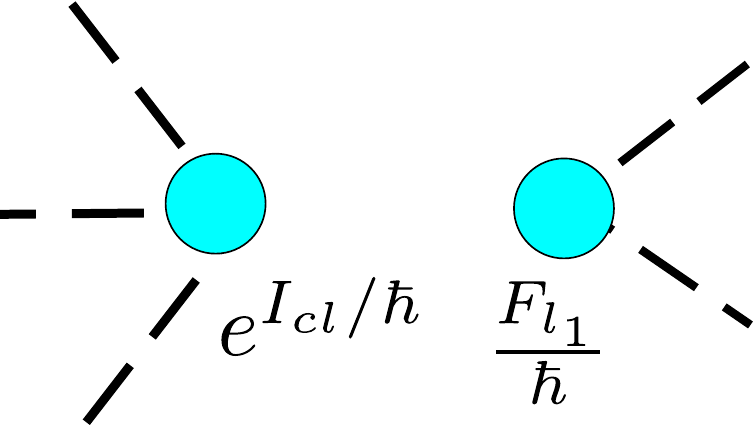}\right).
\end{align*}
The term of the left picture vanishes as $\epsilon\rightarrow 0$, and the statement of the lemma follows.
\end{proof}
\begin{notn}
Since the naive quantization
$$
\sum_{\gamma}\frac{\hbar^{g(\gamma)}}{|\text{Aut}(\gamma)|}W(\tilde{\mathbb{P}}_0^L,I_{cl})
$$
satisfies the QME, we will denote it by $I[L]$ instead of $I_{naive}[L]$.
\end{notn}

\subsection{Quantum master equation with varying Riemannian metrics}\label{section: extended-QME}
In this subsection, we consider the case when there is a family of Riemannian metrics $\{g_t\}$ on $M$ parametrized by the interval $I:=[0,1]$ and construct a solution to the QME in this situation. In this case, the propagator is a $2$-form on $M\times [0,1]$ which  depends on the metrics. Similar to Proposition \ref{prop:lift-propagator-boundary}, we have the following description of the propagators:
\begin{prop}
The following limit
$$
\lim_{\epsilon\rightarrow 0}P_\epsilon^L
$$
on $(M\times M\setminus\Delta)\times [0,1]$ can be lifted to a smooth $2$-form on $M[2]\times [0,1]$, which we also denote by $\tilde{P}_0^L$. In particular, the restriction of $\tilde{P}_0^L$ to  $\partial(M[2])\times [0,1]$ is given by
\begin{equation}\label{eqn:extended-propagator-boundary}
(4\pi)^{-\frac{3}{2}}\det(g_t(z))^{1/2}\epsilon_{ijk}\left(\sqrt{\pi}\hat{u}^id_{vert}\hat{u}^jd_{vert}\hat{u}^k+\sqrt{\pi}\hat{u}^k\tilde{\Omega}_l^ig_t^{lj}(z)\right)
\end{equation}
\end{prop}

Let us fix a trivialization of the tangent bundle $TM$, i.e. a framing of the 3-dimensional manifold $M$. Let $\theta_t$ denote the Levi-Civita connection $1$-forms on $M$, with respect to the chosen trivialization,  corresponding to the metric $g_t$. According to Proposition \ref{prop:weyl-bundle-quasi-isomorphic-Dolbeault},  for any cohomology class
$$
\alpha\in H^*(X,\OO_X),
$$
there exists a corresponding flat section $\tilde{\alpha}$ of the Weyl bundle $\mathcal{W}$ under the differential $D$. By extending $\tilde{\alpha}$ linearly over $\A_M$ and integration over $M$, we obtain the following multi-linear functional  on $\A_M^0\otimes\g_X$
$$
CS_{\tilde{\alpha}}:=\int_M\Tr(\theta_t\wedge d\theta_t+\frac{2}{3}\theta_t^3)\cdot\tilde{\alpha}.
$$

We have the following theorem:
\begin{thm}\label{theorem: QME2}
Let $g_t$ be a family of Riemannian metrics on $M$ parametrized by $[0,1]$. There exist cohomology classes $\alpha_k\in H^*(X,\OO_X)$ for $k\geq 2$, such that the effective functionals
$$
I[L]:=\sum_{\gamma}\frac{\hbar^{g(\gamma)}}{|Aut(\gamma)|}W_\gamma\left(\tilde{\mathbb{P}}_0^L, I_{cl}+\sum_{k\geq 2}\hbar^kCS_{\tilde{\alpha}_k}\right)
$$
satisfy the following quantum master equation
\begin{equation}\label{eqn:extended-QME}
 Q_LI[L]+\fbracket{I[L],I[L]}_L+\hbar\Delta_LI[L]=0.
\end{equation}
Here the operator $Q_L$ includes the de Rham differential on $\A^*([0,1])$.
\end{thm}

\begin{proof}
Similar to the proof of Theorem \ref{theorem: QME}, we are going to apply  Stokes Theorem and try to show the vanishing of some integrals on the boundary of configuration spaces. To see the failure of the vanishing of
equation (\ref{eqn:extended-QME}) without the local quantum corrections $CS_{\alpha_k}$'s, note that unlike the proof of QME for a fixed metric in Theorem \ref{theorem: QME}, there could be the following graph, whose boundary integral is not vanishing. This spoils the vanishing argument in Lemma \ref{lem:Kontsevich-vanishing}. (Note that in the case when the metric is fixed, the corresponding boundary integral vanishes by the type reason.)
\begin{equation}\label{eqn:two-loop-anomaly}
 \figbox{0.23}{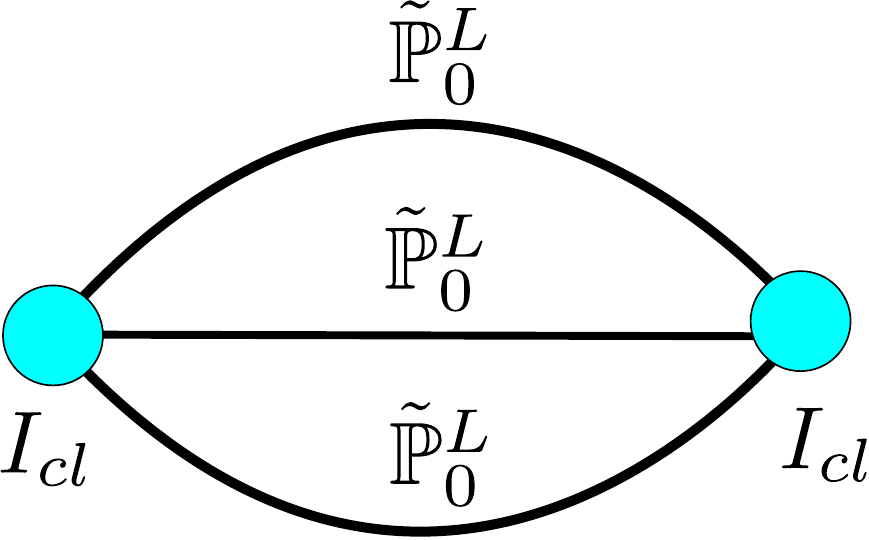}
\end{equation}
The (unique) boundary integral associated to the above picture is given explicitly by a multiple of the following integral:
$$
\int_{M\times [0,1]}\Tr(\tilde{\Omega}\wedge\tilde{\Omega}).
$$
And the combinatorial part of the graph (\ref{eqn:two-loop-anomaly}) contributes a cohomology class
$$
\alpha_2=\Tr(R^{1,1}\wedge R^{1,1})\lrcorner\omega^{-1}\in H^2(X,\OO_X).
$$
Notice that the following graphs with tails contribute  the higher order terms in the flat section $\tilde{\alpha}_2$:
\begin{equation}\label{eqn:2-loop-jets}
\figbox{0.23}{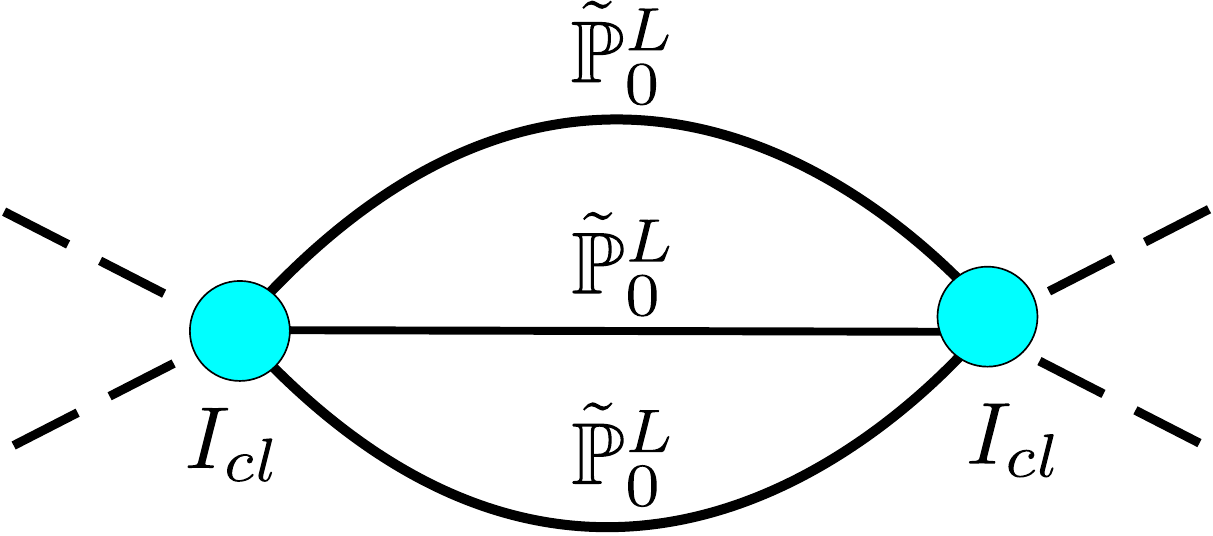}
\end{equation}
Thus, the  boundary integrals corresponding to graphs (\ref{eqn:two-loop-anomaly}) and (\ref{eqn:2-loop-jets}) together  contribute the local gauge-fixing anomaly:
\begin{equation}\label{eqn:2-loop-anomaly}
\int_{M\times [0,1]}\Tr(\tilde{\Omega}\wedge\tilde{\Omega})\cdot\tilde{\alpha}_2
\end{equation}
with inputs in $\in\A_M^0\otimes\g_X$.

By a simple counting argument, it is not difficult to see that the sum of the boundary integrals in equation (\ref{eqn:two-loop-anomaly}) and \eqref{eqn:2-loop-jets} is the anomaly at $2$-loop. This anomaly can be canceled by adding the  Chern-Simons action functional $CS_{\tilde{\alpha}_2}(\theta)$. It follows that the following functional
$$
I_{cl}+\hbar^2CS_{\tilde{\alpha}_2}(\theta)
$$
satisfies the QME (\ref{eqn:extended-QME}) modulo $\hbar^3$. To find the  anomaly of the next order, we define the following effective action functionals:
$$
I_2[L]:=\sum_{\gamma}\frac{\hbar^{g(\gamma)}}{|Aut(\gamma)|}W_\gamma(\tilde{\mathbb{P}}_0^L,I_{cl}+\hbar^2CS_{\tilde{\alpha}_2}(\theta)).
$$
There are two types of Feynman weights in the above expression: either every vertex is labeled by the classical interaction, or at least one vertex is labeled by $CS_{\alpha_2}(\theta)$, as shown in the following picture:
$$
\figbox{0.23}{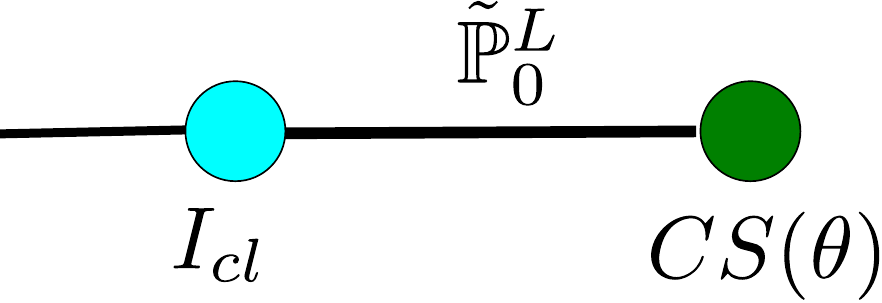}
$$
By the same argument, the next anomaly is contributed by certain integral on $\partial(M[V])\times [0,1]$. We claim that the second type of Feynman weights does not contribute to such integrals: Since the Chern-Simons term labeling the right vertex has only $0$-form input, by the type reason, the propagator must contribute a $2$-form to the copy of $M$ corresponding to the left vertex labeled by $I_{cl}$. This vertex either has at least one input of $1$-form on $M$ or connected by another propagator. In the former case, the integral vanishes by type reason since the term $CS_{\alpha_2}(\theta)$ already contributes a $3$-form on the base manifold $\Delta=M$. In the latter case, the vanishing of the corresponding integral on the boundary components can be shown in the same way as in Theorem \ref{theorem: QME}.  The Feynman weights of the first type contributes to an anomaly similar to equation (\ref{eqn:2-loop-anomaly}). For instance, let $\gamma$ denote the following graph with $4$-vertices, there is the possible boundary integral  corresponding to the situation where all four vertices collide. And the corresponding anomaly is similar to that in equation \eqref{eqn:2-loop-anomaly}.
\begin{equation}
 \figbox{0.36}{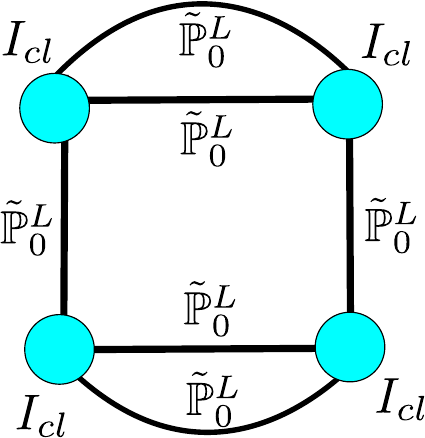}
\end{equation}

For higher loop graphs, similar graph integrals show up, and the previous argument for $2$-loop graphs can be repeated.  It follows that  the RG flow of the following local functional satisfies the QME (\ref{eqn:extended-QME}). In particular this functional depends on a choice of a framing on the 3-dimensional manifold $M$:
\begin{equation}\label{eqn:interaction-choice-of-framing}
I_{fr}:=I_{cl}+\sum_{k\geq 2}CS_{\tilde{\alpha}_k}(\theta)\hbar^k.
\end{equation}
\end{proof}
\begin{rmk}
A simple counting shows that for a fixed holomorphic symplectic manifold $X$, the number of correction terms $CS_{\alpha_k}(\theta)$ must be finite, since every vertex labeled by $I_{cl}$ must contribute a $(0,1)$-form on the target space $X$. For instance, when the target manifold is of complex dimension $2$, the anomaly (\ref{eqn:2-loop-anomaly}) is the only one, since the two vertices in the $\theta$-graph (\ref{eqn:two-loop-anomaly}) already contribute a $(0,2)$-form on the target $X$.
\end{rmk}

It is worthwhile to point out that, this anomaly in Rozansky-Witten model is similar to the gauge fixing anomaly in Chern-Simons theory (see \cite{AS2}), but more complicated: in both CS and RW theories, the analytic parts of the anomalies are given by the integral of the first Pontrjagin form. But the combinatorial part in CS theory is always the Killing form on the corresponding Lie algebra, while in RW theory there are more cohomology classes on $X$ that can contribute to the anomaly.

\section{Observable theory}\label{sec:observable_theory}
Let $(\E,S)$ be a classical field theory on a space-time manifold $M$, for which the interaction part of the action functional $S$ is given by $I_{cl}$, and let $\{I[L]\}_{L>0}$ be a quantization. As explained in full generality in the work of Costello-Gwilliam \cite{Kevin-Owen}, both the classical and quantum observables of this field theory admit the structure of {\em factorization algebras} on $M$. We will investigate the observable theory of our Rozansky-Witten model in this section.

In Section \ref{subsection:classical-observables}, we discuss the classical observables. In Section \ref{subsection:local-quantum-observables}, we will describe the local quantum observables supported on open subsets which are homeomorphic to handle bodies. In particular, we show that the cohomology of quantum observables on a genus $g$ handle body is isomorphic to $\mathcal{H}_g[[\hbar]]$, where $\mathcal{H}_g$ is the Hilbert space associated to a genus $g$ Riemann surface as conjectured by Rozansky and Witten in \cite{RW}. In Section \ref{subsection:correlation-function}, we will look into the global quantum observables and define their correlation functions. We show that the partition function of the model, namely, the correlation function of the constant observable $1$,  gives rise to the Rozansky-Witten invariants.

\subsection{Classical observables}\label{subsection:classical-observables}
The main goal of this subsection is to describe the classical observables on open subsets which are homeomorphic to handle bodies $H_g$.  Let us first recall the definition of classical observables.
\begin{defn}
The {\em classical observables} of our Rozansky-Witten model is the graded commutative factorization algebra whose value on an open subset $U\subset M$ is the following cochain complex:
\begin{equation}\label{eqn:classical-observable}
\text{Obs}^{cl}(U):=\left(\mathcal{O}(\E(U))), Q+\fbracket{I_{cl},-}\right).
\end{equation}
Here $\E(U):=\A_M(U)\otimes\g_X[1]$, and
$$
\mathcal{O}(\E(U)):=\prod_{k\geq 0}\text{Sym}^k(\mathcal{E}(U)^\vee)
$$
are formal power series on $\E(U)$ valued in $\A_X$.
\end{defn}

Let $H_g\subset M$ be an open subset which is homeomorphic to a genus $g$ handle body.
\begin{equation}\label{eqn:handle-body}
 \figbox{0.33}{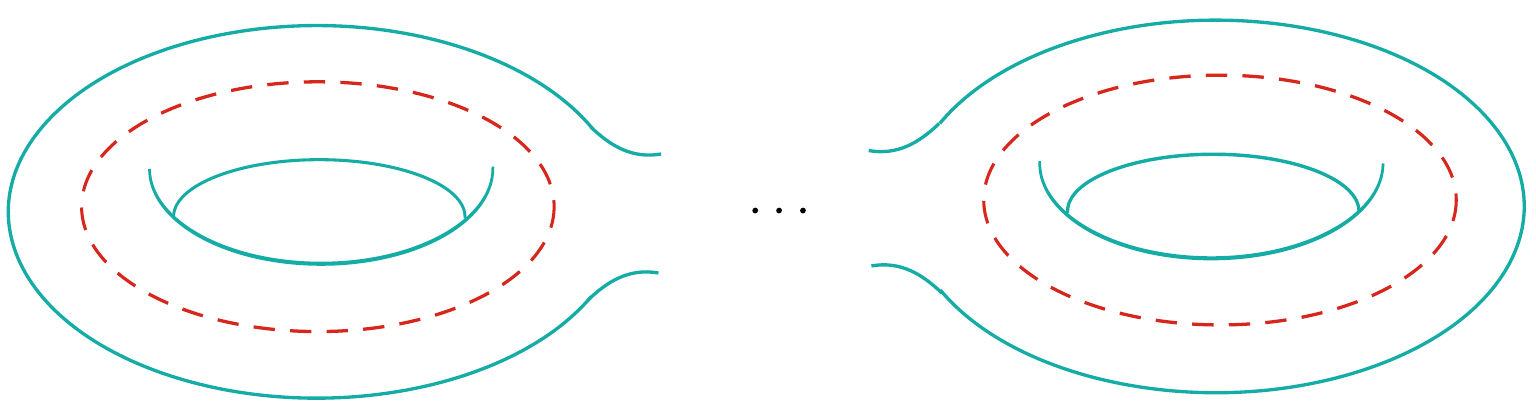}
\end{equation}
The symplectic pairing on the space of fields induces an isomorphism
$$
\mathcal{E}(H_g)^\vee\cong \bar{\mathcal{E}}_{cpt}(H_g),
$$
where $\bar{\mathcal{E}}_{cpt}(H_g)$ denotes the space of compactly-supported distributional sections of $\mathcal{E}$ over $H_g$. To compute the cohomology of the cochain complex \eqref{eqn:classical-observable}, note that the differential $Q+\{I_{cl},-\}$ can be written as the sum of the de Rham differential $d_M$ on $M$ and  $\nabla+\{I_{cl},-\}$. We consider the filtration on $\mathcal{O}(\mathcal{E}(H_g))$ induced by the degree of differential forms in $\A_X$:
$$
F^k(\text{Obs}^{cl})=\A_X^k\ \text{Obs}^{cl}.
$$
It is clear that the terms $\nabla+\{I_{cl},-\}$ in the differential increases the degree by $1$ while $d_M$ keeps it fixed.

By Atiyah-Bott's lemma \cite{AB}, when computing the de Rham cohomology, the chain complex of compactly-supported currents can be replaced by that of compactly-supported smooth differential forms since they are quasi-isomorphic.  Thus the first page of the spectral sequence associated to this filtration is given by
$$
\left(\text{Sym}^*\left(H_{cpt}^*(H_g)\otimes\g_X[1]^\vee\right),\nabla+\fbracket{I_{cl},-}\right).
$$
The compactly-supported de Rham cohomology on $H_g$ is described in the following lemma, whose proof follows from Poincar\'{e} duality and a standard computation of singular homology of $H_g$ using Mayer-Vietoris sequences:
\begin{lem}\label{lemma:compactly-supported-cohomology-handle-body}
The compactly-supported de Rham cohomology of $H_g$ is given by:
\begin{equation*}
 \begin{aligned}
  &H^3_{cpt}(H_g)\cong H_0(H_g)\cong\C,\\
  &H^2_{cpt}(H_g)\cong H_1(H_g)\cong\C^g,\\
  &H^i_{cpt}(H_g)=0,\ \textit{for}\ i=0,1.
 \end{aligned}
\end{equation*}
\end{lem}
\begin{rmk}
A basis of $H_1(H_g)$ is given by the dashed circles in the picture \eqref{eqn:handle-body}. The corresponding functional on smooth $1$-forms is simply the integral along these circles, which are of $\delta$-function type. For later consideration of RG flows of observables, we will take smooth representatives of these functionals.
\end{rmk}
\begin{rmk}
For the computation of  (both classical and quantum) local observables later, the framing correction in equation \eqref{eqn:interaction-choice-of-framing} does not cause problems. This is because we can choose a locally flat metric and a framing such that locally the connection $1$-forms vanish.
\end{rmk}

An immediate corollary of the above lemma is the following:
\begin{prop}\label{proposition:quasi-iso-observable-jet}
With the choice of a basis of smooth functionals on $\A^1(H_g)$ corresponding to the dashed circles in \eqref{eqn:handle-body} under the isomorphism $H^2_{cpt}(H_g)\cong H_1(H_g)$,  the first page of the spectral sequence is given by
\begin{equation}\label{eqn:classical-observable-L-infty-jet}
\left(Sym^*\left(H_{cpt}^*(H_g)\otimes\g_X[1]^\vee\right),\nabla+\fbracket{I_{cl},-}\right)\cong\left(\mathcal{W}\otimes_{\OO_X}\left(\wedge^*T_X^\vee\right)^{\otimes g},D\right).
\end{equation}
\end{prop}
\begin{proof}
Let $e_0\in H_0(H_g;\mathbb{Z})$ be a generator, and let $\{e_1,\cdots,e_g\}$ be a basis of $H_1(H_g;\mathbb{Z})$. It is clear from Lemma
\ref{lemma:compactly-supported-cohomology-handle-body} that $\{e_0,e_1,\cdots, e_g\}$ induces the following isomorphisms:
\begin{align*}
\Sym^*\left(H_{cpt}^*(H_g)\otimes\g_X[1]^\vee\right)&\cong\Sym^*\left(H_0(H_g)\otimes\g_X[1]^\vee\right)\otimes_{\A_X}
\Sym^*\left(H_1(H_g)\otimes\g_X[1]^\vee\right)\\
&\cong C^*(\g_X,\wedge^*(\C^g\otimes\g_X[1]^\vee))\\
&\cong C^*(\g_X,(\wedge^*\g_X[1]^\vee)^{\otimes g}),
\end{align*}
where the last line is an $L_\infty$-module over $\g_X$ with a canonical Chevalley-Eilenberg differential. The above isomorphism is actually compatible with the Chevalley-Eilenberg differential and the differential on the left hand side. There is the following isomorphism of cochain complexes, which finishes the proof:
\begin{equation}\label{eqn:local-observable-CE-jet}
\left(C^*(\g_X,(\wedge^*\g_X[1])^{\otimes g}), d_{CE}\right)\cong\left(\mathcal{W}\otimes_{\OO_X}\left(\wedge^*T_X^\vee\right)^{\otimes g},D\right).
\end{equation}
\end{proof}

Similar to the proof of Proposition \ref{prop:weyl-bundle-quasi-isomorphic-Dolbeault}, the cochain complex of the RHS of equation \eqref{eqn:local-observable-CE-jet} is quasi-isomorphic to the Dolbeault complex of $\left(\wedge^*T_X^\vee\right)^{\otimes g}$, and we conclude that
\begin{thm}\label{thm:local-classical-observable}
The cohomology of local classical observables on the handle body $H_g\subset M$ is given by
$$
H^*\left(\text{Obs}^{cl}(H_g),Q+\{I_{cl},-\}\right)\cong H^*\left(X, \left(\wedge^*T_X^\vee\right)^{\otimes g}\right) \cong H^*\left(X,\left(\wedge^*T_X\right)^{\otimes g}\right).
$$
The last identity follows from the isomorphism $T_X\cong T_X^\vee$ induced by the holomorphic symplectic structure $\omega$.
\end{thm}

\subsection{Local quantum observables}\label{subsection:local-quantum-observables}
In this subsection, we study the local quantum observables of our Rozansky-Witten model. In particular, we show that the local quantum observables do not receive any quantum corrections, and hence their cohomology is isomorphic to that of local classical observables.

We first briefly recall the definition and properties of local quantum observables; again, we refer the readers to \cite{Kevin-Owen} for more details. The philosophy of defining quantum observables is similar to that of defining quantum interactions: a quantum observable is, instead of a single observable, a family parametrized by scales. However, one subtlety is that the property of being supported on a fixed open set is not preserved under the RG flow using the propagators $\mathbb{P}_\epsilon^L$. To fix this issue, we have to apply the technique of parametrices:
\begin{defn}
A {\em parametrix} $\Phi$ is a distributional section
$$
   \Phi \in \Sym^2\bracket{\overline{\E}}
$$
with the following properties:
\begin{enumerate}
\item $\Phi$ is of cohomological degree $1$ and $(Q\otimes 1+1\otimes Q)\Phi=0$,
\item $\frac{1}{2}\bracket{H\otimes 1+1\otimes H}\Phi-\mathbb{K}_{0} \in \Sym^2\bracket{\E}$ is smooth, where $H=[Q,Q^{GF}]$ is the Laplacian and $\mathbb{K}_{0}=\lim\limits_{L\to 0}\mathbb{K}_L$ is the kernel of the identity operator.
\end{enumerate}
\end{defn}
In particular, the heat kernels $\mathbb{K}[L]$ are parametrices. Moreover, the propagators and heat kernels associated to parametrices can be defined in a similar way as $\mathbb{P}[L]$.

\begin{defn}
We define the {\em propagator} $\mathbb{P}(\Phi)$ and {\em BV kernel} $\mathbb{K}_{\Phi}$ associated to a parametrix $\Phi$ by
$$
  \mathbb{P}(\Phi):={1\over 2}\bracket{Q^{GF}\otimes 1+1\otimes Q^{GF}}\Phi \in \Sym^2\bracket{\overline{\E}}, \quad \mathbb{K}_{\Phi}:=\mathbb{K}_0-{1\over 2}\bracket{H\otimes 1+1\otimes H}\Phi.
$$
The {\em effective BV operator} $\Delta_\Phi:={\pa \over \pa \mathbb{K}_\Phi}$ induces a {\em BV bracket} $\{-,-\}_\Phi$ on $\OO\bracket{\E}$ in a similar way as the scale $L$ BV bracket $\{-,-\}_L$.
\end{defn}

Let $\Phi$ and $\Psi$ be two parametrices, we can define the RG flow operator from scale $\Phi$ to scale $\Psi$ similar to Definition \ref{defn:RG-flow-operator}, with $\mathbb{P}_\epsilon^L$ now replaced by $\mathbb{P}[\Psi]-\mathbb{P}[\Phi]$. The good thing of general parametrices is that there exist $\Phi$'s such that $\mathbb{K}[\Phi]$ (and $\mathbb{P}[\Phi]$) are supported close enough to the diagonal in $M\times M$, unlike $\mathbb{K}[L]$. Thus RG flow using such parametrices increases the support of observables only in a mild way.

\begin{defn}
A {\em local quantum observable} $O$ supported on an open set $U\subset M$ assigns to each parametrix $\Phi$ a functional
$$
O[\Phi]\in\OO(\E)[[\hbar]],
$$
such that the following conditions are satisfied:
\begin{enumerate}
 \item The renormalization group equation (RGE) is satisfied,
 \item $O[\Phi]$ is supported on $U$ when the support of $\Phi$ is close enough to the diagonal $\Delta\subset M\times M$.
\end{enumerate}
\end{defn}

The local quantum observable on $U$ also form a cochain complex with the differential $\hat{Q}$ defined by
$$
(\hat{Q}O)[\Phi]:=Q_\Phi(O[\Phi])+\fbracket{I[\Phi],O[\Phi]}_{\Phi}+\hbar\Delta_{\Phi}O[\Phi].
$$
$\hat{Q}$ squares to $0$ exactly because $I[\Phi]$ satisfies the QME.

\begin{rmk}
The differential can be equivalently written as
$$
(Q+\hbar\Delta_{\Phi})(O[\Phi]e^{I[\Phi]/\hbar})=(Q(O[\Phi])+\fbracket{I[\Phi],O[\Phi]}_{\Phi}+\hbar\Delta_{\Phi}O[\Phi])e^{I[\Phi]/\hbar}.
$$
\end{rmk}

It takes a little bit of work to show that $\hat{Q}O$ is also supported on $U$; see \cite{Kevin-Owen} for the proof. By the general theory of observables, the cohomology of local quantum observables is a deformation of that of classical observables:
$$
H^*(\text{Obs}^q,\hat{Q})\otimes_{\C[[\hbar]]}\C\cong H^*(\text{Obs}^{cl},Q+\{I_{cl},-\}).
$$

By Proposition \ref{proposition:quasi-iso-observable-jet}, a choice of basis of $H^*_{cpt}(H_g)$ induces a quasi-isomorphic embedding
$$
\left(\mathcal{W}\otimes_{\OO_X}\left(\wedge^*T_X^\vee\right)^{\otimes g},D\right)\hookrightarrow\text{Obs}^{cl}(H_g).
$$
Let $O_\mu\in\mathcal{W}\otimes_{\OO_X}(\wedge^*T_X^\vee)^{\otimes g}$ be a smooth local classical observable on $H_g$ representing $\mu\in H^*\left(X,\left(\wedge^*T_X^\vee\right)^{\otimes g}\right)$. To construct a quantum observable satisfying
$$
\lim_{L\rightarrow 0}O_\mu[L]=O_{\mu} \mod\hbar,
$$
it is enough to define the corresponding quantum observable associated to the parametrices $\mathbb{K}[L]$ as the value at a general parametrix $\Phi$ can be obtained by running the RG flow from $\mathbb{P}_0^L$ to $\mathbb{P}[\Phi]$.

Here we give an explicit description of the observable $O_\mu$ in terms of graphs. In the following picture, each yellow vertex denotes a smooth $2$-form Poincar\'{e} dual to an embedded circle on $H_g$, which represents a functional on $1$-form inputs on $M$.  And each green vertex denotes the volume form on $M$ which represents a functional on $0$-forms inputs on $M$. Note that for an observable $O_\mu$, there could be infinite many green vertices but only finite many yellow ones.
$$
\figbox{0.23}{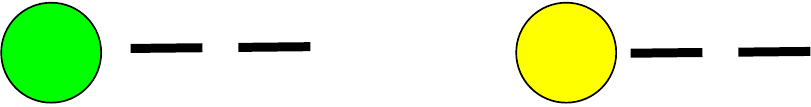}
$$
In the above picture, every vertex allows only one input. To relate the graph with equation (\ref{eqn:local-observable-CE-jet}), note that there is the following:
$$
C^*\left(\g_X,\left(\wedge^*\g_X[1]^\vee\right)^{\otimes g}\right)\cong\mathcal{W}\otimes_{\OO_X}\left(\wedge^*T_X^\vee\right)^{\otimes g}.
$$
The green vertices contribute to the Weyl bundle component $\mathcal{W}$ and the yellow vertices contribute to the component $\left(\wedge^*T_X^\vee\right)^{\otimes g}$.

By the smoothness of the local observable $O_\mu$, we can construct the RG flow of these observables using the technique of configuration spaces. We have the following homomorphism:
\begin{equation}\label{eqn:splitting-classical-quantum-observable}
\begin{aligned}
\mathcal{W}\otimes_{\OO_X}\left(\wedge^*T_X^\vee\right)^{\otimes g}&\rightarrow\text{Obs}^q(H_g)\\
O_\mu&\mapsto O_\mu[L] := W\left(\tilde{\mathbb{P}}_0^L, I_{cl}, O_\mu\right).
\end{aligned}
\end{equation}

\begin{rmk}
By abuse of notations, we use $O_\mu$ to stand for both the classical and quantum observables.
\end{rmk}

It is not difficult to see that the quantum observable $O_\mu$ is supported on $H_g$ by construction. Here $W(\tilde{\mathbb{P}}_\epsilon^L, I_{cl}, O_\mu)$ denotes the sum of those Feynman weights such that exactly one vertex is labeled by $O_\mu$ and all the other vertices are labeled by $I_{cl}$. The following lemma shows that the homomorphism \eqref{eqn:splitting-classical-quantum-observable} is in fact a cochain map:
\begin{lem}
The homomorphism \eqref{eqn:splitting-classical-quantum-observable} commutes with differentials.
\end{lem}
\begin{proof}
To look at the differential
\begin{equation}\label{eqn:quantum-differential-observable}
Q_LO_\mu[L]+\fbracket{I[L],O_\mu[L]}_L+\hbar\Delta_LO_\mu[L],
\end{equation}
we can apply the same method as in the proof of Theorem \ref{theorem: QME}, and equation (\ref{eqn:quantum-differential-observable}) becomes some integral on the codimension $1$ strata of the configuration spaces. We only look at those primitive boundary integrals, and we have the following observations:
\begin{itemize}
 \item Note that since $I[L]$ already satisfies the scale $L$ QME, we only need to look at those graphs containing a vertex representing the observable $O_\mu$.
 \item The same argument as in Theorem \ref{theorem: QME} shows that if the graph contains at least two vertices labeled by $I_{cl}$, then the corresponding primitive boundary integral vanishes.
\end{itemize}
Thus it remains to consider the following graphs with exactly two vertices, one labeled by $I_{cl}$, and the other by the observable $O_\mu$:
$$
\figbox{0.23}{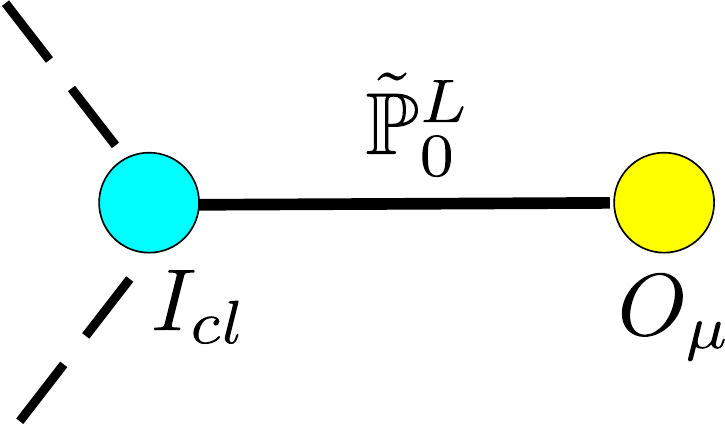}
$$
The propagator in the above picture contributes the volume form on the fiber of $S(TM)$, and such graphs are exactly representing the functional
$$
QO_\mu+\fbracket{I_{cl},O_\mu}=O_{d_{D_X}\mu}.
$$
Thus we conclude that
$$
Q_LO_\mu[L]+\fbracket{I[L],O_\mu[L]}_L+\hbar\Delta_LO_\mu[L]
$$
is exactly the same as $O_{d_{D_X}\mu}[L]$.
\end{proof}

Note that the cochain complex of local quantum observables is $\Z_2$-graded. We have the following exact sequence for every
$k\geq 0$:
$$
\xymatrix{
 H^0(\hbar^{k+1}\cdot\text{Obs}^q(H_g))\ar[r] & H^0(\hbar^k\cdot\text{Obs}^q(H_g)) \ar[r] &
H^0(\text{Obs}^{cl}(H_g))\ar[d]\\
 H^1(\text{Obs}^{cl}(H_g)) \ar[u] & H^1(\hbar^k\cdot\text{Obs}^q(H_g)) \ar[l] &
H^1(\hbar^{k+1}\cdot\text{Obs}^q(H_g))\ar[l]
}
$$
and the cochain map \eqref{eqn:splitting-classical-quantum-observable} induces the following split exact sequences:
$$
0\rightarrow H^i(\hbar^{k+1}\cdot\text{Obs}^q(H_g))\rightarrow H^i(\hbar^k\cdot\text{Obs}^q(H_g))\rightarrow
H^i(\text{Obs}^{cl}(H_g))\rightarrow 0, \qquad i=0,1,\ k\geq 0.
$$
We summarize our discussions in the following theorem:
\begin{thm}[=Theorem \ref{thm:main1}]\label{thm:classical-to-quantum-observables}
The cohomology of local quantum observables on a genus $g$ handle body  is isomorphic to that of classical observables:
$$
H^*\left(\text{Obs}^q(H_g),\hat{Q}\right)\cong H^*\left(\text{Obs}^{cl}(H_g),Q+\{I_{cl},-\}\right)[[\hbar]]\cong
H^*(X,(\wedge^*T_X)^{\otimes g})[[\hbar]].
$$
\end{thm}
Corollary \ref{cor:main} follows immediately from this and Theorem \ref{thm:local-classical-observable}.

\subsection{Global quantum observables and correlation functions}\label{subsection:correlation-function}
In this subsection, we investigate global quantum observables of our Rozansky-Witten model and define their correlation functions. Our main purpose is to demonstrate that the partition function of our model gives rise to the original Rozansky-Witten invariants \cite{RW}.

Since the cochain complexes of observables are essentially independent of scales, we will restrict ourselves to the $\infty$-scale:
\begin{equation}\label{eqn:global-quantum-observables}
\text{Obs}^q(M):=\left(\OO(\E),Q_\infty+\{I[\infty],-\}_\infty+\hbar\Delta_\infty\right).
\end{equation}
\begin{defn}
The {\em space of harmonic fields} are those annihilated by the Laplacian $[Q,Q^{GF}]$:
$$
\H:=\fbracket{e\in\E: [Q,Q^{GF}]e=0}.
$$
\end{defn}

The following lemma reduces the computation of the cohomology of global observables on the whole fields $\E$ to that on the harmonic fields $\H$.
\begin{lem}
The following two cochain complexes are quasi-isomorphic:
$$
\left(\OO(\E),Q_\infty+\fbracket{I[\infty],-}_\infty+\hbar\Delta_\infty\right)\cong
\left(\OO(\H),(Q_\infty+\fbracket{I[\infty],-}_\infty+\hbar\Delta_\infty)|_\H\right).
$$
\end{lem}
\begin{proof}
The map is simply the restriction of functionals on harmonic fields, which is easily seen to be a cochain map due to the fact that
$$
\mathbb{K}[\infty]|_{\H}=\mathbb{K}[\infty].
$$
We now check that this is a quasi-isomorphism. Surjectivity is trivial as every functional on $\H$ can be extended to $\E$ because $\H$ is a closed subspace of $\E$. To show injectivity, let $O\in\OO(\E)[[\hbar]]$ be a global observable which is exact when restricted to $\H$:
\begin{equation}\label{eqn:observable-exact-harmonic}
O[\infty]|_\H=\nabla(O')+\{I[\infty],O'\}_\infty+\hbar\Delta_\infty O'.
\end{equation}
Here $O'$ denotes a functional on $\H$, and by abuse of notations again we will use the same notation for an extension to $\OO(\E)$. We will now write $O$ and $O'$ in their expansion in powers of $\hbar$:
$$
O[\infty]=O[\infty]^{(0)}+\hbar O[\infty]^{(1)}+\hbar^2O[\infty]^{(2)}+\cdots, \qquad O'[\infty]=O'[\infty]^{(0)}+\hbar
O'[\infty]^{(1)}+\hbar^2O'[\infty]^{(2)}+\cdots
$$
It is clear from equation \eqref{eqn:observable-exact-harmonic} that $O[\infty]_{(0)}$ is the $\infty$-scale tree-level RG flow of a
classical observable which is exact. Thus there exists $O_0[\infty]\in\OO(\E)$ such that
$$
O[\infty]^{(0)}=QO_0[\infty]+\{I[\infty]^{(0)},O_0[\infty]\}_\infty.
$$
Then we pick any global quantum observable $O_0^q$ whose tree-level (classical) part is just $O_0$. Now the observable
$$
O-\hat{Q}O^q_0\in\hbar\cdot\text{Obs}^q(M),
$$
i.e. it starts from the $\hbar$ term, and remains exact when restricted to $\H$ (since the restriction is a cochain map). This procedure can be repeated and we see that $O$ is exact.
\end{proof}

With the above lemma, we can describe the global quantum observables in terms of the geometry of the following BV bundle over $X$.
\begin{defn}
The BV bundle on $X$ associated to a 3-dimensional manifold $M$ is defined as:
\begin{equation}\label{eqn:BV-bundle}
 \Sym^*\left(\bigoplus_{i=0}^3(\H^i(M)\otimes T_X)^\vee\right)[[\hbar]].
\end{equation}

\end{defn}
In particular, the cochain complex of global observables of Rozansky-Witten model is quasi-isomorphic to the BV bundle with the differential $Q_\infty+\fbracket{I[\infty],-}_\infty+\hbar\Delta_\infty$. The BV bundle  consists of the ``bosons'' and ``fermions'', corresponding to the dual of even and odd cohomologies on $M$, respectively.

Suppose $\dim_{\C}(X) = 2n$. Then the holomorphic symplectic form $\omega$ gives a holomorphic volume form
$$
\Omega_X:=\omega^{n}\in\Gamma(X,\wedge^{2n}\Omega_X^1).
$$
\begin{notn}\label{notn:top-fermion}
We will let $\Omega_X^{\otimes(b_1(M)+1)}$ denote the  canonically normalized non-vanishing holomorphic section of the line bundle
$$
\wedge^{top}(\H^1(M)\otimes T_X)^\vee\otimes\wedge^{top}(\H^3(M)\otimes T_X)^\vee
$$
on $X$, which is naturally embedded in the BV bundle. The section $\Omega_X^{\otimes(b_1(M)+1)}$ is normalized canonically in the following sense: we can pick an integral basis of $H^1(M;\mathbb{Z})\subset H^1(M;\mathbb{R})$. This gives us a well-defined generator of $\wedge^{b_1(M)}H^1(M;\mathbb{R})$ up to a sign. The fact that holomorphic symplectic  manifolds are even dimensional over $\mathbb{C}$ resolves this ambiguity.
\end{notn}
\begin{rmk}\label{rmk:top-fermion-property}
 This section $\Omega_X^{\otimes(b_1(M)+1)}$ is the top degree fermion in the BV bundle. Moreover, it is annihilated by the BV operator $\Delta_\infty$ by its construction.
\end{rmk}


\subsubsection{BV integral and correlation function}
For the consideration of correlation functions, we first give an equivalent description of the complex of global observables, or equivalently the cochain complex of BV bundle. More precisely, there is the following isomorphism:
\begin{align*}
\left(\OO(\H)((\hbar)),Q_\infty+\fbracket{I[\infty],-}_\infty+\hbar\Delta_\infty\right)&\rightarrow\left(\mathcal{O}(\H)((\hbar)), Q_\infty+\hbar\Delta_\infty+\frac{\rho(R)}{\hbar}\right)\\
O&\mapsto  e^{I[\infty]/\hbar}\cdot O.
\end{align*}
In other words, there is the following identity of operators on the BV bundle:
\begin{equation}\label{eqn:different-form-BV-differential}
 Q_\infty+\fbracket{I[\infty],-}_\infty+\hbar\Delta_\infty=e^{-I[\infty]/\hbar}\circ \left(Q_\infty+\hbar\Delta_\infty+\frac{\rho(R)}{\hbar}\right)\circ e^{I[\infty]/\hbar}.
\end{equation}

\begin{prop}\label{prop:quasi-iso-jet-de-rham}
The following map is a quasi-isomorphism of complexes of sheaves, where the differential on the LHS is the usual de Rham differential, and the differential on the RHS is $Q_\infty+\hbar\Delta_\infty+\frac{\rho(R)}{\hbar}$
\begin{equation}\label{eqn:A_X-to-global-observable}
\begin{aligned}
 l:\A_X((\hbar))&\rightarrow \OO(\H)((\hbar)),\\
  \alpha&\mapsto\alpha\otimes\Omega_X^{\otimes(b_1(M)+1)}.
\end{aligned}
\end{equation}
\end{prop}
\begin{proof}
There is the following sequence of identities:
\begin{align*}
\left(Q_\infty+\hbar\Delta_\infty+\frac{\rho(R)}{\hbar}\right)\left(\alpha\otimes\Omega_X^{\otimes(b_1(M)+1)}\right)&\overset{(1)}{=}\left(\nabla+\hbar\Delta_\infty+\frac{\rho(R)}{\hbar}\right)\left(\alpha\otimes\Omega_X^{\otimes(b_1(M)+1)}\right)\\
&\overset{(2)}{=}\nabla\left(\alpha\otimes\Omega_X^{\otimes(b_1(M)+1)}\right)\\
&\overset{(3)}{=}d_X(\alpha)\otimes\Omega_X^{\otimes(b_1(M)+1)}.
\end{align*}
Here identity $(1)$ follows from the fact that the restriction of $Q_\infty$ on harmonic fields is just $\nabla$, identity $(2)$ follows from the following observations:
\begin{itemize}
 \item $\frac{\rho(R)}{\hbar}\cdot\Omega_X^{\otimes(b_1(M)+1)}=0$ by type reason: $\Omega_X^{\otimes(b_1(M)+1)}$ already contains the ``top fermion''.
 \item $\hbar\Delta_\infty$ annihilates $\Omega_X^{\otimes(b_1(M)+1)}$.
\end{itemize}
Identity $(3)$ follows from the compatibility of $\omega$ with $\nabla$. 

It follows immediately that $l$ is a cochain map, and that the induced map on cohomology sheaves is injective. The surjectivity follows easily from a spectral sequence associated to the filtration induced by degrees of forms in $\A_X$, and  the following Poincar\'{e} lemma:
\begin{lem}\label{lem:Poincare-lemma}
Let $\{x^i\}$ be even elements and let $\{\xi_i\}$ be odd elements, then we have
$$
H^*\left(\mathbb{C}[[x^1,\xi_1,\cdots,x^n,\xi_n]],\Delta=\sum_{i=1}^n\dfrac{\partial}{\partial
x^i}\dfrac{\partial}{\partial\xi_i}\right)=\mathbb{C}\xi_1\wedge\cdots\wedge\xi_n.
$$
\end{lem}
\end{proof}


This quasi-isomorphism of cochain complexes gives an identification of the global quantum observables of the Rozansky-Witten model and the de Rham coholomogy of $X$. In particular, this enables us to define the correlation functions of (closed) quantum observables in terms of integrals of differential forms on $X$:
\begin{defn}
Let $O$ be a quantum observable (at scale $\infty$) which is closed under the quantized differential:
$$
Q_\infty(O)+\fbracket{I[\infty], O}_\infty+\hbar\Delta_\infty O=0.
$$
Let $|H_1(M;\mathbb{Z})|'$ denote the number of torsion elements in $H_1(M;\mathbb{Z})$. Then the {\em correlation function} of $O$ is defined as:
$$
\langle O\rangle_M:=|H_1(M;\mathbb{Z})|'\cdot\int_{X} [e^{I[\infty]/\hbar}\cdot O].
$$
Here $[e^{I[\infty]/\hbar}\cdot O]$ denotes the corresponding de Rham cohomology class of $e^{I[\infty]/\hbar}\cdot O$ under the quasi-isomorphism (\ref{eqn:A_X-to-global-observable}).
\end{defn}
\begin{rmk}
We add the normalization factor $|H_1(M;\mathbb{Z})|'$ which is the number of torsion elements in $H_1(M;\mathbb{Z})$, in order to be consistent with the results in physics.
\end{rmk}

To conclude, we have an explicit algorithm of computing the correlation function of a global observable $O$ which is closed under the scale $\infty$ quantum differential $Q_\infty+\fbracket{I[\infty],-}_\infty+\hbar\Delta_\infty$. We start from the observable $e^{I[\infty]/\hbar}\cdot O$, which is a section of the BV bundle. Then we project it to the trivial line bundle in Notation \ref{notn:top-fermion} and take the coefficients with respect to the trivialization therein to obtain a differential form in $\mathcal{A}_X$, whose integral over $X$ is the correlation function of the observable $O$. Later, we will compute the partition function, i.e. the correlation of the observable $1$ as an example.
\begin{rmk}
In physics terminology, the ``top fermion'' $\Omega_X^{\otimes(b_1(M)+1)}$ gives rise to the volume form on the fermions, and the integral of $\alpha$ on $X$ arises from the volume form on the bosons. Thus, for a global quantum observble $O$, finding the cohomology class
$$
[O]\in H_{dR}(X)
$$
is essentially ``integration over the fermions''.
\end{rmk}

Using the factorization algebra structure, we can define the correlation function of local quantum observables:
\begin{defn}\label{definition:factorization-product}
Let $U_1, \cdots, U_n$ be disjoint open subsets of $M$. The {\em factorization product}
$$
  \text{Obs}^q(U_1)\times \cdots \times \text{Obs}^q(U_n)\to \text{Obs}^q(M)
$$
of local observables $O_i\in \text{Obs}^q(U_i)$ will be denoted by $O_1\star\cdots\star O_n$. This product descends to cohomologies.
\end{defn}

\begin{defn}
Let $U_1, \cdots, U_n$ be disjoint open subsets of $M$, and let $O_{i} \in \text{Obs}^q(U_i)$ be closed local quantum observables supported on $U_i$. We define their {\em correlation function} by
$$
 \abracket{O_1,\cdots, O_n}_M:=\abracket{O_1\star\cdots \star O_n }_{M}\in \C((\hbar)).
$$
\end{defn}

By the previous discussions, in order to compute the partition function, we only need to find the ``top fermions'' in $e^{I[\infty]/\hbar}$ which give rise to a top degree differential form on $X$, and integrate it on $X$. First we have the following lemma:
\begin{lem}
Let $\gamma$ be a connected graph with $|V(\gamma)|\geq 2$. If one of the input is a $3$-form on $M$, then the corresponding Feynman weight vanishes when restricted to the harmonic fields $\H$.
\end{lem}
\begin{proof}
We can assume that the vertex allowing a harmonic $3$-form on $M$ is at least bivalent to internal edges, since otherwise $d^*$ in the propagator annihilates the harmonic $3$-form. A simple counting shows that  at least one of the following two situations must happen:
\begin{enumerate}
 \item There exists a vertex which is bivalent to internal edges whose inputs are all $0$-forms on $M$.
 \item There exists a vertex which is univalent to an internal edge whose input is a  harmonic form.
\end{enumerate}
In both  situations, the Feynman weights must vanish.
\end{proof}

Thus all the Feynman graphs in $e^{I[\infty]/\hbar}$ that contribute $3$-forms must contain a single vertex of the following form: 
$$
\figbox{0.3}{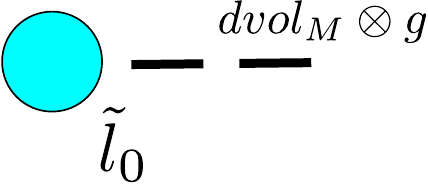}
$$
Suppose $\dim_\C(X)=2n$, then we have the following simple observations:
\begin{enumerate}
\item Every single vertex that corresponds to $\tilde{l}_0$ contributes a $(1,0)$ form in $\A_X$. The number of such vertices must be exactly $2n$ when we take the top fermion in $e^{I[\infty]/\hbar}$.
\item Every vertex in a connected Feynman graph $\gamma$ with $|V(\gamma)|\geq 2$ can only contribute a $(0,1)$-form on $X$, by the previous observation. And the number of all such vertices must also be $2n$, for the integration on $X$.
\end{enumerate}
\begin{rmk}
Roughly, the classical interaction $I_{cl}$ corresponds to the curvature of the connection $\nabla$ we chose at the beginning. In general, the curvature is the sum of $(2,0)$ and $(1,1)$ parts:
$$
R=R^{(1,1)}+R^{(2,0)},
$$
in which $F^{(1,1)}$ is also known as the Atiyah class. The above observations explain the reason why only the Atiyah classes contributes in the partition function of the Rozansky-Witten model.
\end{rmk}

Another constraint comes from the differential forms on the source $M$. Note that the single vertices labeled by $\tilde{l}_0$ already contribute all ``fermions'' corresponding to harmonic $3$-forms on $M$. The number of inputs of the form
$$
\alpha\otimes g\in\H^1(M)\otimes\g_X
$$
must be exactly $2n\cdot b_1$. Thus it remains to find all graphs (possibly non-connected) with non-trivial Feynman weights whose inputs are all of the above form. Such a graph $\gamma$ satisfies the following conditions:
\begin{enumerate}
 \item The number of vertices $|V(\gamma)|=2n$,
 \item The number of tails $|T(\gamma)|=2n\cdot b_1$, where the inputs are all of the form $\alpha\otimes g\in\H^1(M)\otimes\g_X$,
 \item The number of edges $|E(\gamma)|=\frac{3|V(\gamma)|-|T(\gamma)|}{2}=(3-b_1)n$. (This follows from the fact that tails and edges should contribute altogether a $6n$-form on $M^{2n}$).
\end{enumerate}
In particular, the third constraint implies that the partition function is trivial when $b_1(M)>3$. Another simple consequence is the following: the average valency of vertices (to both internal edges and tails) is $\frac{2|E(\gamma)|+|T(\gamma)|}{|V(\gamma)|}=3$. Thus each vertex must be exactly trivalent.
\begin{rmk}
These graphs are called ``minimal Feynman diagrams'' in \cite{RW}. What we have shown is that in our model the contribution of the Feynman weights of those ``non-minimal''  graphs is trivial in the partition function as the case in \cite{RW}.
\end{rmk}

We conclude with the following theorem:
\begin{thm}[=Theorem \ref{thm:main2}]\label{thm:partition-fcn-RW}
The partition function $\abracket{1}_M$ of our Rozansky-Witten model with domain $M$ and target manifold $X$ agrees with the original perturbative partition function $Z_X(M)$ \eqref{eqn:RW_partition_function} as computed by Rozansky and Witten in \cite{RW}.
\end{thm}
\begin{proof}
We have shown in the previous part that if $b_1(M)>3$, then the partition function of $M$ vanishes. Thus we only need to prove the theorem for the cases $b_1(M)=0,1,2,3$ respectively. It is clear that the partition function of our Rozansky-Witten model consists of two parts. One is the analytic part contributed from the configuration spaces integral associated to the source $M$, which is the same as the physical definition of RW invariants. Another is the combinatorial part contributed from the differential forms on $X$, which are called the Rozansky-Witten classes of $X$ in \cite{Kapranov}. Since the classical interaction $I_{cl}$ labeling each vertex gives rise to the Atiyah class of $X$, the combinatorial parts of our Feynman weights are the same as the Rozansky-Witten classes in \cite{Kapranov}. Thus we only need to show that in our computation of the partition function, the graphs which contribute nontrivially are the same as those in \cite{RW}.
\begin{enumerate}
 \item $b_1(M)=3$. Let $\alpha_1,\alpha_2,\alpha_3$ be a basis for harmonic $1$-forms on $M$. We only need to consider all those trivalent graphs $\Gamma$ with $2n$ vertices, such that each of them will absorb the all the harmonic  $1$-forms $\alpha_1,\alpha_2$ and $\alpha_3$ on $M$. Thus all the vertices in $\Gamma$ are disconnected, and $\Gamma$ is unique.

 \item $b_1(M)=2$: Let $\alpha_1,\alpha_2$ be a basis for harmonic $1$-forms on $M$. The only trivalent graphs that will contribute non-trivially to the partition must contain $2n$-vertices, each of them exactly absorbs two harmonic forms $\alpha_1$ and $\alpha_2$. These $2n$ vertices must be connected by the $\infty$-scale propagator $P_0^\infty$ (the Green function) to form $n$ pairs.

 \item $b_1(M)=1$. Let $\alpha_1$ be a generator of harmonic $1$-forms on $M$. Since $\alpha_1\wedge\alpha_1=0$, each vertex in a graph $\Gamma$ can absorb exactly one $\alpha_1$. Such graphs $\Gamma$ must be the disjoint of circles with vertices on it, as shown in the following picture (n=3):
$$
\figbox{0.28}{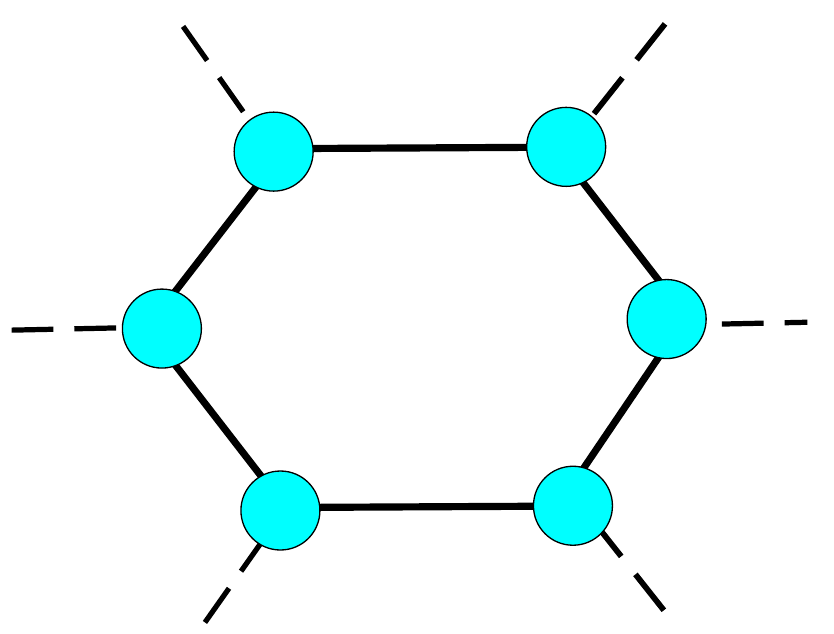}\hspace{8mm},\hspace{8mm}\figbox{0.28}{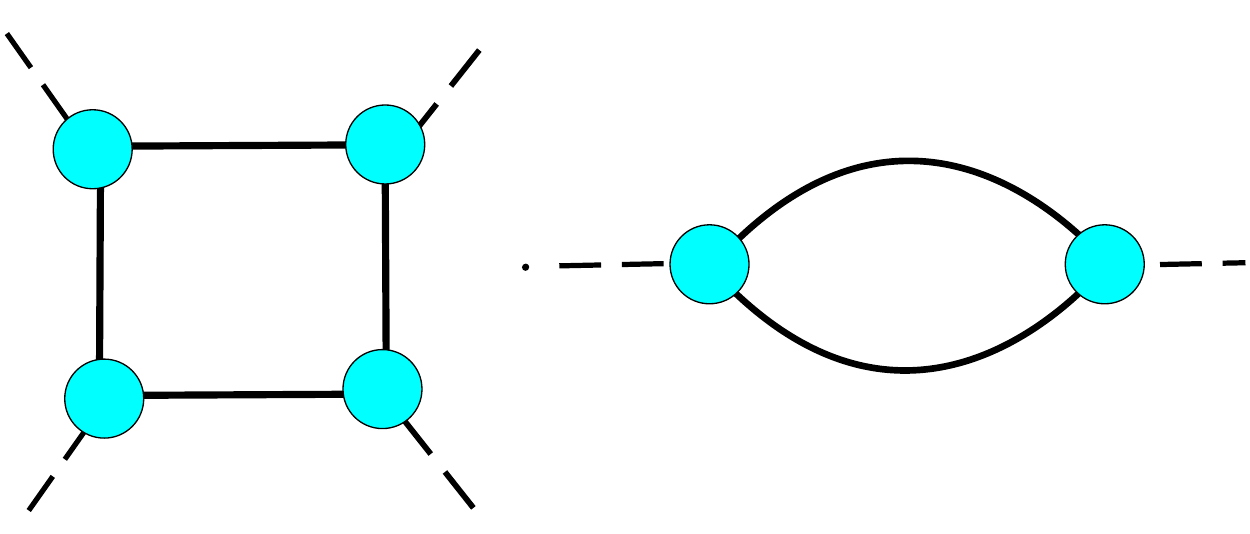}\hspace{8mm},\hspace{8mm}\cdots
$$
 \item $b_1(M)=0$. Since there are no harmonic $1$-forms,  a graph $\Gamma$ that contributes non-trivially must contain exactly $2n$-vertices such that they are connected to each other by the Green function.
\end{enumerate}
In all these cases, the partition function coincides with that in \cite{RW}.
\end{proof}
\begin{rmk}
It is not difficult to see from the above proof that there could be more complicated graphs that contribute non-trivially to the partition function as $b_1(M)$ becomes smaller and the dimension of $X$ becomes bigger. For instance, for the case where $n=2$ and  $b_1(M)=0$, i.e. when $M$ is a rational homology sphere, there could be the following graphs:
$$
\figbox{0.3}{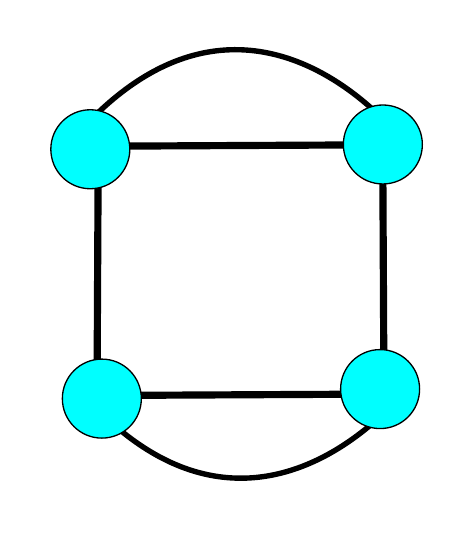}\hspace{10mm}\figbox{0.25}{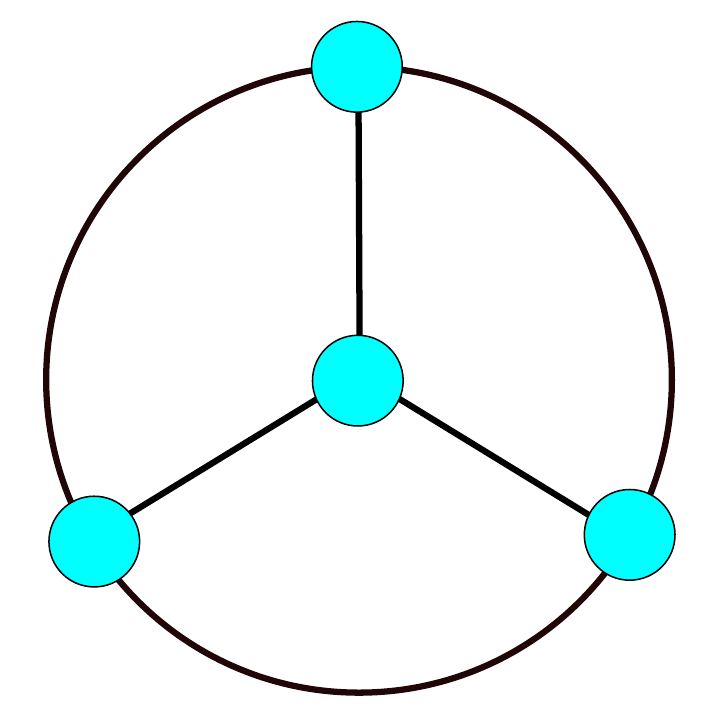}
$$
\end{rmk}
\begin{example}
Let $X$ be a K3 surface, then the graphs which contribute to the partition function can be shown as in the following pictures, depending on the first Betti number of the domain $M$:
\begin{enumerate}
 \item $b_1(M)=0$
       $$
       \figbox{0.23}{3-form-input}\hspace{8mm} \figbox{0.23}{3-form-input}
       $$
       $$
       \figbox{0.23}{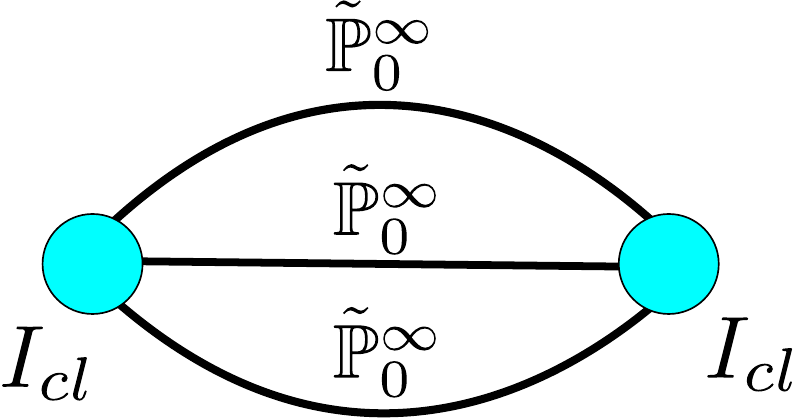}
       $$
 \item $b_1(M)=1$
       $$
       \figbox{0.23}{3-form-input}\hspace{8mm} \figbox{0.23}{3-form-input}
       $$
       $$
       \figbox{0.23}{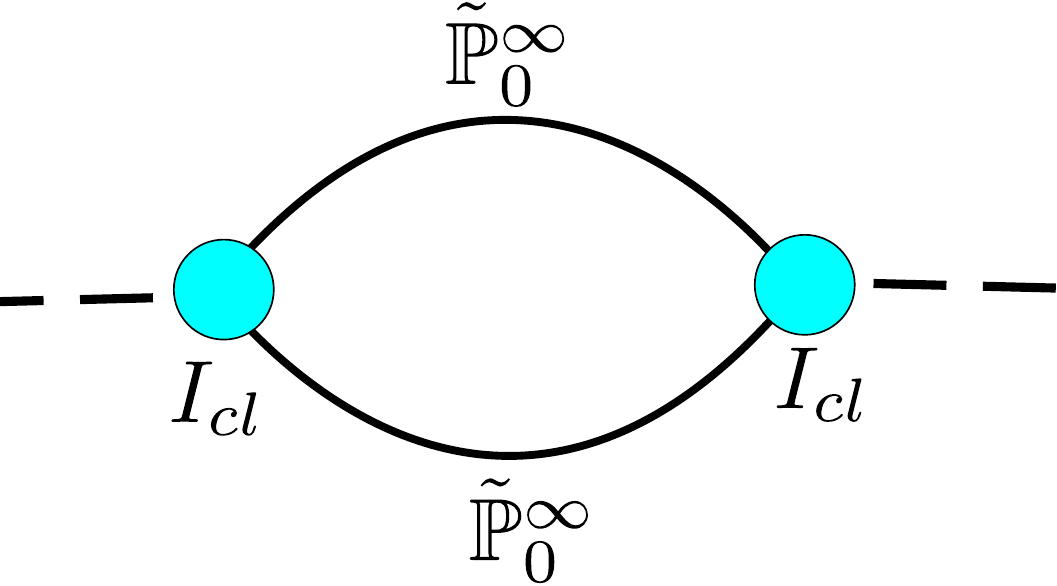}
       $$
  \item $b_1(M)=2$
       $$
       \figbox{0.23}{3-form-input}\hspace{8mm} \figbox{0.23}{3-form-input}
       $$
       $$
       \figbox{0.23}{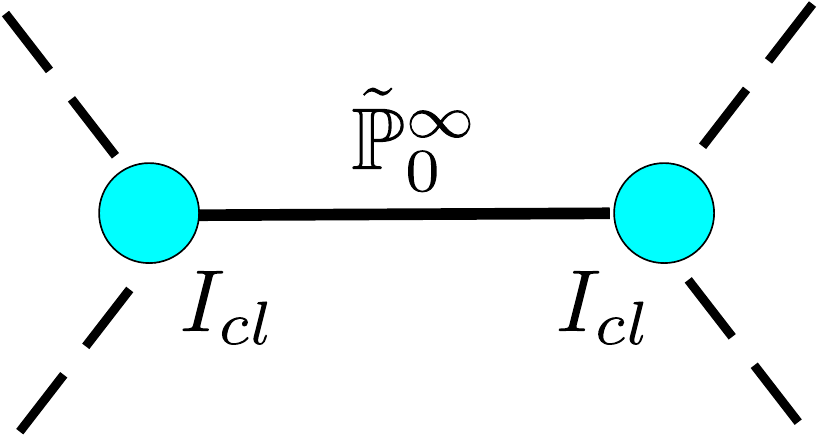}
       $$
   \item $b_1(M)=3$
       $$
       \figbox{0.23}{3-form-input}\hspace{8mm} \figbox{0.23}{3-form-input}
       $$
       $$
       \figbox{0.23}{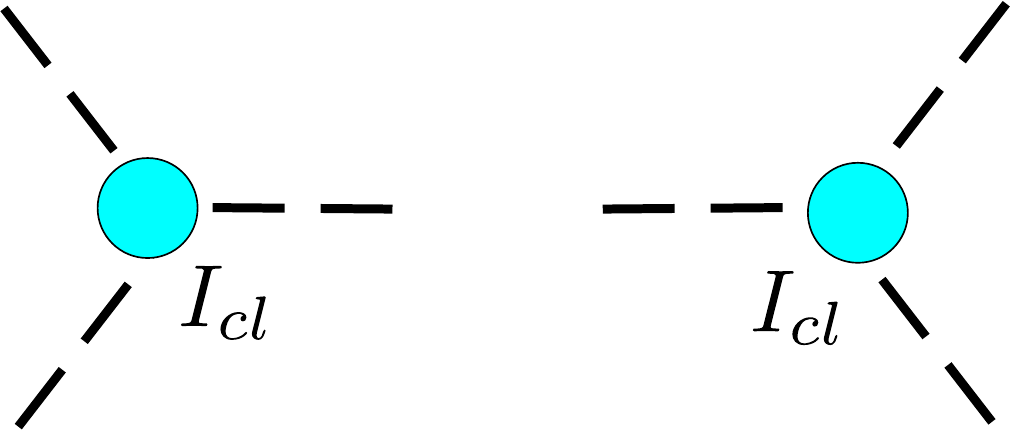}
       $$
\end{enumerate}
\end{example}
\begin{rmk}
We have omitted all the inputs that are $0$-forms on $M$ in the above pictures.
\end{rmk}

\appendix
\section{}\label{appendix:proof-regularization}
\subsubsection*{Proof of Lemma \ref{lemma:leading-term-K-t}}:
It is obvious that
\begin{align*}
&(\OO_x)_{[-2]}\left(6(4\pi)^{-\frac{3}{2}}e^{-\frac{||u||_z^2}{4t}}t^{-\frac{1}{2}}\Omega_l^ig^{lj}(z)d_{vert}u^k\right)\\
=&(\OO_x)_{[-2]}\left(6(4\pi)^{-\frac{3}{2}}e^{-\frac{||u||_z^2}{4t}}t^{-\frac{1}{2}}\Omega_l^ig^{lj}(z)(du^k+\Gamma_{mn}^kdz^mu^n)\right)\\
=&\left(\frac{\partial}{\partial t}-\frac{1}{4}g^{pq}(z)\left(\frac{\partial}{\partial u^p}-\Gamma_{pr}^\alpha(z)dz^ri\left(\frac{\partial}{\partial u^\alpha}\right)\right)\left(\frac{\partial}{\partial u^q}-\Gamma_{qs}^t(z)dz^si\left(\frac{\partial}{\partial u^t}\right)\right)\right)\\
 &\left(6(4\pi)^{-\frac{3}{2}}e^{-\frac{||u||_z^2}{4t}}t^{-\frac{1}{2}}\Omega_l^ig^{lj}(z)(du^k+\Gamma_{mn}^kdz^mu^n)\right)\\
=&\left(\frac{\partial}{\partial t}-\frac{1}{4}g^{pq}(z)\left(\frac{\partial}{\partial u^p}\frac{\partial}{\partial u^q}-2\Gamma_{pr}^\alpha(z)dz^r\frac{\partial}{\partial u^q}i\left(\frac{\partial}{\partial u^\alpha}\right)\right)\right)\\
 &\left(6(4\pi)^{-\frac{3}{2}}e^{-\frac{||u||_z^2}{4t}}t^{-\frac{1}{2}}\Omega_l^ig^{lj}(z)(du^k+\Gamma_{mn}^kdz^mu^n)\right)\\
\overset{(1)}{=}&6(4\pi)^{-\frac{3}{2}}e^{-\frac{||u||_z^2}{4t}}t^{-\frac{3}{2}}\Omega_l^ig^{lj}(z)d_{vert}u^k.
\end{align*}
Some details for the identity $(1)$:
\begin{equation}\label{eqn: identity-(1)-1}
\begin{aligned}
  &\frac{\partial}{\partial t}\left(6(4\pi)^{-\frac{3}{2}}e^{-\frac{||u||_z^2}{4t}}t^{-\frac{1}{2}}\Omega_l^ig^{lj}d_{vert}u^k\right)\\
=&(4\pi)^{-\frac{3}{2}}\left(-\frac{6}{2}e^{-\frac{||u||_z^2}{4t}}t^{-\frac{3}{2}}+\frac{||u||_z^2}{4t^2}6t^{-\frac{1}{2}}\right)\Omega_l^ig^{lj}d_{
vert}u^k,
\end{aligned}
\end{equation}
and
\begin{equation}\label{eqn: identity-(1)-2}
\begin{aligned}
 &\frac{1}{2}g^{pq}\Gamma_{pr}^\alpha dz^r\frac{\partial}{\partial u^q}\iota\left(\frac{\partial}{\partial
u^\alpha}\right)\left(6(4\pi)^{-\frac{3}{2}}e^{-\frac{||u||_z^2}{4t}}t^{-\frac{1}{2}}\Omega_l^ig^{lj}(du^k+\Gamma_{mn}^kdz^mu^n)\right)\\
=&(4\pi)^{-\frac{3}{2}}\frac{6}{2}\cdot
g^{pq}\Gamma_{pr}^kdz^r\left(-\frac{2g_{q\alpha}u^\alpha}{t}\right)e^{-\frac{||u||_z^2}{4t}}t^{-\frac{1}{2}}\Omega_l^ig^{lj}\\
=&(-6)(4\pi)^{-\frac{3}{2}}\Gamma_{pr}^ku^pe^{-\frac{||u||_z^2}{4t}}t^{-\frac{3}{2}}\Omega_l^ig^{lj}dz^r,
\end{aligned}
\end{equation}

\begin{equation}\label{eqn: identity-(1)-3}
\begin{aligned}
&-\frac{1}{4}g^{pq}\frac{\partial}{\partial u^p}\frac{\partial}{\partial
u^q}\left(6(4\pi)^{-\frac{3}{2}}e^{-\frac{||u||_z^2}{4t}}t^{-\frac{1}{2}}\Omega_l^ig^{lj}(du^k+\Gamma_{mn}^kdz^mu^n)\right)\\
=&-\frac{(4\pi)^{-\frac{3}{2}}}{4}g^{pq}\frac{\partial}{\partial
u^p}\left\{ 6\left(-\frac{2g_{q\alpha}u^{\alpha}}{t}\right)e^{-\frac{||u||_z^2}{4t}}t^{-\frac{1}{2}}\Omega_l^ig^{lj}(du^k+\Gamma_{
mn}^kdz^mu^n)+ 6e^{-\frac{||u||_z^2}{4t}}t^{-\frac{1}{2}}\Omega_l^ig^{lj}\Gamma_{mq}^kdz^m \right\}\\
=& \frac{6}{2}(4\pi)^{-\frac{3}{2}}g^{pq}g_{qp}e^{-\frac{||u||_z^2}{4t}}t^{-\frac{3}{2}}\Omega_l^ig^{lj}d_{vert}u^k+\frac{6}{2t}(4\pi)^{-\frac{3}{2}}g^{pq}g_{q\alpha}
u^\alpha\left(-\frac{2g_{p\beta}u^\beta}{t}\right)e^{-\frac{||u||_z^2}{4t}}t^{-\frac{1}{2}}\Omega_l^ig^{lj}d_{vert}u^k\\
&\qquad + \frac{6}{2t}(4\pi)^{-\frac{3}{2}}g^{pq}g_{q\alpha}u^\alpha e^{-\frac{||u||_z^2}{4t}}t^{-\frac{1}{2}}\Omega_l^ig^{lj}\Gamma_{mp}^kdz^m-(4\pi)^{-\frac{3}{2}}g^{pq}\cdot
\frac{3}{2}\left(-\frac{2g_{p\alpha}u^\alpha}{t}\right)e^{-\frac{||u||_z^2}{4t}}t^{-\frac{1}{2}}\Omega_l^ig^{lj}\Gamma_{mq}^kdz^m\\
=&\frac{18}{2}(4\pi)^{-\frac{3}{2}}e^{-\frac{||u||_z^2}{4t}}t^{-\frac{3}{2}}\Omega_l^ig^{lj}d_{vert}u^k-\frac{3}{2t^{\frac{5}{2}}}(4\pi)^{-\frac{3}{2}}||u||_z^2e^{-\frac{||u||_z^2 } { 4t }}\Omega_l^ig^{lj}d_{vert}u^k+\frac{6}{t^{\frac{3}{2}}}(4\pi)^{-\frac{3}{2}}u^pe^{-\frac{||u||_z^2}{4t}}\Omega_l^ig^{lj}\Gamma_{mp}^kdz^m.
\end{aligned}
\end{equation}
The identity $(1)$ follows from by adding \eqref{eqn: identity-(1)-1}, \eqref{eqn: identity-(1)-2} and \eqref{eqn:
identity-(1)-3}.

On the other hand, it is clear that
\begin{align*}
&g^{lj}(z)\Omega^i_li\left(\frac{\partial}{\partial u^i}\right)i\left(\frac{\partial}{\partial
u^j}\right)\left((4\pi)^{-\frac{3}{2}}e^{-\frac{||u||_z^2}{4t}}\det(g(z))\epsilon_{ijk}t^{-\frac{3}{2}}d_{vert}u^id_{vert}u^jd_{vert}
u^k\right)\\
=&-6(4\pi)^{-\frac{3}{2}}e^{-\frac{||u||_z^2}{4t}}t^{-\frac{3}{2}}\Omega_l^ig^{lj}(z)d_{vert}u^k
\end{align*}
Thus it remains to show the following:
\begin{equation}\label{eqn:two-derivatives}
\begin{aligned}
&\left(\frac{\partial}{\partial t}-\frac{1}{4}g^{ij}(z)\left(\frac{\partial}{\partial
u^i}-\Gamma_{ik}^l(z)dz^ki(\frac{\partial}{\partial u^l})\right)\left(\frac{\partial}{\partial
u^j}-\Gamma_{jm}^n(z)dz^mi(\frac{\partial}{\partial u^n})\right)\right)\\
&\left((4\pi)^{-\frac{3}{2}}e^{-\frac{||u||_z^2}{4t}}\det(g_{mn}(z))\epsilon_{ijk}t^{-\frac{3}{2}}d_{vert}u^id_{vert}u^jd_{vert}u^k\right)\\
&=0.
\end{aligned}
\end{equation}
We do this in detail. Note that if both derivatives $\frac{\partial}{\partial u^i}$ and $\frac{\partial}{\partial u^j}$ in equation \eqref{eqn:two-derivatives} act on $e^{-\frac{||u||_z^2}{4t}}$, the output will cancel with that of the operator $\frac{\partial}{\partial t}$ (the same as the computation in the flat case). Thus it is enough to show that the operator
$$
-g^{ij}(z)\left(\frac{\partial}{\partial u^i}-\Gamma_{ik}^l(z)dz^ki(\frac{\partial}{\partial u^l})\right)\left(\frac{\partial}{\partial u^j}-\Gamma_{jm}^n(z)dz^mi(\frac{\partial}{\partial u^n})\right)
$$
annihilates $\left(e^{-\frac{||u||_z^2}{4t}}\det(g_{mn}(z))\epsilon_{ijk}t^{-\frac{3}{2}}d_{vert}u^id_{vert}u^jd_{vert}u^k\right)$ if at least one derivative $\frac{\partial}{\partial u^i}$ does not act on $e^{-\frac{||u||_z^2}{4t}}$. First we have:
\begin{equation}\label{eqn:idenity-(2)-1}
 \begin{aligned}
  &-g^{pq}\Gamma_{pr}^l\Gamma_{qm}^ndz^rdz^m i(\frac{\partial}{\partial u^l})i(\frac{\partial}{\partial
u^n})\left((4\pi)^{-\frac{3}{2}}e^{-\frac{||u||_z^2}{4t}}\det(g)\epsilon_{ijk}t^{-\frac{3}{2}}d_{vert}u^id_{vert}u^jd_{vert}u^k\right)\\
=&-g^{pq}\Gamma_{pr}^i\Gamma_{qm}^jdz^rdz^m(4\pi)^{-\frac{3}{2}}e^{-\frac{||u||_z^2}{4t}}\det(g)\epsilon_{ijk}t^{-\frac{3}{2}}d_{vert}u^k,
 \end{aligned}
\end{equation}

\begin{equation}\label{eqn:idenity-(2)-2}
 \begin{aligned}
&\left(2g^{pq}\frac{\partial}{\partial u^p}\Gamma_{qm}^ndz^m i(\frac{\partial}{\partial u^n})\right)\left((4\pi)^{-\frac{3}{2}}e^{-\frac{||u||_z^2}{4t}}\det(g(z))\epsilon_{ijk}t^{-\frac{3}{2}}d_{vert}u^id_{vert}u^jd_{vert}u^k\right)\\
=&\left(2g^{pq}\frac{\partial}{\partial u^p}\Gamma_{qm}^idz^m\right)\left((4\pi)^{-\frac{3}{2}}e^{-\frac{||u||_z^2}{4t}}\det(g(z))\epsilon_{ijk}t^{-\frac{3}{2}}d_{vert}u^jd_{vert}u^k\right)\\
=&2g^{pq}\Gamma_{qm}^idz^m\frac{\partial}{\partial u^p}\left((4\pi)^{-\frac{3}{2}}e^{-\frac{||u||_z^2}{4t}}\det(g(z))\epsilon_{ijk}t^{-\frac{3}{2}}(du^j+\Gamma_{\alpha\beta}^jdz^\alpha u^\beta)d_{vert}u^k\right)\\
=&2g^{pq}\Gamma_{qm}^idz^m\left((4\pi)^{-\frac{3}{2}}e^{-\frac{||u||_z^2}{4t}}\det(g(z))\epsilon_{ijk}t^{-\frac{3}{2}}
\Gamma_{\alpha p}^jdz^\alpha d_{vert}u^k\right)\\
&\qquad - 4\Gamma_{qm}^idz^mu^q (4\pi)^{-\frac{3}{2}}e^{-\frac{||u||_z^2}{4t}}\det(g(z))\epsilon_{ijk}t^{-\frac{5}{2}}
 d_{vert}u^jd_{vert}u^k.
\end{aligned}
\end{equation}
And we have
\begin{equation}\label{eqn:idenity-(2)-3}
 \begin{aligned}
 &(-g^{pq}\frac{\partial}{\partial u^p}\frac{\partial}{\partial
u^q})\left((4\pi)^{-\frac{3}{2}}e^{-\frac{||u||_z^2}{4t}}\det(g(z))\epsilon_{ijk}t^{-\frac{3}{2}}d_{vert}u^id_{vert}u^jd_{vert}u^k\right)\\
&\qquad - \det(g(z))\epsilon_{ijk}t^{-\frac{3}{2}}d_{vert}u^id_{vert}u^jd_{vert}u^k(-g^{pq}\frac{\partial}{\partial
u^p}\frac{\partial}{\partial u^q})((4\pi)^{-\frac{3}{2}}e^{-\frac{||u||_z^2}{4t}})\\
=&(4\pi)^{-\frac{3}{2}}e^{-\frac{||u||_z^2}{4t}}\det(g(z))\epsilon_{ijk}t^{-\frac{3}{2}}(-g^{pq}\frac{\partial}{\partial u^p}\frac{\partial}{\partial
u^q})(du^i+\Gamma_{\alpha\beta}^idz^\alpha u^\beta)(du^j+\Gamma_{\gamma\epsilon}^jdz^\gamma u^\epsilon)d_{vert}u^k\\
&\qquad - 2g^{pq}\frac{\partial}{\partial u^q}\left((4\pi)^{-\frac{3}{2}}e^{-\frac{||u||_z^2}{4t}}\right)\frac{\partial}{\partial
u^p}\left(\det(g(z))\epsilon_{ijk}t^{-\frac{3}{2}}(du^i+\Gamma_{\alpha\beta}^idz^\alpha u^\beta)d_{vert}u^jd_{vert}u^k\right)\\
=&-(4\pi)^{-\frac{3}{2}}g^{pq}e^{-\frac{||u||_z^2}{4t}}\det(g(z))\epsilon_{ijk}t^{-\frac{3}{2}}\Gamma_{\alpha p}^idz^\alpha\Gamma_{\gamma q}^jdz^\gamma d_{vert}u^k\\
&\qquad + (4\pi)^{-\frac{3}{2}}\frac{4u^p}{t^{\frac{5}{2}}}e^{-\frac{||u||_z^2}{4t}}\det(g(z))\epsilon_{ijk}\Gamma_{\alpha p}^idz^\alpha d_{vert}u^jd_{vert}u^k.
 \end{aligned}
\end{equation}
Adding equations \eqref{eqn:idenity-(2)-1}, \eqref{eqn:idenity-(2)-2} and \eqref{eqn:idenity-(2)-3}, we get the desired
cancellation.

It is easy to check that the leading term of $d_x^*$ with respect to the total degree is given by
$$
(d_x^*)_{[-2]}=-\frac{1}{4}g^{ij}(z)i\left(\frac{\partial}{\partial u^i}\right)\left(\frac{\partial}{\partial
u^j}-(\Gamma_{kj}^ldz^k)i(\frac{\partial}{\partial u^l})\right).
$$
Thus the leading singularity of $d_x^*K_t$ is given by $(d_x^*K_t)_{[-2]}=(d_x^*)_{[-2]}(K_t)_{[0]}$. First notice that we have
\begin{align*}
&-\frac{1}{4}g^{mn}(z)i\left(\frac{\partial}{\partial u^m}\right)\left(\frac{\partial}{\partial
u^n}-(\Gamma_{kn}^pdz^k)i(\frac{\partial}{\partial
u^p})\right)\left((4\pi)^{-\frac{3}{2}}e^{-\frac{||u||_z^2}{4t}}\det(g(z))\epsilon_{ijk}\left(t^{
-\frac{1}{2}} \Omega_l^ig^{lj}(z)d_{vert}u^k\right)\right)\\
=&-\frac{1}{4}g^{mn}(z)i\left(\frac{\partial}{\partial u^m}\right)\left(\frac{\partial}{\partial
u^n}\right)\left((4\pi)^{-\frac{3}{2}}e^{-\frac{||u||_z^2}{4t}}\det(g(z))\epsilon_{ijk}\left(t^{
-\frac{1}{2}} \Omega_l^ig^{lj}(z)d_{vert}u^k\right)\right)\\
=&\ \frac{1}{4}g^{mn}\cdot\frac{2g_{pn}u^p}{t}\delta_m^k\left((4\pi)^{-\frac{3}{2}}e^{-\frac{||u||_z^2}{4t}}\det(g(z))\epsilon_{ijk}t^{
-\frac{1}{2}} \Omega_l^ig^{lj}(z)\right)\\
=&\ \frac{u^k}{2t^{3/2}}(4\pi)^{-\frac{3}{2}}e^{-\frac{||u||_z^2}{4t}}\det(g(z))\epsilon_{ijk} \Omega_l^ig^{lj}(z).
\end{align*}
On the other hand, there is
\begin{align*}
&-\frac{1}{4}g^{mn}(z)i\left(\frac{\partial}{\partial u^m}\right)\left(\frac{\partial}{\partial
u^n}-(\Gamma_{pn}^ldz^p)i(\frac{\partial}{\partial u^l})\right)
\left((4\pi)^{-\frac{3}{2}}e^{-\frac{||u||_z^2}{4t}}\det(g(z))\epsilon_{ijk}t^{-\frac{3}{2}}d_{vert}u^id_{vert}u^jd_{vert}
u^k\right)\\
=&-\frac{1}{4}g^{mn}(z)i\left(\frac{\partial}{\partial u^m}\right)\left(\frac{\partial}{\partial
u^n}\right)\left((4\pi)^{-\frac{3}{2}}e^{-\frac{||u||_z^2}{4t}}\det(g(z))\epsilon_{ijk}t^{-\frac{3}{2}}d_{vert}u^id_{vert}u^jd_{vert}
u^k\right)\\
&\qquad + \frac{1}{4}g^{mn}(z)\Gamma_{pn}^ldz^pi\left(\frac{\partial}{\partial u^m}\right)i\left(\frac{\partial}{\partial
u^l}\right)\left((4\pi)^{-\frac{3}{2}}e^{-\frac{||u||_z^2}{4t}}\det(g(z))\epsilon_{ijk}t^{-\frac{3}{2}}d_{vert}u^id_{vert}u^jd_{vert}
u^k\right)\\
=&-\frac{1}{4}g^{mn}(z)\frac{\partial}{\partial u^n}\left((4\pi)^{-\frac{3}{2}}e^{-\frac{||u||_z^2}{4t}}\right)i\left(\frac{\partial}{\partial u^m}\right)\left(\det(g(z))\epsilon_{ijk}t^{-\frac{3}{2}}d_{vert}u^id_{vert}u^jd_{vert}u^k\right)\\
&\qquad - \frac{1}{4}g^{mn}(z)(4\pi)^{-\frac{3}{2}}e^{-\frac{||u||_z^2}{4t}}i\left(\frac{\partial}{\partial
u^m}\right)\frac{\partial}{\partial u^n}\left(\det(g(z))\epsilon_{ijk}t^{-\frac{3}{2}}d_{vert}u^id_{vert}u^jd_{vert}u^k\right)\\
&\qquad + \frac{1}{4}g^{in}(z)\Gamma_{pn}^jdz^p(4\pi)^{-\frac{3}{2}}e^{-\frac{||u||_z^2}{4t}}\det(g(z))\epsilon_{ijk}t^{-\frac{3}{2}}d_{vert}u^k\\
=&-\frac{1}{4}g^{mn}(z)\frac{\partial}{\partial u^n}\left((4\pi)^{-\frac{3}{2}}e^{-\frac{||u||_z^2}{4t}}\right)i\left(\frac{\partial}{\partial
u^m}\right)\left(\det(g(z))\epsilon_{ijk}t^{-\frac{3}{2}}d_{vert}u^id_{vert}u^jd_{vert}u^k\right)\\
&\qquad - \frac{1}{4}g^{jn}(z)\Gamma_{\alpha n}^idz^\alpha
(4\pi)^{-\frac{3}{2}}e^{-\frac{||u||_z^2}{4t}}\det(g(z))\epsilon_{ijk}t^{-\frac{3}{2}}d_{vert}u^k\\
&\qquad + \frac{1}{4}g^{in}(z)\Gamma_{pn}^jdz^p(4\pi)^{-\frac{3}{2}}e^{-\frac{||u||_z^2}{4t}}\det(g(z))\epsilon_{ijk}t^{-\frac{3}{2}}d_{vert}u^k\\
=&-\frac{1}{4}g^{mn}(z)\frac{\partial}{\partial u^n}\left((4\pi)^{-\frac{3}{2}}e^{-\frac{||u||_z^2}{4t}}\right)i\left(\frac{\partial}{\partial
u^m}\right)\left(\det(g(z))\epsilon_{ijk}t^{-\frac{3}{2}}d_{vert}u^id_{vert}u^jd_{vert}u^k\right)\\
=&-\frac{1}{4}g^{mn}(z)\left(-\frac{2}{t}\right)(4\pi)^{-\frac{3}{2}}e^{-\frac{||u||_z^2}{4t}}g_{ln}(z)u^l\delta_m^i\left(\det(g(z))\epsilon_{ijk}t^{-\frac{3}{2
}}d_{vert} u^jd_{vert}u^k\right)\\
=&3\cdot\frac{(4\pi)^{-\frac{3}{2}}}{2t^{5/2}}e^{-\frac{||u||_z^2}{4t}}\det(g(z))\epsilon_{ijk}u^id_{vert}u^jd_{vert}u^k,
\end{align*}
and the last statement of the lemma follows.

\end{document}